\documentclass[a4paper,10pt]{amsart}
\usepackage{amsmath,amsfonts,amssymb, amsthm, xypic}
\usepackage[all]{xy}
\usepackage[left=2cm, right=2cm, top=2cm]{geometry}
 
\def\dim{\mathrm{dim}}
\def\N{\mathbb{N}}
\newcommand\Hom{\mathrm{Hom}}
\def\Z{\mathbb{Z}}

\def\a{\alpha}

\def\qlb{\overline{\mathbb{Q}_l}}
\def\C{\mathbb{C}}
\def\uz{\underline{z}}
\def\A{\mathcal{A}}
\def\X{\mathcal{X}}
\def\Y{\mathcal{Y}}
\def\cZ{\mathcal{Z}}
\def\W{\mathcal{W}}
\def\qlb{\overline{\mathbb{Q}_l}}
\def\Exp{\text{Exp}}
\def\Ext{\text{Ext}}
\def\Hom{\text{Hom}}
\def\Ker{\text{Ker}}
\def\dim{\text{dim}}

\def\End{\text{End}}
\def\quot{/ \hspace{-.03in} /}
\def\Lc{\mathcal{L}}

\newtheorem{theo}{\bf{Theorem}}[section]
\newtheorem{lem}[theo]{Lemma}
\newtheorem{cor}[theo]{Corollary}
\newtheorem{prop}[theo]{Proposition}

\newtheorem{conj}[theo]{Conjecture}

\numberwithin{equation}{section}

\title[Kac polynomials for curves]{Indecomposable vector bundles and stable Higgs bundles over smooth projective curves}
\author{Olivier Schiffmann\;\dag}
\thanks{\dag\; partially supported by ANR grant 13-BS01-0001-01}
\date{}

\begin{document}

\begin{abstract}
We prove that the number of indecomposable vector bundles of fixed rank $r$ and degree $d$ over a smooth projective curve $X$ defined over a finite field is given by a polynomial (depending only on $r,d$ and the genus $g$ of $X$) in the Weil numbers of $X$. We provide a closed formula -expressed in terms of generating series- for this polynomial.
We also show that the same polynomial computes the number of points of the moduli space of stable Higgs bundles of
rank $r$ and degree $d$ over $X$ (at least for large characteristics). This entails a closed formula for the Poincarr\'e polynomial of the moduli spaces of stable Higgs bundles over a compact Riemann surface, and hence also for the Poincarr\'e polynomials of the character varieties for the groups $GL(r)$.
\end{abstract}

\maketitle

\vspace{.1in}

\section{Introduction}

\vspace{.2in}

\paragraph{\textbf{1.1}}Let $g \geq 0$ and let $X$ be a smooth projective geometrically connected curve of genus $g$ defined over a finite field
$\mathbb{F}_q$. Let $l$ be a prime number not dividing $q$ and let $\sigma_1, \ldots, \sigma_{2g}$ stand for the associated Weil numbers of $X$ (i.e the eigenvalues of the Frobenius acting on $H^1(X \otimes_{\mathbb{F}_q} \overline{\mathbb{F}_q}, \qlb)$ ), so that the zeta function of
$X$ is 
$$\zeta_X(z)=\frac{\prod_{i=1}^{2g}(1-\sigma_iz)}{(1-z)(1-qz)}.$$
Fixing an embedding $\iota : \overline{\mathbb{Q}_l} \to \mathbb{C}$ we may view the $\sigma_i$ as complex numbers satisfying $\overline{\sigma_{2i-1}}= \sigma_{2i}$ and $\sigma_{2i-1} \sigma_{2i}=q$ for $i=1, \ldots, g$.
Consider the torus
$$T_g=\{(\a_1, \ldots, \a_{2g}) \in \mathbb{G}_m^{2g}\;|\; \a_{2i-1} \a_{2i}=\a_{2j-1}\a_{2j}\;\forall\; i, j\}.$$
The group
$$W_g =\mathfrak{S}_g \ltimes (\mathfrak{S}_2)^g$$
naturally acts on $T_g$ and the 
collection $\{\sigma_1, \ldots, \sigma_{2g}\}$ defines a canonical element $\sigma_X$ in the quotient $T_g(\mathbb{C})/W_g$. We denote by the same letter $q$ the size of the finite field $\mathbb{F}_q$ and the element $q=\a_{2i-1}\a_{2i} \in \mathbb{Q}[T_g]^{W_g}$, hoping that this will not create any confusion. Let $K_g$ be the localization of $\mathbb{Q}[T_g]^{W_g}$ at the multiplicative set generated by $\{q^l-1\;|\; l \geq 1\}$.

\vspace{.2in}

\paragraph{\textbf{1.2.}}For $r >0$ and $d \in \mathbb{Z}$ let $\A_{r,d}(X)$ stand for the number of geometrically indecomposable
vector bundles on $X$ (i.e vector bundles $\mathcal{V}$ over $X$ such that $\mathcal{V} \otimes_{\mathbb{F}_q}
\overline{\mathbb{F}_q}$ is indecomposable ) of rank $r$ and degree $d$. The finiteness of $\A_{r,d}(X)$ results from standard arguments based on the Harder-Narasimhan filtration, see e.g. Section~\textbf{2.1.}

\vspace{.1in}

The first main result of this paper is the following~:

\vspace{.1in}

\begin{theo}\label{T:1} For any fixed genus $g$ and any pair $(r,d)\in \mathbb{N} \times \mathbb{Z}$ there exists a unique element $A_{g,r,d} \in K_g$
such that for any smooth projective geometrically connected curve $X$ of genus $g$ defined over a finite field we have
$$\A_{r,d}(X)=A_{g,r,d}(\sigma_X).$$
\end{theo}

\vspace{.1in}

Conjecturally, $A_{g,r,d}$ belongs to $\mathbb{Q}[T_g]^{W_g}$ and there is no need to consider the localization $K_g$ (see Conjecture~\ref{C:conj1} and Corollary~\ref{Cor:2} below).

\vspace{.1in}
 
\noindent
\textit{Remark.} The existence of polynomials counting indecomposable representations of quivers over finite fields
is a well-known result of Kac, \cite{Kac}. The above theorem may be viewed as a global analog of Kac's theorem. 

\vspace{.1in}

\paragraph{\textbf{1.3.}} The proof of Theorem~\ref{T:1} is effective, i.e. we can explicitly compute the
polynomial $A_{g,r,d}$. In order to write down the (rather involved) formula for $A_{g,r,d}$, we first need to introduce a few notations.

\vspace{.1in}

\noindent
\textit{Notations for partitions.} Let $\lambda=(\lambda_1 \geq \lambda_2 \geq \cdots \geq \lambda_l>0)$ be a partition. The Young diagram associated to $\lambda$ is the set of boxes with integer coordinates $(i,j)$ with $1 \leq i \leq \lambda_j$. For a box $s \in \lambda$ we denote by $a(s)$ (\textit{armlength}), resp. $l(s)$ (\textit{leglength}) the number of boxes in $\lambda$ lying strictly to
the right (resp. strictly above) $s$. Here is an example for the partition $(10,9^3,6,3^2)$~:

\centerline{
\begin{picture}(200,120)
\put(0,20){\line(0,1){70}}
\put(10,90){\line(0,-1){70}}
\put(0,90){\line(1,0){30}}
\put(20,90){\line(0,-1){70}}
\put(30,90){\line(0,-1){70}} \put(0,80){\line(1,0){30}}
\put(40,70){\line(0,-1){50}} \put(0,70){\line(1,0){60}}
\put(50,70){\line(0,-1){50}} \put(60,70){\line(0,-1){50}}
\put(0,60){\line(1,0){90}} \put(70,60){\line(0,-1){40}}
\put(80,60){\line(0,-1){40}} \put(90,60){\line(0,-1){40}}
\put(0,50){\line(1,0){90}} \put(0,40){\line(1,0){90}}
\put(0,30){\line(1,0){100}}
\put(0,20){\line(1,0){100}} \put(100,30){\line(0,-1){10}}
\put(33,43){{$s$}} \put(33,72){{$l$}}
\put(94,40){{$a$}} 
\put(180,34){{$a(s)=5$}}
\put(182,21){{$l(s)=2$}} \put(35,52){\line(0,1){16}}
\put(35,68){\line(1,-1){6}} \put(42,45){\line(1,0){46}}
\end{picture}}
\vspace{.05in}
\centerline{{Figure 1.} Notations for partitions.}

\vspace{.2in}

If $\lambda, \mu$ are partitions then we set $\langle \lambda, \mu \rangle=\sum_i \lambda'_i \mu'_i$ where $\lambda', \mu'$ are the conjugate partitions of $\lambda,\mu$.

\vspace{.1in}

\noindent
\textit{Notations for plethystic operators.} Consider the space $K_g[[z,T]]$ of power series
in the variables $z,T$. For $l \geq 1$ we define the $l$th Adams operator $\psi_l$ as the $\mathbb{Q}$-algebra map
$$\psi_l~: K_g[[z,T]] \to K_g[[z,T]], \qquad \a_i \mapsto \a_i^l, \;z \mapsto z^l, \;T\mapsto T^l.$$
Set $K_g[[z,T]]^+=zK_g[[z,T]] + T K_g[[z,T]]$.
The plethystic exponential and logarithm functions are inverse maps
$$\Exp~:K_g[[z,T]]^+ \longrightarrow 1+ K_g[[z,T]]^+, \qquad \text{Log}~:1+K_g[[z,T]]^+ \longrightarrow K_g[[z,T]]^+$$
respectively defined by
 $$\Exp(f)=exp \left(\sum_k \frac{1}{k}
\psi_k(f)\right), \qquad \text{Log}(f)=\sum_{k\geq 1} \frac{\mu(k)}{k} \psi_k \left(log(f)\right).$$
These operators satisfy the usual properties, i.e. $\Exp(f+g)=\Exp(f)\Exp(g)$ and $\text{Log}(fg)=\text{Log}(f)+\text{Log}(g)$.
We refer to e.g. \cite[Sec. 2]{Mozgovoy_Fermionic} for more on these plethystic operators.

\vspace{.1in}

Observe that $\Exp(z)=1/(1-z)$. We set
$$\zeta(z)=\Exp\left((1-\sum_i \a_i + q)z\right)=\frac{\prod_i (1-\a_iz)}{(1-z)(1-qz)}, \qquad \widetilde{\zeta}(z)=z^{1-g}\zeta(z)$$
 
\vspace{.1in}

\noindent
\textit{Remark.} For any smooth projective curve $X$ as above, we have $\zeta_X(z)=\zeta(z)(\sigma_X)$,
which may loosely be interpreted as an equality $\Exp(|X(\mathbb{F}_q)|z)=\zeta_X(z)$.

\vspace{.1in}

We are now ready to introduce the ingredients entering in the explicit computation of $A_{g,r,d}$.  Let $\lambda=(1^{r_1}2^{r_2}\ldots t^{r_t})$ be a partition. Let us set 
$$J_{\lambda}(z)=\prod_{s \in \lambda} \zeta^*_X(q^{-1-l(s)}z^{a(s)})$$
where
$$\zeta^*_X(q^{-u}z^v)=\begin{cases}\zeta_X(q^{-u}z^v)\qquad & \text{if}\; (u,v) \neq (1,0) \\ \prod_i(1-\a_i^{-1})/(1-q^{-1}) \qquad & \text{if}\; (u,v)=(1,0) \end{cases}.$$ 
Next, write $n=l(\lambda)=\sum_i r_i$, and 
$$r_{<i}=\sum_{k<i} r_k, \qquad r_{>i}=\sum_{k>i}r_k, \qquad r_{[i,j]}=\sum_{k=i}^j r_k$$
and consider the multi-variable rational function
$$L(z_n, \ldots, z_1)=\frac{1}{\prod_{i<j} \widetilde{\zeta}\big(\frac{z_i}{z_j}\big)} \sum_{\sigma \in \mathfrak{S}_n} \sigma \left\{ \prod_{i<j}
\widetilde{\zeta}\big(\frac{z_i}{z_j}\big) \cdot \frac{1}{\prod_{i<n} \big( 1-q\frac{z_{i+1}}{z_i}\big)} \cdot 
\frac{1}{1-z_1}\right\}.$$
Denote by $\text{Res}_{\lambda}$ the operator of taking the iterated residue along
\begin{align*}
&\frac{z_n}{z_{n-1}}=\frac{z_{n-1}}{z_{n-2}}= \cdots = \frac{z_{2+r_{<t}}}{z_{1+r_{<t}}}=q^{-1}\\
&\vdots \qquad \qquad \vdots \qquad  \qquad\qquad \vdots\\
&\frac{z_{r_1}}{z_{r_1-1}}= \frac{z_{r_1-1}}{z_{r_1-2}}= \cdots = \frac{z_{2}}{z_{1}}=q^{-1}.
\end{align*}
Put
$$\widetilde{H}_{\lambda}(z_{1+ r_{<t}}, \ldots, z_{1+r_{<i}}, \ldots, z_1)=\text{Res}_{\lambda}\left[ L(z_n, \ldots, z_1) \prod_{\substack{j =1 \\ j \not\in \{r_{\leq i}\}}}^{n}\frac{dz_j}{z_j} \right]$$
and finally
$$H_{\lambda}(z)=\widetilde{H}_{\lambda}(z^tq^{-r_{<t}}, \ldots, z^iq^{-r_{<i}}, \ldots, z).$$
Note that if $r_i=0$ for some $i$ then the function $\widetilde{H}_{\lambda}$ is independent of its $i$th argument.

\begin{theo}\label{T:2} Define rational functions $A_{g,r}(z) \in K_g(z)$ by the relation
\begin{equation}\label{E:mainT1}
\sum_{r \geq 1} A_{g,r}(z)T^r=(q-1)Log\left( \sum_{\lambda} q^{(g-1)\langle \lambda,\lambda\rangle} J_{\lambda}(z) H_{\lambda}(z)T^{|\lambda|}\right).
\end{equation}
Then for any $d \in \Z$ we have
$$
A_{g,r,d}=-\sum_{\xi \in \mu_r} \xi^{-d} \text{Res}_{z=\xi} \left( A_{g,r}(z)\frac{dz}{z}\right) 
$$
where $\mu_r$ stands for the set of $r$th roots of unity.
\end{theo}

\vspace{.1in}

\paragraph{\textit{Examples.}} We list below the polynomials $A_{g,r,d}$ for $r \leq 2$~:
$$A_{g,1,d}=\prod_{i=1}^{2g}(1-\a_i)$$
$$A_{g,2,d}=\prod_{i=1}^{2g}(1-\a_i) \cdot \left( \frac{\prod_{i} (1-q\a_i)}{(q-1)(q^2-1)} -\frac{\prod_i (1+\a_i)}{4(1+q)} + \frac{\prod_i (1-\a_i)}{2(q-1)} \left[ \frac{1}{2}-\frac{1}{q-1}-\sum_i \frac{1}{1-\a_i}\right]\right)$$

\vspace{.1in}

Theorems~\ref{T:1} and \ref{T:2} are proved in Sections 4 and 5 respectively. 

\vspace{.1in}

\begin{conj}\label{C:conj1} The polynomial $A_{g,r,d}$ does not depend on $d$.\end{conj}
In view of Theorem~\ref{T:2}, this conjecture may be recast in purely combinatorial terms as follows :

\vspace{.1in}

\begin{conj} The rational function $A_{g,r}(z)$ is regular at nontrivial $r$th roots of unity.
\end{conj}
It follows from the proof of Theorem~\ref{T:2} that $A_{g,r}(z)$ is regular outside of $\mu_r$ and has at most simple poles. Thus the above conjecture says that $(1-z)A_{g,r}(z)$ belongs to $K_g[z]$, in which case we would simply have $$A_{g,r,d}=\left[(1-z)A_{g,r}(z)\right]_{|z=1}, \qquad \forall d.$$

\vspace{.1in}

 As supporting evidence, we prove Conjecture~\ref{C:conj1}) by direct computation when $r$ is prime, see Appendix~C.

\vspace{.2in}

\paragraph{\textbf{1.4.}} Let us now assume that $r$ and $d$ are relatively prime. Let $Higgs^{st}_{r,d}({X})$ stand for the moduli space of stable Higgs bundles over $X$ (see Section~\textbf{6.2}). This is a (smooth, quasi-projective) variety defined over $\mathbb{F}_q$ and we may consider its set of $\mathbb{F}_q$-rational points $Higgs^{st}_{r,d}(X)(\mathbb{F}_q)$. The second main result of this paper, whose proof is very much inspired by the work of
Crawley-Boevey and Van den Bergh (see \cite{CBV}) is the following~:

\vspace{.1in}

\begin{theo}\label{T:3} Let $(r,d)$ be relatively prime. There exists an (explicit) constant $C=C(r,d)$ such that for any smooth projective geometrically connected curve $X$ of genus $g$ defined over $\mathbb{F}_{q}$ with $char(\mathbb{F}_q) >C$, we have
$$|Higgs^{st}_{r,d}(X)(\mathbb{F}_q)|=q^{1+(g-1)r^2 }\mathcal{A}_{r,d}(X).$$
\end{theo}

In the companion paper \cite{MS} written in collaboration with S. Mozgovoy, we give a different proof of the above theorem, which works in all characteristics, as well as a generalization to the case of the moduli spaces of twisted Higgs bundles.
In particular, by Theorems~\ref{T:1} and \ref{T:2}, the number of $\mathbb{F}_q$-rational points of $Higgs^{st}_{r,d}({X})$ is given by some explicit polynomial in $K_g$. 
 Theorems~\ref{T:1}, \ref{T:2} and \ref{T:3} have the following two corollaries~:

\vspace{.1in}

\begin{cor}\label{Cor:1} Let $(r,d)$ be relatively prime. Then $A_{g,r,d} \in \mathbb{N}[-z_1, \ldots, -z_{2g}]^{W_g}$.
Moreover, for any smooth, geometrically connected projective curve $X$ of genus $g$ defined over a field $\mathbb{F}_q$ of characteristic $p >C(r,d)$ the Frobenius eigenvalues in $H^n_c(Higgs^{st}_{r,d}(X \otimes \overline{\mathbb{F}_q}), \qlb)$ are all of the form
$$\lambda=\prod_j \sigma_j^{n_j}, \qquad \sum n_i=n$$ 
where $\sigma_X=(\sigma_1, \ldots, \sigma_{2g})$. Write $A_{g,r,d}=\sum_{i_1, \ldots, i_{2g}} a_{i_1, \ldots, i_{2g}} (-z_1)^{i_1} \cdots (-z_{2g})^{i_{2g}}$. The multiplicity of the eigenvalue $\sigma_1^{i_1} \cdots \sigma_{2g}^{i_{2g}}$ is equal to $a_{i_1, \ldots, i_{2g}}$. In particular,
$$\sum_n dim\; H_c^n(Higgs^{st}_{r,d}(\overline{X}), \overline{\mathbb{Q}_l}) t^n= t^{2(1+ (g-1)r^2)}A_{g,r,d}(t,t, \ldots, t).$$
\end{cor}

\vspace{.1in}

It would be a consequence of Conjecture~\ref{C:conj1} that Corollary~\ref{Cor:1} holds without the coprimality assumption, thus yielding a global analog of Kac's positivity conjecture for $A$-polynomials of quivers (the latter has recently been proved in \cite{HLRVAnnals}). 

\vspace{.1in}

\begin{cor}\label{Cor:2} Let $X_{\mathbb{C}}$ be a compact Riemann surface of genus $g$. Let $(r,d)$ be relatively prime. Then 
$$\sum_n dim\; H_c^n(Higgs^{st}_{r,d}({X}_{\mathbb{C}}), \mathbb{Q}) t^n= t^{2(1+ (g-1)r^2)}A_{g,r,d}(t,t, \ldots, t),$$
where $H^n_c$ denotes singular cohomology with compact support.
\end{cor} 

\vspace{.1in}

The Poincarr\'e polynomial of the moduli space of stable Higgs bundles on a compact Riemann surface $X_{\mathbb{C}}$ of genus $g$ has been computed in rank $2$ by Hitchin (\cite{Hitchin}), in rank $3$ by Gothen (see \cite{Gothen}) and in rank $4$ by Garcia-Prada, Heinloth and Schmitt (see \cite{Heinloth}). Hausel and Rodriguez-Villegas derived in (\cite{HRVInvent}) a conjectural formula for the mixed Hodge polynomial of the genus $g$ \textit{character variety}
for the group $GL_r$. The latter being homeomorphic $Higgs^{st}_{r,d}(X_{\mathbb{C}})$ (for any $d$)
their formula yields in particular a conjectural formula for the Poincarr\'e polynomials $H_c^n(Higgs^{st}_{r,d}(X_{\mathbb{C}}),\mathbb{Q})$.
This conjecture was extended by Mozgovoy (c.f. \cite{Mozgovoy}) to a conjectural formula for the motive of $Higgs^{st}_{r,d}({X}_{\overline{\mathbb{F}_q}})$, where $X_{\overline{\mathbb{F}_q}}$ is now a smooth projective curve of genus $g$ defined over $\overline{\mathbb{F}_q}$.
Our formula (Theorem~\ref{T:2}) bears a strong similarity to the formula conjectured by Hausel-Rodriguez-Villegas and to its extension by Mozgovoy~: it essentially differs from theirs by the presence of the rational functions $H_{\lambda}(z)$ (although of course if the main conjectures of \cite{HRVInvent} and \cite{Mozgovoy} are true then these formulas and ours compute the same numbers). Note also that Conjecture~\ref{C:conj1} is, by Theorem~\ref{T:2}, essentially equivalent to a conjecture by Hausel and Thaddeus (see \cite[Conj. 3.2]{HT}) claiming that the motive of $Higgs^{st}_{r,d}$ is independent of $d$.

\vspace{.1in}

From the fact that the moduli space $Higgs^{st}_{r,d}(X_{\mathbb{C}})$ is connected and of dimension $1+ (g-1)r^2$, we immediately deduce the following result~:

\vspace{.1in}

\begin{cor}Let $(r,d)$ be coprime. Then $A_{g,r,d}$ is unitary and of degree $2 (1 + (g-1)r^2)$.
\end{cor}

\vspace{.1in}

Theorem~\ref{T:2} has the following geometric corollary. Let $\mu~: Higgs^{st}_{r,d} \to \mathbb{A}$ be the Hitchin map (see e.g. \cite{Hitchin}) and let $\Lambda^{st}_{r,d}$ denote the zero fiber of $\mu$ (the stable global nilpotent cone). It is known that $\Lambda^{st}_{r,d}$ is a projective (in general singular) lagrangian subvariety of $Higgs_{r,d}^{st}$.

\vspace{.1in}

\begin{cor}\label{Cor:25} Assume that $gcd(r,d)=1$. The following hold~:
\begin{enumerate}
\item[i)] $(k=\mathbb{F}_q)$.  The variety $\Lambda^{st}_{r,d}$ is cohomologically pure, and the Frobenius eigenvalues in $H^i(\Lambda^{st}_{r,d}, \qlb)$ all of the form $\prod_j \a_j^{l_j}$ with $\sum_i l_j=i$. Moreover,
$$| \Lambda^{st}_{r,d}(\mathbb{F}_q) |=\overline{A}_{g,r,d}(\sigma_X)$$
where $\overline{A}_{g,r,d}(z_1, \ldots, z_{2g})=q^{2(1+(g-1)r^2} A_{g,r,d}(z_1^{-1}, \ldots, z_{2g}^{-1})$ is the Poincarr\'e dual of $A_{g,r,d}$.
\item[ii)] $(k=\mathbb{F}_q, \mathbb{C})$. We have
$$\text{dim}\; H^{1 + (g-1)r^2} (\Lambda^{st}_{r,d}, \mathbb{C})=\text{dim}\; H^{1 + (g-1)r^2}(Higgs_{r,d}^{st}, \mathbb{C})=A_{g,r,d}(0).$$
In particular, the number of irreducible components of $\Lambda_{r,d}^{st}$ is equal to $A_{g,r,d}(0)$.
\end{enumerate}
\end{cor}

\vspace{.1in}

It turns out that the constant term of $A_{g,r,d}$ can be described by a generating series formula similar to (\ref{E:mainT1})~:

\vspace{.1in}

\begin{cor}\label{Cor:26} The values of $A_{g,r,d}(0)$ are computed from the following generating series. Set
\begin{equation}\label{E:mainT10}
\sum_r A^0_{g,r}(z) T^r= - \text{Log}\left( \sum_{\lambda} z^{(g-1)\langle \lambda, \lambda \rangle-l(\lambda)}
K_{\lambda}(z) T^{|\lambda|}\right),
\end{equation}
where for $\lambda=(1^{r_1}, 2^{r_2}, \ldots)$ we have set
$$K_{\lambda}(z)=\frac{1}{\prod_{i}\prod_{j=1}^{r_i}(1-z^{-j})}.$$
Then
$$
A_{g,r,d}(0)=\sum_{\xi \in \mu_r} \xi^{-d} \text{Res}_{z=\xi} \left( A^0_{g,r}(z)\frac{dz}{z}\right).
$$
\end{cor}

\vspace{.1in}

\paragraph{\textit{Examples.}} We list below the values $A_{g,r,d}(0)$ for $r \leq 4$~:
$$A_{g,1,d}(0)=1, \qquad A_{g,2,d}(0)=\begin{pmatrix}g \\ 1\end{pmatrix}, \qquad A_{g,3,d}=4 \begin{pmatrix}g\\2 \end{pmatrix} + \begin{pmatrix} g \\1\end{pmatrix}, \qquad A_{g,4,d}=32 \begin{pmatrix}g\\3 \end{pmatrix} + 20\begin{pmatrix}g\\2 \end{pmatrix} + \begin{pmatrix}g\\1 \end{pmatrix}. $$

It is easy to see that  for any $r \geq 1$, $A_{g,r,d}(0)$ is a polynomial in $g$ of degree $r-1$.

\vspace{.1in}

\noindent
\textit{Remarks}. i) Just like $A_{g,r}(z)$, the rational function $A^0_{g,r}(z)$ has at most simple poles at $r$th roots of unity, and is conjectured to be regular outside of $z=1$.\\
ii) Let $\Sigma_g$ be the quiver with one loop and $g$ vertices, and let $A_{\Sigma_g, r} \in \mathbb{N}[q]$ be the Kac polynomial counting geometrically indecomposable representation of $\Sigma_g$ over $\mathbb{F}_q$ of dimension $r$ (see \cite{Kac}). It is conjectured that $A_{g,r,d}(0)=A_{\Sigma_g,r}(1)$-- this for instance would follow from the main conjecture in \cite{HRVInvent}. However, just as our formula (\ref{E:mainT1}) slightly differs from that conjectured
in \cite{HRVInvent}, so does (\ref{E:mainT10}) slightly differ from Hua's formula computing $A_{\Sigma_g,r}(1)$, see \cite{Hua} (in the latter case, the difference is only in the extra term $-l(\lambda)$ in (\ref{E:mainT10}) !).

\vspace{.1in}

Theorem~\ref{T:3} is proved in Section~6, as are Corollaries~\ref{Cor:1}, \ref{Cor:2}, \ref{Cor:25} and \ref{Cor:26}.

\vspace{.2in}

\paragraph{\textbf{1.5.}} In the last sections ~7 and 8, we point towards two types of possible extensions of the above results~: counting indecomposable vector bundles equipped with quasi-parabolic structures along some (fixed) divisor $D$ of $X$, and counting indecomposables with a prescribed Harder-Narasimhan polygon. We also propose an analog, in our context, of the famous conjecture of Kac (proved by Hausel, see \cite{HauselKac}) relating the constant terms of the Kac polynomials associated to a quiver $Q$ with the dimensions of the root spaces of the corresponding Kac-Moody algebra.

\vspace{.2in}

\section{Orbifolds of pairs}

\vspace{.1in}

\paragraph{\textbf{2.1.}} We fix a smooth projective curve $X$ over a finite field $\mathbb{F}_q$ as in \textbf{1.1.}. Let $Coh(X)$ be the category of coherent sheaves on $X$. By the class of a sheaf $\mathcal{F}$ we will mean the pair 
$$[\mathcal{F}]=(rank(\mathcal{F}), deg(\mathcal{F})) \in (\mathbb{Z}^2)^+=\{(r,d) \in \mathbb{N} \times \mathbb{Z}\;|\; d >0 \;if\; r=0\}.$$
If $\a=(r,d)$ then we put $rk(\a)=r, deg(\a)=d$.
Let $K_0(Coh(X))$ stand for the Grothendieck group of the category $Coh(X)$, and let $\langle\;,\;\rangle : K_0(X) \otimes_{\mathbb{Z}} K_0(X) \to \Z$ be the Euler form, defined by $\langle \mathcal{F},\mathcal{G}\rangle=\text{dim}\;\text{Hom}(\mathcal{F},\mathcal{G})-\text{dim}\;\text{Ext}^1(\mathcal{F},\mathcal{G})$. The form 
$\langle \mathcal{F},\mathcal{G}\rangle$ only depends on the classes $[\mathcal{F}], [\mathcal{G}]$ and is given by
$$\langle (r,d), (r',d') \rangle= (1-g)rr' + (rd'-r'd).$$
We let $(\;,\;)$ stand for the symetrized Euler form, i.e. $(\alpha, \beta)=\langle \alpha, \beta \rangle + \langle \beta,\alpha \rangle$.

\vspace{.1in}

For $\alpha \in (\mathbb{Z}^2)^+$ we denote by $Coh_{\alpha}(X)$ the subcategory of coherent sheaves of class $\alpha$.
We consider the standard slope function $\mu(\mathcal{F})=deg(\mathcal{F})/rank(\mathcal{F})$ and for any
$\nu \in \mathbb{Q} \sqcup \{\infty\}$ we denote by $\mathcal{C}_{\nu}$ the subcategory of $Coh(X)$ consisting of semistable sheaves of slope $\nu$.

\vspace{.1in}

More generally, given a collection $\alpha_1=(r_1, d_1), \ldots, \alpha_t=(r_t,d_t)$ of elements of $(\mathbb{Z}^2)^+$ satisfying $\mu(\alpha_1) > \cdots > \mu(\alpha_t)$
we denote by $\mathcal{C}_{\alpha_1, \ldots, \alpha_t}$ the full subcategory of $Coh_{\alpha_1 + \cdots \alpha_t}(X)$
consisting of objects $\mathcal{F}$ with a Harder-Narasimhan filtration $\mathcal{F}_1 \subset \cdots \subset \mathcal{F}_t=\mathcal{F}$ satisfying $[\mathcal{F}_i/\mathcal{F}_{i-1}]=\alpha_i$ for all $i$. Observe that
the number of isomorphism classes of coherent sheaves in $\mathcal{C}_{\alpha_1, \ldots, \alpha_t}$ is finite, since the number of semistable sheaves of any given class is finite, and $\text{dim\;Ext}^1(\mathcal{H}, \mathcal{G}) < \infty$ for any $(\mathcal{H}, \mathcal{G})$. 

\vspace{.1in}

Finally, let $\mathcal{C}_{\geq \nu}$ be the subcategory of sheaves $\mathcal{F}$ whose
 Harder-Narasimhan filtration
$$\mathcal{F}_1 \subset \mathcal{F}_{2} \subset \cdots \subset \mathcal{F}_t = \mathcal{F}$$
satisfies $$\mu(\mathcal{F}_{1}) > \mu(\mathcal{F}_2/\mathcal{F}_1) > \cdots > \mu(\mathcal{F}/\mathcal{F}_{t-1}) \geq \nu.$$
The categories $\mathcal{C}_{\leq \nu}$ are defined in a similar fashion. We will mostly be interested in the category
$\mathcal{C}_{\geq 0}$. A sheaf belongs to $\mathcal{C}_{\geq 0}$ if and only if it belongs to some $\mathcal{C}_{\alpha_1, \ldots, \alpha_t}$
with $\mu(\alpha_1) > \cdots > \mu(\alpha_t) \geq 0$. For any given $\alpha$, let us put
$$D(\alpha)=\{(\alpha_1, \ldots, \alpha_t)\;|\;  \alpha=\sum \alpha_i,\; \mu(\alpha_1) > \cdots > \mu(\alpha_t) \geq 0\}.$$
The set $D(\alpha)$ is finite. For any given $\a$ the number of isomorphism classes
of coherent sheaves in $\mathcal{C}_{\geq 0}$ of class $\a$ is finite.

The full subcategory $\mathcal{C}_{\geq 0}$ is stable under quotients and extensions. In particular, an object of
$\mathcal{C}_{\geq 0}$ is isomorphic to a direct sum of indecomposable objects, all of which belong to $\mathcal{C}_{\geq 0}$.
The category $\mathcal{C}_{\geq 0}$ is also preserved in a natural sense  by extension of the base field (see e.g. \cite{HuybrechtsLehn}).

\vspace{.2in}

\paragraph{\textbf{2.2.}} For any coherent sheaf $\mathcal{F}$ let us denote by $\text{End}^{nil}(\mathcal{F})$ the subspace of $\text{End}(\mathcal{F})$
consisting of nilpotent endomorphisms. Let us consider the following groupoids
$$\textbf{Nil}_\a(X)=\langle\;(\mathcal{F}, \theta)\;|\; \mathcal{F} \in Coh_{\a}(X),\;\theta \in \text{End}^{nil}(\mathcal{F})\;\rangle$$ 
where an isomorphism $j : (\mathcal{F}, \theta) \stackrel{\sim}{\to} (\mathcal{G}, \phi)$
is an isomorphism $j: \mathcal{F} \stackrel{\sim}{\to} \mathcal{G}$ such that $j \theta = \phi j$.

\vspace{.1in}

The groupoid $\mathbf{Nil}_{\a}(X)$ is of infinite volume as soon as $rk(\a) >0$, i.e. the sum
$$\sum_{(\mathcal{F}, \theta) \in Obj(\mathbf{Nil}_{\a}(X))/\sim} \frac{1}{| \text{Aut}((\mathcal{F}, \theta))|}
=\sum_{\mathcal{F} \in Obj(Coh_{\a}(X))/\sim} \frac{|\text{End}^{nil}(\mathcal{F})|}{|\text{Aut}(\mathcal{F})|}$$
diverges. We also consider the full sub-groupoid 
$$\mathbf{Nil}_{\a}^{\geq 0}(X)=\langle\; (\mathcal{F},\theta) \in \mathbf{Nil}_{\a}(X)\;|\; \mathcal{F} \in \mathcal{C}_{\geq 0}\rangle$$
which has the advantage of being of finite volume
$$vol (\mathbf{Nil}^{\geq 0}_{\a}(X))=\sum_{\mathcal{F} \in
Obj(Coh^{\geq 0}_{\a}(X))/\sim} \frac{|\text{End}^{nil}(\mathcal{F})|}{|\text{Aut}(\mathcal{F})|}<\infty.$$
Observe also that $\mathbf{Nil}^{\geq 0}_{\a}(X)$ is empty if $deg(\a) <0$. 

\vspace{.1in}

The relation between the groupoid $\mathbf{Nil}_{\a}^{\geq 0}(X)$ and the problem of counting indecomposable
coherent sheaves is described by Propositions~\ref{P:1} and \ref{P:2} below. Let us denote by $\mathcal{A}_{\a}^{\geq 0}(X)$ the number of
geometrically indecomposable coherent sheaves in $\mathcal{C}_{\geq 0}$ of class $\a$.

\vspace{.1in}

\begin{prop}\label{P:1} The following relation holds in the ring $\mathbb{Q}[[z^{(1,0)}, z^{(0,1)}]]$~:
$$\sum_{\a} vol (\mathbf{Nil}^{\geq 0}_{\a}(X)) z^{\a} = \text{exp} \left(\sum_{l \geq 1} 
\frac{1}{l} \sum_{\a} \frac{\A_{\a}^{\geq 0}(X \otimes \mathbb{F}_{q^l})}{q^l-1}z^{l\a}\right).$$

\end{prop}
\begin{proof}
We begin by collecting a few standard results on indecomposable coherent sheaves (see e.g. \cite{Atiyah}).

\begin{lem}\label{L:1} The following statements hold~:
\begin{enumerate}
\item[i)] $Coh(X)$ is a Krull-Schmidt category, i.e. every coherent sheaf $M$ decomposes as a direct sum
$$M = M_1^{\oplus n_1} \oplus \cdots \oplus M_s^{\oplus n_s}$$
where the $M_i$ are distinct indecomposables. The $(M_i,n_i)$ are uniquely determined up to permutation.
\item[ii)] Let $M$ be indecomposable. Then $k_M:=\text{End}(M)/rad (\text{End}(M))$ is a field.
\item[ii)] Let $M,M'$ be distinct indecomposables. Then any composed map $M \to M' \to M$ lies  in $rad(\text{End}(M))$.
\end{enumerate}
\end{lem}

\vspace{.1in}

\begin{lem}\label{L:2} Let $M_1, \ldots, M_s$ be distinct indecomposables, and $n_1, \ldots, n_s \in \mathbb{N}$. Put $M = \bigoplus_i M_i^{\oplus n_i}$. Then
\begin{equation}\label{E:L2}
rad(\text{End}(M))=\bigoplus_{i \neq j} Hom(M_i^{\oplus n_i}, M_j^{\oplus n_j}) \oplus \bigoplus_i rad(\text{End}(M_i))^{\oplus n_i}.
\end{equation}
\end{lem}
\begin{proof} Let $U$ denote the right hand side of (\ref{E:L2}). It follows from Lemma~\ref{L:1}, ii) that $U$ is a nilpotent ideal in $\text{End}(M)$
and that $\text{End}(M)/U \simeq \prod_i \mathfrak{gl}(n_i,k_{M_i})$ is a semisimple algebra.
\end{proof}

\vspace{.1in}

\begin{cor}\label{C:1} Let $M$ be as in Lemma~\ref{L:2}. Then we have
\begin{equation}\label{E:C1}
\frac{ |\text{End}^{\;nil}(M)|}{|\text{Aut}(M)|}=\prod_i \frac{|k_{M_i}|^{n_i(n_i-1)}}{|GL(n_i,k_{M_i})|}.
\end{equation}
\end{cor}
\begin{proof} Put $A=\text{End}(M)$ and denote by $p: A \to A/rad(A) \simeq \prod_i \mathfrak{gl}(n_i,k_{M_i})$ the natural projection. Then
$\text{End}^{nil}(M)=p^{-1}(\prod_i \mathcal{N}_{n_i, k_{M_i}})$ where $\mathcal{N}_{n,k} \subset \mathfrak{gl}(n,k)$ denotes the nilpotent cone.
Similarly, $\text{Aut}(M)=p^{-1}(\prod_i GL(n_i,k_{M_i}))$. The result now follows from Lemma~\ref{L:2} and the well-known formula $|\mathcal{N}_{n,k}|=|k|^{n(n-1)}$.
\end{proof}

We now start the proof of Proposition~\ref{P:1}. We first introduce a few useful notations.
Let $\chi$ stand for the set of isomorphism classes of indecomposables in $\mathcal{C}_{\geq 0}$. We
choose a representative $M_{\iota}$ in each class $\iota$ and set $l_{\iota}=[k_{M_\iota} : \mathbb{F}_q]$.
We have an obvious partition $\chi =\bigsqcup_{\alpha} \chi_{\alpha}$ according to the class $\alpha \in (\Z^2)^+$, and we write
$d(\iota) =\alpha$ if $\iota \in \chi_{\alpha}$. 
Note that $\chi_{\alpha}$ is empty if $\alpha \not\in \N^2$.
By Lemma~\ref{L:1}, i) the set of isoclasses of objects in $\mathcal{C}_{\geq 0}$ is
$$Obj(\mathcal{C}_{\geq 0})/\hspace{-.05in}\sim \;=\;\bigg\{ \bigoplus_{\iota \in \chi} M_{\iota}^{\oplus n_{\iota}}\;|\; n_{\iota}=0\;\text{for\;almost\;all\;} \iota\bigg\}.$$
Let $\Theta=\{(n_{\iota}) \in \mathbb{N}^{\chi}\;|\; n_{\iota}=0\;\text{for\;almost\;all\;} \iota\}$. Then, by Corollary~\ref{C:1}
\begin{equation*}
 \begin{split}
  \sum_{\alpha} vol(\textbf{Nil}^{\geq 0}_{\alpha}(X))z^{\alpha}&=\sum_{(n_\iota) \in \Theta}\bigg\{ \prod_\iota \frac{q^{-l_{\iota}n_{\iota}}}{(1-q^{-l_{\iota}n_{\iota}})\cdots (1-q^{-l_{\iota}})} 
z^{\sum_{\iota} n_{\iota}d(\iota)}\bigg\}\\
&=\prod_{\iota \in \chi} \bigg( \sum_{n \geq 0} \frac{q^{-l_{\iota}n}}{(1-q^{-l_{\iota}n})\cdots (1-q^{-l_{\iota}})} 
z^{nd(\iota)}\bigg).
 \end{split}
\end{equation*}
Note that the infinite product converges in the ring $\mathbb{Q}[[z^{(0,1)}, z^{(1,0)}]]$ because any element in $\N^2$ may be written in
only finitely many different ways as a sum $\sum_{\iota} n_{\iota} d(\iota)$ (recall that each $\chi_{\alpha}$ is of finite cardinality).

\vspace{.1in}

Applying Heine's formula
$$\sum_{n \geq 0} \frac{u^n}{(1-v^n) \cdots (1-v)}=exp \bigg( \sum_{l \geq 1} \frac{u^l}{l(1-v^l)}\bigg)$$
we get
\begin{equation*}
 \sum_{\alpha} vol(\textbf{Nil}^{\geq 0}_{\alpha})z^{\alpha}=exp \bigg( \sum_{l \geq 1} \sum_{\iota \in \chi} \frac{z^{ld(\iota)}}{l(q^{ll_{\iota}}-1)}\bigg).
\end{equation*}
To prove Proposition~\ref{P:1} it only remains to show that
\begin{equation}\label{E:P1}
\sum_{l \geq 1} \sum_{\iota \in \chi} \frac{z^{ld(\iota)}}{l(q^{ll_{\iota}}-1)}=\sum_{l \geq 1} \sum_{\alpha} \frac{\A_{\alpha}^{\geq 0}(X \otimes \mathbb{F}_{q^l})}{l(q^l-1)}z^{l\alpha}.
\end{equation}
To this aim, let us denote by $\chi_{\alpha,d}$ the set of elements $\iota \in \chi_{\alpha}$ satisfying $l_{\iota}=d$. Note that $\chi_{\alpha,d}$ is empty
if $\alpha \not\in d \N^2$. The group $Gal(\overline{\mathbb{F}_{q}}/ \mathbb{F}_q)$ acts naturally on the set of isoclasses of indecomposable coherent sheaves on
$X \otimes \overline{\mathbb{F}_{q}}$ of class $\alpha$, preserving the subset of sheaves in $\mathcal{C}_{\geq 0}$. Let $\xi_{\alpha,d}$ stand for the set of 
$Gal(\overline{\mathbb{F}_{q}}/ \mathbb{F}_q)$-orbits of size $d$. Thus
$$| \xi_{\alpha,d} |=| \chi_{d\alpha,d}|, \qquad \A^{\geq 0}_{\alpha}(X \otimes \mathbb{F}_{q^l})=\sum_{d | l} d | \xi_{\alpha,d} |.$$
Using this, we compute
\begin{equation*}
 \begin{split}
  \sum_{l \geq 1} \sum_{\iota \in \chi} \frac{z^{ld(\iota)}}{l(q^{ll_{\iota}}-1)}&=\sum_{l \geq 1}\sum_{\alpha} \sum_{d | \alpha} \frac{| \chi_{\alpha,d}|}{l(q^{dl}-1)}z^{l\alpha}
=\sum_{l \geq 1} \sum_{\nu} \sum_{d \geq 1} \frac{|\chi_{d\nu,d}|}{l(q^{dl}-1)}z^{dl\nu}\\
&=\sum_{l \geq 1} \sum_{\nu} \sum_{ d \geq 1} \frac{|\xi_{\nu,d}|}{l(q^{dl}-1)}z^{ld\nu} =\sum_{l' \geq 1} \sum_{\nu} \sum_{ d | l'} \frac{d|\xi_{\nu,d}|}{l'(q^{l'}-1)}z^{l'\nu}\\
&=\sum_{l' \geq 1} \sum_{\nu} \frac{\mathcal{A}_{\nu}^{\geq 0}(X \otimes \mathbb{F}_{q^{l'}})}{l(q^{l'}-1)}z^{l'\nu}
 \end{split}
\end{equation*}
as wanted. Proposition~\ref{P:1} is proved.
\end{proof}

\vspace{.1in}

\noindent
\textit{Remark.} Proposition~\ref{P:1} may loosely be rephrased as an equality
$$ \sum_{\a} vol (\mathbf{Nil}^{\geq 0}_{\a}(X)) z^{\a} = \text{Exp} \left(\sum_{\a} \frac{\A_{\a}^{\geq 0}(X)}{q-1}z^{\a}\right).$$

\vspace{.1in}

\begin{prop}\label{P:2} Assume that $d > (g-1) r(r-1)$. Then any indecomposable vector bundle of rank $r$ and degree $d$ lies in
$\mathcal{C}_{\geq 0}$, i.e.
$$\A_{r,d}^{\geq 0}(X)=\A_{r,d}(X).$$
\end{prop}
\begin{proof}
Let $\mathcal{F}$ be any coherent sheaf of rank $r$ and degree $d$, and let us denote by $\alpha_1=(r_1, d_1), \ldots, \alpha_t=(r_t,d_t)$
its Harder-Narasimhan type, that is $\mathcal{F} \in \mathcal{C}_{\alpha_1, \ldots, \alpha_t}$. By Serre duality,
any sequence
$$ \xymatrix{
0 \ar[r] &\mathcal{A} \ar[r] & \mathcal{B} \ar[r] & \mathcal{C} \ar[r]& 0}$$
with $\mathcal{A} \in \mathcal{C}_{\geq \nu}$, $\mathcal{C} \in \mathcal{C}_{\leq \nu'}$ and $\nu - \nu' > 2g-2$
splits since then $\text{Ext}^1(\mathcal{C}, \mathcal{A})=\text{Hom}(\mathcal{A}, \mathcal{C} \otimes \Omega_X)=0$. In particular, if
$\mu(\alpha_i) - \mu(\alpha_{i+1}) > 2g-2$ for some $i$ then $\mathcal{F}$ is decomposable. The Proposition easily follows.
\end{proof}

\vspace{.2in}

\section{Jordan stratification}

\vspace{.1in}

\paragraph{\textbf{3.1.}} By Propositions~\ref{P:1}, \ref{P:2}, computing $\A_{\alpha}(X)$ (for all $\alpha$ and for all base field extensions of $X$) 
amounts to computing the volumes of the groupoids $\textbf{Nil}^{\geq 0}_{\alpha}(X)$. We will achieve this by first stratifying $\textbf{Nil}_{\alpha}(X)$
according to Jordan types. We define the Jordan type of a pair $(\mathcal{F}, \theta) \in \textbf{Nil}_{\alpha}(X)$ with $\theta^s=0$ and $\theta^{s-1}\neq 0$ as follows :
$$J(\mathcal{F},\theta)=(\alpha_1, \ldots, \alpha_s)$$
where
\begin{equation}\label{E:defJ}
[\text{Ker}(\theta^i)]-[\text{Ker}(\theta^{i-1})]=\alpha_i + \alpha_{i+1} + \cdots + \alpha_s, \qquad (i =1, \ldots, s).
\end{equation}

Observe that 
$$\sum_i i \alpha_i =\alpha$$
and that some of the $\alpha_i$ may be zero (but $\alpha_s \neq 0$). 

The Jordan type of a pair $(\mathcal{F},\theta)$ contains more information than the Jordan type (in the usual sense) of $\theta$
over the generic point of $X$, as it also keeps track of the degrees of the kernels of powers of $\theta$. We put
$$J^{gen}(\mathcal{F},\theta)=(rk(\alpha_1), \ldots, rk(\alpha_t))$$
where $t$ is the largest index for which $rk(\alpha_t) \neq 0$.

 It may be helpful to visualize the pair $(\mathcal{F}, \theta)$
as a Young diagram. For instance, when $\theta^3=0$ we view $\mathcal{F}$ as the diagram

\vspace{.2in}

\centerline{
\begin{picture}(90,60)
\put(0,0){\line(0,1){60}}
\put(20,0){\line(0,1){60}}
\put(40,0){\line(0,1){40}}
\put(60,0){\line(0,1){20}}
\put(0,0){\line(1,0){60}}
\put(0,20){\line(1,0){60}}
\put(0,40){\line(1,0){40}}
\put(0,60){\line(1,0){20}}
\put(45,8){$\alpha_1$}
\put(25,28){$\alpha_2$}
\put(25,8){$\alpha_2$}
\put(5,8){$\alpha_3$}
\put(5,28){$\alpha_3$}
\put(5,48){$\alpha_3$}
\end{picture}}
\vspace{.05in}
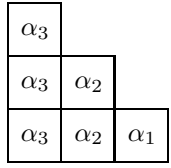
\centerline{\textbf{Figure 2.} Jordan type of a nilpotent endomorphism.}

\vspace{.2in}

\noindent
in which every region $R$ which is saturated in the south and west directions corresponds to a canonical
$\theta$-stable subsheaf $\mathcal{F}_R$ of $\mathcal{F}$. For instance, in the picture above, the
subsheaf $(\text{Ker}(\theta)) + (\text{Ker}(\theta^2) \cap \text{Im}(\theta))$ corresponds to the region

\vspace{.2in}

\centerline{
\begin{picture}(60,40)
\put(0,0){\line(0,1){40}}
\put(20,0){\line(0,1){40}}
\put(40,0){\line(0,1){20}}
\put(60,0){\line(0,1){20}}
\put(0,0){\line(1,0){60}}
\put(0,20){\line(1,0){60}}
\put(0,40){\line(1,0){20}}
\put(45,8){$\alpha_1$}
\put(25,8){$\alpha_2$}
\put(5,8){$\alpha_3$}
\put(5,28){$\alpha_3$}
\end{picture}}
\vspace{.05in}
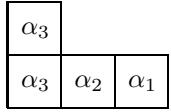
\centerline{\textbf{Figure 3.} A canonical subsheaf.}

\vspace{.2in}

For $\underline{\alpha}=(\alpha_1, \ldots, \alpha_r)$ we denote by $\textbf{Nil}_{\underline{\alpha}}(X)$ the groupoid
consisting of pairs $(\mathcal{F},\theta)$ with $J(\mathcal{F},\theta)=\underline{\alpha}$. Hence we have a stratification 
$$\textbf{Nil}_{\alpha}(X) = \bigsqcup_{|\underline{\alpha}|=\alpha} \textbf{Nil}_{\underline{\alpha}}(X)$$
where we have set $|\underline{\alpha}|=\sum_i i \alpha_i$.

\vspace{.1in}

We introduce several more groupoids~: $\textbf{Coh}_{\nu}(X)$ denotes the groupoid of coherent sheaves on $X$ of class $\nu$; $\textbf{Coh}^{\geq 0}_{\nu}(X)$ is the full sub-groupoid of $\textbf{Coh}_{\nu}(X)$ consisting of coherent sheaves which belong to $\mathcal{C}_{\geq 0}$; for $\underline{\nu}=(\nu_1, \ldots, \nu_s)$ we denote by $\widetilde{\textbf{Coh}}_{\underline{\nu}}(X)$ the groupoid whose objects
are pairs $(\mathcal{H}, \mathcal{H}_\bullet)$ where $\mathcal{H}$ is a coherent sheaf on $X$ of class $\nu_1 + \cdots + \nu_s$ and where $\mathcal{H}_{\bullet}$ is a filtration
$$\mathcal{H}_1 \subset \mathcal{H}_2 \subset \cdots \subset \mathcal{H}_s=\mathcal{H}$$
satisfying $[\mathcal{H}_i]=\nu_1 + \cdots + \nu_i$ for $i=1, \ldots, s$; finally, we let $\widetilde{\textbf{Coh}}^{\geq 0}_{\underline{\nu}}\hspace{-.05in}(X)$ stand for the full sub-groupoid of $\widetilde{\textbf{Coh}}_{\underline{\nu}}(X)$ consisting
of pairs $(\mathcal{H},\mathcal{H}_{\bullet})$ with $\mathcal{H} \in \mathcal{C}_{\geq 0}$. Unless there is a risk of confusion, we will drop the reference to the curve $X$ from now on.

\vspace{.2in}

\paragraph{\textbf{3.2.}} There is a natural functor
$$\pi_{\underline{\alpha}}~: \textbf{Nil}_{\underline{\alpha}} \to \textbf{Coh}_{\alpha_1} \times \cdots \times \textbf{Coh}_{\alpha_s}$$
sending a pair $(\mathcal{F},\theta)$ to the tuple $(\mathcal{F}_1, \ldots, \mathcal{F}_s)$ where
\begin{align*}
&\mathcal{F}_1=\text{Ker}(\theta) / ( \text{Ker}(\theta) \cap \text{Im}(\theta)),\\
&\mathcal{F}_2= \text{Ker}(\theta^2) / ( \text{Ker}(\theta^2) \cap (\text{Im}(\theta) + \text{Ker}(\theta)))\\
&\;\;\vdots \qquad \vdots\\
&\mathcal{F}_{s-1}=\text{Ker}(\theta^{s-1}) / (\text{Ker}(\theta^{s-1}) \cap ( \text{Im}(\theta) + \text{Ker}(\theta^{s-2})))\\
&\mathcal{F}_s=\mathcal{F}/\text{Ker}(\theta^{s-1}).
\end{align*} 

\vspace{.1in}

The functor $\pi_{\underline{\alpha}}$ factors as the composition $\pi_{\underline{\alpha}}=\pi''_{\underline{\alpha}} \circ \pi'_{\underline{\alpha}}$ of the two functors $\pi'_{\underline{\alpha}}~: \textbf{Nil}_{\underline{\alpha}} \to \widetilde{\textbf{Coh}}_{\underline{\alpha}}$ and
$\pi''_{\underline{\alpha}}~: \widetilde{\textbf{Coh}}_{\underline{\alpha}} \to  \textbf{Coh}_{\alpha_1} \times \cdots \times \textbf{Coh}_{\alpha_s}$
respectively defined by
$$\pi'_{\underline{\alpha}} (\mathcal{F}, \theta)=(\mathcal{H}, \mathcal{H}_{\bullet}),$$ 
$$\mathcal{H}=\mathcal{F}/\text{Im}(\theta), \qquad \mathcal{H}_i=\text{Ker}(\theta^i) / ( \text{Ker}(\theta^i) \cap \text{Im}(\theta))$$
and
$$\pi''_{\underline{\alpha}}(\mathcal{H}, \mathcal{H}_{\bullet})=(\mathcal{H}_1, \mathcal{H}_2/\mathcal{H}_1, \ldots, \mathcal{H}/\mathcal{H}_{s-1}).$$

\vspace{.2in}

 Recall that $\langle\;,\;\rangle$, resp. $(\;,\;)$ stands for the Euler form, resp. symetrized Euler form (see Section~\textbf{2.1}). If $\phi~: \mathcal{A} \to \mathcal{B}$ is a functor between groupoids
and $\mathcal{B}' \subset \mathcal{B}$ is a full sub-groupoid then $\phi^{-1}(\mathcal{B}')$ is by definition the full sub-groupoid of $\mathcal{A}$ whose objects satisfy the following condition~: $\phi(A) \simeq B$ for some $B \in \mathcal{B}'$.
This next Proposition is crucial for us.

\vspace{.1in}

\begin{prop}\label{P:3} The following hold~:
\begin{enumerate}
\item[i)] For any $(\mathcal{F}_1, \ldots, \mathcal{F}_s) \in \textbf{Coh}_{\alpha_1} \times \cdots \times \textbf{Coh}_{\alpha_s}$ we have 
$$vol\big(\pi_{\underline{\alpha}}^{-1}(\mathcal{F}_1, \ldots, \mathcal{F}_s)\big)=q^{d(\underline{\alpha})}$$ 
where
$$d(\underline{\alpha})=-\bigg\{\sum_{i}(i-1) \langle \alpha_i, \alpha_i \rangle + \sum_{i<j} i (\alpha_i, \alpha_j) \bigg\}.$$
\item[ii)] For any  $(\mathcal{F}_1, \ldots, \mathcal{F}_s) \in \textbf{Coh}_{\alpha_1} \times \cdots \times \textbf{Coh}_{\alpha_s}$ we have 
$$vol\big((\pi''_{\underline{\alpha}})^{-1}(\mathcal{F}_1, \ldots, \mathcal{F}_s)\big)=q^{d''(\underline{\alpha})}$$ 
where
$$d''(\underline{\alpha})=-\sum_{i<j} \langle \alpha_j, \alpha_i \rangle.$$
\item[iii)] For any $(\mathcal{H}, \mathcal{H}_{\bullet}) \in \widetilde{\textbf{Coh}}_{\underline{\alpha}}$ we have
$$vol\big((\pi'_{\underline{\alpha}})^{-1}(\mathcal{H}, \mathcal{H}^{\bullet})\big)=q^{d'(\underline{\alpha})}$$
where $d'(\underline{\alpha})=d(\underline{\alpha})-d''(\underline{\alpha})$.
\item[iv)] We have 
$$(\pi'_{\underline{\alpha}})^{-1}\big( \widetilde{\textbf{Coh}}_{\underline{\alpha}}^{\geq 0}\big)
= \textbf{Nil}_{\underline{\alpha}}^{\geq 0}.$$
\end{enumerate}
\end{prop}
\begin{proof} The proofs of statements i)--iii) are completely analogous to \cite[Prop. 3.1, Cor. 3.2]{Heinloth}.
We turn to iv). Given $(\mathcal{F},\theta) \in \textbf{Nil}_{\underline{\alpha}}$ we have to show that
$\mathcal{F} \in \mathcal{C}_{\geq 0}$ if and only if $\mathcal{F} / \text{Im}(\theta) \in \mathcal{C}_{\geq 0}$. As $\mathcal{C}_{\geq 0}$ is closed under quotients, $\mathcal{F} \in \mathcal{C}_{\geq 0} \Rightarrow \mathcal{F}/ \text{Im}(\theta) \in \mathcal{C}_{\geq 0}$. To get the reverse implication, note that $\theta$ induces surjective morphisms $\mathcal{F} / \text{Im}(\theta) \twoheadrightarrow \text{Im}(\theta) / \text{Im}(\theta^2) \twoheadrightarrow \cdots \twoheadrightarrow \text{Im}(\theta^{s-1})$. Hence if $\mathcal{F}/\text{Im}(\theta) \in \mathcal{C}_{\geq 0}$ then so do $\text{Im}(\theta^i) / \text{Im}(\theta^{i+1})$ for $i=1, \ldots, s-1$. But as $\mathcal{C}_{\geq 0}$ is also stable under extensions, this implies that $\mathcal{F} \in \mathcal{C}_{\geq 0}$.
Proposition~\ref{P:3} is proved.
\end{proof}

\vspace{.1in}

\begin{cor}\label{C:2} We have
$$vol\big( \textbf{Nil}_{\underline{\alpha}}^{\geq 0}\big)=q^{d'(\underline{\alpha})} vol \big( \widetilde{\textbf{Coh}}_{\underline{\alpha}}^{\geq 0}\big).$$
\end{cor}

\vspace{.2in}

\section{Hall algebras of curves}

\vspace{.1in}

\paragraph{\textbf{4.1.}} The volume of $ \widetilde{\textbf{Coh}}_{\underline{\alpha}}^{\geq 0}$ may be
computed using some standard techniques in the theory of automorphic functions over function fields for the groups $GL(n)$. 
We will use the language of Hall algebras which we briefly recall in this section. We refer to e.g. \cite{Kapranov}, \cite{KSV}, or \cite[Lect. 4]{SLectures} for details.

\vspace{.1in}

For any $\nu \in (\Z^2)^+$ we set $\mathcal{I}_{\nu}=Obj(\textbf{Coh}_{\nu})/\hspace{-.05in}\sim$ and we let  $\mathcal{H}_{\nu}=\text{Fun}(\mathcal{I}_{\nu}, \C)$ be the $\C$-vector space of all functions $\mathcal{I}_{\nu}\to \C$. There is a natural
convolution diagram
$$\xymatrix{\textbf{Coh}_{\nu_1} \times \textbf{Coh}_{\nu_2}& \widetilde{\textbf{Coh}}_{\nu_2,\nu_1} \ar[l]_-{p} \ar[r]^-{s} & \textbf{Coh}_{\nu_1+\nu_2}}$$
where $p(\mathcal{H}, \mathcal{H}_1 \subset \mathcal{H})=(\mathcal{H}/\mathcal{H}_1, \mathcal{H}_1)$ and $s(\mathcal{H},\mathcal{H}_1 \subset \mathcal{H})=\mathcal{H}$. This induces maps 
\begin{equation*}
\begin{split}
m_{\nu_1,\nu_2}~:\mathcal{H}_{\nu_1} \otimes \mathcal{H}_{\nu_2} &\to \mathcal{H}_{\nu_1+\nu_2}\\
f \otimes g &\mapsto q^{\frac{1}{2}\langle \nu_1, \nu_2\rangle}s_*p^*(f \boxtimes g),
\end{split}
\end{equation*}
and
\begin{equation*}
\begin{split}
\Delta'_{\nu_1,\nu_2}~:\mathcal{H}_{\nu_1+\nu_2} &\to \text{Fun}(\mathcal{I}_{\nu_1} \times \mathcal{I}_{\nu_2},\C)\\
 h &\mapsto q^{\frac{1}{2}\langle \nu_1, \nu_2\rangle} p_*s^*(h).
\end{split}
\end{equation*}
Note that $\text{Fun}(\mathcal{I}_{\nu_1} \times \mathcal{I}_{\nu_2},\C)$ is a natural completion of 
$\mathcal{H}_{\nu_1} \otimes \mathcal{H}_{\nu_2}$. We will denote this completion by
$\mathcal{H}_{\nu_1} \hat{\otimes} \mathcal{H}_{\nu_2}$. Taking the direct sum over all $\nu$ yields an algebra and (topological) coalgebra structure on $\mathcal{H}'=\bigoplus_{\nu} \mathcal{H}_{\nu}$. As defined, this is a bialgebra
only in a twisted sense. Let $\mathbf{K}=\bigoplus_{\nu \in \Z^2} \C \mathbf{k}_{\nu}$ be the group algebra of $\Z^2$.
The (extended) Hall algebra of $X$ is the semidirect tensor product $\mathcal{H}=\mathcal{H}' \otimes \mathbf{K}$ with respect to the action
$$\mathbf{k}_{\nu} f\mathbf{k}_{-\nu}=q^{\frac{1}{2}(\nu,\alpha)}f, \qquad \text{for}\;f \in \mathcal{H}_{\alpha}.$$
It is equipped with a comultiplication satisfying
$$\Delta(\mathbf{k}_{\nu})=\mathbf{k}_{\nu} \otimes \mathbf{k}_{\nu}, $$
$$\Delta(f)=\sum_{\nu_1+\nu_2=\nu} \Delta'_{\nu_1,\nu_2}(f) \cdot ( \mathbf{k}_{\nu_2} \otimes 1)\qquad \text{for}\; f \in \mathcal{H}_{\nu}.$$
By Green's theorem, $\mathcal{H}$ is a (topological) bialgebra. We will occasionally write $\Delta_{r,s}$
for the component of $\Delta$ of rank $(r,s)$ (hence letting the degrees vary). Observe that $\mathbf{k}_{(0,1)}$ is central. When it bears no consequence, it is sometimes convenient to omit the degree in the notation (for instance, writing simply $\mathbf{k}_1$ for $\mathbf{k}_{1,d}$). We hope that that the reader won't find this slight abuse of notation too confusing.

\vspace{.1in}

 Let $\mathcal{H}^{fin} \subset \mathcal{H}$ be the subalgebra of $\mathcal{H}$ consisting of functions with finite support. The algebra $\mathcal{H}^{fin}$ is equipped with a symmetric Hopf pairing satisfying
$$(\mathbf{k}_{\nu}\;|\; \mathbf{k}_{\mu})=q^{\frac{1}{2}(\nu,\mu)},\qquad (1_{\mathcal{F}}\;|\; 1_{\mathcal{G}})=\frac{\delta_{\mathcal{F},\mathcal{G}}}{|\text{Aut}(\mathcal{F})|}$$
and
$$( ab\;|\;c)=(a \otimes b\;|\; \Delta(c)),\qquad \forall\; a,b,c \in \mathcal{H}^{fin}.$$

\vspace{.2in}

\paragraph{\textbf{4.2.}} We will use the following notation~: for $\nu \in (\Z^2)^+$ we denote by $1_{\nu}, 1^{vec}_{\nu}, 1^{\geq 0}_{\nu}$
the characteristic functions of $\textbf{Coh}_{\nu}$, of the sub-orbifold $\textbf{Bun}_{\nu}$ of $\textbf{Coh}_{\nu}$ parametrizing vector bundles and of $\textbf{Coh}^{\geq 0}_{\nu}$ respectively. Thus
$$(1_{\nu}\;|\;1_{\nu})=vol(\textbf{Coh}_{\nu}), \quad (1_{\nu}^{vec}\;|\;1^{vec}_{\nu})=vol(\textbf{Bun}_{\nu}), \quad (1^{\geq 0}_{\nu}\;|\; 1^{\geq 0}_{\nu})=vol(\textbf{Coh}^{\geq 0}_{\nu}).$$
Moreover, it is easy to see from the definitions that if $\nu=(r,d)$ with $r \geq 1$ then
$$1_{\nu}=\sum_{l \geq 0}q^{-\frac{1}{2}l}1^{vec}_{\nu-(0,l)} 1_{0,l}.$$

\vspace{.1in}

Unraveling the definitions we have that for any $\underline{\alpha}=(\alpha_1, \ldots, \alpha_s)$,
\begin{equation}\label{E:cohtilde}
vol\left(\widetilde{\textbf{Coh}}_{\underline{\alpha}}^{\geq 0}\right)=q^{-\frac{1}{2}\sum_{i > j} \langle \alpha_i,\alpha_j\rangle}\left( 1_{\alpha_s} \cdots 1_{\alpha_1}\;|\; 1_{\sum \alpha_i}^{\geq 0}\right).
\end{equation}

\vspace{.1in}

\begin{theo}\label{T:volform} We have~:
\begin{enumerate}
\item[i)] If $\nu=(r,d)$ with $r >0$ then
$$(1^{vec}_{\nu}\;|\;1^{vec}_{\nu})=\frac{q^{(g-1)(r^2-1)}}{q-1}|\text{Pic}^0(X)| \zeta_X(q^{-2}) \cdots \zeta_X(q^{-r}).$$
\item[ii)] We have
$$\sum_{l \geq 0} (1_{0,l}\;|\;1_{0,l})s^l=\text{Exp}\left( \frac{|X(\mathbb{F}_q)|}{q-1}s\right)=\prod_{i=1}^{\infty} \zeta_X(q^{-i}s).$$
\item[iii)] If $\nu=(r,d)$ with $r>0$ then
$$(1_{\nu}\;|\; 1_{\nu})=\frac{q^{(g-1)(r^2-1)}}{q-1}|\text{Pic}^0(X)| \prod_{i=2}^{\infty}\zeta_X(q^{-i}).$$
\end{enumerate}
\end{theo}
\begin{proof}
The first statement is well known as the Siegel formula, see \cite[Prop. 2.3.4]{HN} and also \cite{Behrend} for a motivic analog. The second statement is also well known but we indicate a proof in the appendix as we have not been abe to locate a precise reference. The last statement is an easy consequence of i) and ii) together with the fact that
$$vol(\textbf{Coh}_{\nu})=\sum_{l \geq 0} q^{-rl} vol(\textbf{Bun}_{\nu-(0,l)}) vol(\textbf{Coh}_{0,l}).$$
We note that the cohomology of the moduli stacks $Coh_{\nu}$ have been determined for all $\nu$ by J. Heinloth (see \cite{Jochen}); the above formulas ii) and iii) may alternatively be deduced from \textit{loc.\,cit} together with the fact that the stacks $Coh_{\nu}$ are very pure.
\end{proof}

For any $r>0$ we set
$$vol_r=vol(\textbf{Bun}_{r,d}(X))$$
(this is independent of $d$, given explicitly in Theorem~\ref{T:volform} i) ).

\vspace{.2in}

\paragraph{\textbf{4.3.}} Our computation uses some well-known elementary properties of Eisenstein series, which we recall in this paragraph. For any $r \geq 0$ let us consider the series
$$E_r(z) =\sum_{d \in \Z} 1_{r,d} z^d, \qquad E_r^{vec}(z)=\sum_{d \in \Z} 1^{vec}_{r,d}z^d$$
which both belong to $\prod_{d \in \Z} \mathcal{H}_{r,d}$. 
We also put
$$E_{r_s, \ldots, r_1}(z_s, \ldots, z_1)=E_{r_s}(z_s) \cdots E_{r_1}(z_1) \in \prod_{d \in \Z} (\mathcal{H}_{r,d}[[z_s^{\pm 1}, \ldots, z_1^{\pm 1}]])$$
where $r=\sum_i r_i$, and define $E_{r_s, \ldots, r_1}^{vec}(z_s, \ldots, z_1)$ likewise. It was shown by Harder (see \cite{Harder}) that for any coherent sheaf $\mathcal{F}$ of rank $r$ the value of $E_{r_s, \ldots, r_1}(z_s, \ldots, z_1)$ on $\mathcal{F}$ is the expansion in the region $z_1 \gg z_2 \gg \cdots \gg z_s$ of a rational function.

\vspace{.1in}

\begin{lem}\label{L:Hecke} The following relations hold~:
\begin{enumerate}
\item[i)] $$E_0(z)E_0(w)=E_0(w)E_0(z),$$
\item[ii)] for any $r \geq 1$ we have
$$E_0(z) E_r^{vec}(w)=\left(\prod_{i=0}^{r-1} \zeta\left(q^{-\frac{r}{2}+i} \frac{z}{w} \right)\right)\; E_r^{vec}(w)E_0(z),$$
where the rational function $\prod_i \zeta_X\left(q^{-\frac{r}{2}+i} \frac{z}{w} \right)$ is expanded in the region $w \gg z$.
\end{enumerate}
\end{lem}
\begin{proof} This is classical, see e.g. \cite[Thm. 6.3]{SVCompositio} (the proof is given there when $g_X=1$ but the same proof works for an arbitrary curve). The first statement simply expresses the commutativity of Hecke operators, while the second expresses the fact that the constant functions are Hecke eigenfunctions.
\end{proof}

\vspace{.1in}

\begin{lem}\label{L:volcoprod} For any $r \geq 0$ we have
$$\Delta'(E_r(z))=\sum_{s+t=r}q^{\frac{1}{2}st(g-1)} E_{s}(q^{\frac{t}{2}}z) \otimes E_{t}(q^{-\frac{s}{2}}z)$$
$$\Delta'(E^{vec}_r(z))=\sum_{s+t=r}q^{\frac{1}{2}st(g-1)} E^{vec}_{s}(q^{\frac{t}{2}}z)E_0(q^{\frac{t-s}{2}}z)E_0^{-1}(q^{-\frac{t+s}{2}}z) \otimes E^{vec}_{t}(q^{-\frac{s}{2}}z).$$
\end{lem} 
\begin{proof} This is the formula for the constant term of the constant function; see e.g. \cite[Prop. 6.2]{SVCompositio}.
\end{proof}

\vspace{.1in}

\begin{lem}\label{L:residue} For $r\geq 1$ we have
\begin{equation*}
E_r^{vec}(q^{\frac{1}{2}(1-r)}z_1)
=C \cdot \text{Res}_{\frac{z_r}{z_{r-1}}=\frac{z_{r-1}}{z_{r-2}}=\cdots =\frac{z_2}{z_1}=q^{-1}}\left( E^{vec}_1(z_r) \cdots E^{vec}_1(z_1)\frac{dz_1}{z_1} \cdots \frac{dz_{r-1}}{z_{r-1}}\right)
\end{equation*}
where
$$C=q^{-\frac{1}{4}(g-1)r(r-1)}vol_1^{-r}vol_r. $$
\end{lem}
\begin{proof} This is the formula expressing the constant function on $Bun_{GL(r)}(X)$ as a residue of an Eisenstein series, see e.g. \cite{Harder}.
\end{proof}

\vspace{.2in}

\paragraph{\textbf{4.4.}} We will need some appropriate truncations of the series $E_r(z)$ and $E^{vec}_r(z)$. Put
$$1^{\geq 0}_{r,d}=1_{\textbf{Coh}^{\geq 0}_{r,d}}, \qquad 1^{vec,\geq 0}_{r,d}=1_{\textbf{Bun}^{\geq 0}_{r,d}}, \qquad 1^{<0}_{r,d}=1_{\textbf{Bun}^{ <0}_{r,d}}$$
where $\textbf{Bun}^{<0}_{r,d}$ is the full subgroupoid of $\textbf{Bun}_{r,d}$ whose objects are vector bundles belonging to $\mathcal{C}_{<0}$. We also set
$$E^{\geq 0}_r(z)=\sum_{d \in \Z} 1^{\geq 0}_{r,d}z^d, \qquad E^{vec,\geq 0}_{r}(z)=\sum_{d \in \Z} 1^{vec,\geq 0}_{r,d}z^d, \qquad E^{<0}_{r}(z)=\sum_{d \in \Z} 1^{<0}_{r,d}z^d.$$
The unicity of the Harder-Narasimhan filtration yields the following relations~:
\begin{equation}\label{E:><}
\begin{split}
E_r(z)&=\sum_{\substack{s+t=r\\ s,t \geq 0}} q^{\frac{1}{2}(g-1)st}E_s^{<0}(q^{\frac{t}{2}}z)E^{\geq 0}_t(q^{-\frac{1}{2}s}z), \\
E^{vec}_r(z)&=\sum_{\substack{s+t=r\\s,t \geq 0}} q^{\frac{1}{2}(g-1)st}E_s^{<0}(q^{\frac{t}{2}}z)E^{\geq 0}_t(q^{-\frac{1}{2}s}z).
\end{split}
\end{equation}

\vspace{.2in}

\paragraph{\textbf{4.5.}}  Let $\mathcal{H}^{sph} \subset \mathcal{H}^{fin}$ be the subalgebra generated by $\mathbf{K}$ and the characteristic functions $\mathbf{1}^{vec}_{1,d}$ 
and $\mathbf{1}_{0,d}$ of the connected components of $\textbf{Pic}(X)$ and $\textbf{Coh}_0(X)$, the orbifold of torsion sheaves on $X$. 
This subalgebra is studied in \cite{SVMathAnn} and \cite{SKor}. In particular it is shown in \cite[Thm~3.1]{SKor} that the characteristic function $1_{\mathcal{C}_{\a_1, \ldots, \a_t}}$ of any HN strata $\mathcal{C}_{\a_1, \ldots, \a_t}$ belongs to $\mathcal{H}^{sph}$. 
One nice feature of $\mathcal{H}^{sph}$ is that it possesses an \textit{integral (or generic) form} in the following sense. Let us fix a genus $g \geq 0$,
put $R_g=\mathbb{Q}[T_g]^{W_g}$ and recall that $K_g$ is the localization of $R_g$ at the set $\{q^l-1\;|\; l \geq 1\}$ where by definition $q(\sigma_1, \ldots, \sigma_{2g})=\sigma_{2i-1}\sigma_{2i}$
for any $1 \leq i \leq g$. (see Section~\textbf{1.1}). For any choice of smooth projective curve $X$ of genus $g$ there is a natural map $K_g \to \mathbb{C}, f \mapsto f(\sigma_X)$. 

\vspace{.1in}

\begin{theo}[\cite{SVMathAnn}, \cite{SKor}]\label{T:sv} There exists an $R_g$-Hopf algebra ${}_R\mathcal{H}^{sph}$ equipped with
 a Hopf pairing $$(\;\;)~: {}_R\mathcal{H}^{sph} \otimes {}_R\mathcal{H}^{sph} \to K_g,$$
generated by elements ${}_R1_{0,l}, {}_R1^{vec}_{1,d}, l \geq 1, d \in \mathbb{Z}$, containing elements
${}_R1_{\mathcal{C}_{\a_1, \ldots, \a_t}}$ for any HN strata $\mathcal{C}_{\a_1, \ldots, \a_t}$ and having the following property~: for any smooth connected projective curve
$X$ of genus $g$ defined over a finite field $\mathbb{F}_q$ there exists a specialisation morphism of Hopf algebras
$$\Psi_X~: {}_R\mathcal{H}^{sph}\otimes_{R_g} \mathbb{C} \twoheadrightarrow \mathcal{H}_X^{sph}$$
such that
$$\Psi_X({}_R1_{\mathcal{C}_{\a_1, \ldots, \a_t}})=1_{\mathcal{C}_{\a_1, \ldots, \a_t}}$$
for any HN strata $\mathcal{C}_{\a_1, \ldots, \a_t}$.
\end{theo}
\begin{proof} The existence of ${}_R\mathcal{H}^{sph}$ is shown in \cite[1.11]{SVMathAnn}. The existence of the elements ${}_R1_{\mathcal{C}_{\a_1, \ldots, \a_t}}$ is proved in exactly
the same fashion as in \cite[Thm.~3.1]{SKor}.
\end{proof}

\vspace{.1in}

\begin{cor}\label{C:cor23} For any tuple $\underline{\alpha}=(\alpha_1,\ldots, \alpha_s)$ there exists an element $B^{\geq 0}_{g,\underline{\alpha}} \in K_g$ such that $$vol\bigg(\widetilde{\textbf{Coh}}^{\geq 0}_{\underline{\alpha}}(X)\bigg)=B^{\geq 0}_{g,\underline{\alpha}}(\sigma_X)$$ for any $X$.
\end{cor}
\begin{proof} By (\ref{E:cohtilde}) and Theorem~\ref{T:sv} it is enough to show that the pairing 
\begin{equation}\label{E:cor3}
( 1_{\a_s}\cdots 1_{\a_1}\;|\; 1_{\sum \a_i}^{\geq 0})
\end{equation}
may be expressed as a pairing between explicit polynomials in elements $1_{\mathcal{C}_{\beta_1, \ldots, \beta_t}}$. On the one hand, we have
$$1_{\sum \a_i}^{\geq 0}=\sum_{\underline{\beta}} 1_{\mathcal{C}_{\underline{\beta}}}$$
where $\underline{\beta}$ ranges among the (finite) set of all HN types $(\beta_1, \ldots, \beta_t)$ such that $\sum \beta_i=\sum \a_i$ and $\mu(\beta_1) \geq 0$. On the other hand, we have
\begin{equation}\label{E:cor2}
1_{\a_s} \cdots 1_{\a_1}=\sum_{\underline{\beta}_s, \ldots, \underline{\beta}_1} 1_{\mathcal{C}_{\underline{\beta}_s}} \cdots 1_{\mathcal{C}_{\underline{\beta}_1}}
\end{equation}
where the sum ranges over all tuples $(\underline{\beta}_s, \ldots, \underline{\beta}_1)$ of HN types of respective class
$\a_s, \ldots, \a_1$. Write
$$\underline{\beta}_i=(\beta^{(i)}_1, \ldots, \beta^{(i)}_{t_i}), \qquad (1 \leq i \leq s).$$
We claim that the pairing
$(1_{\mathcal{C}_{\underline{\beta}_s}} \cdots 1_{\mathcal{C}_{\underline{\beta}_1}}\;|\; 1^{\geq 0}_{\sum \a_i})$
may be nonzero only when
\begin{equation}\label{E:cor1}
\mu(\a_s + \cdots + \a_{i+1} + \beta^{(i)}_{1} + \cdots + \beta^{(i)}_{l}) \geq 0
\end{equation}
for all possible choices of $i$ and $l$. 
Indeed, if (\ref{E:cor1}) does not hold then there exists a coherent sheaf $\mathcal{F} \in \mathcal{C}_{\geq 0}$ satisfying $(1_{\mathcal{C}_{\underline{\beta}_s}} \cdots 1_{\mathcal{C}_{\underline{\beta}_1}}\;|\;\mathcal{F}) \neq 0$
having some quotient of negative slope. Observe that condition (\ref{E:cor1}) reduces the number of summands in
(\ref{E:cor2}) contributing to (\ref{E:cor3}) to a finite set. We are done.
\end{proof}

From the above corollary one deduces that for any $\a \in (\mathbb{Z}^2)^+$ there exist an element $C_{g,\a}^{\geq 0} \in K_g$ such that
$$ vol\bigg( \textbf{Nil}^{\geq 0}_{\a}(X)\bigg)=C_{g,\a}^{\geq 0}(\sigma_X)$$
for any $X$. Therefore using Proposition~\ref{P:1} we obtain the relation
\begin{equation}\label{E:thm1}
\sum_{l \geq 1} 
\frac{1}{l} \sum_{\a} \frac{\A_{\a}^{\geq 0}(X \otimes \mathbb{F}_{q^l})}{q^l-1}z^{l\a}=log \left( \sum_{\a} C_{g,\a}^{\geq 0}z^{\a}\right).
\end{equation}
This implies that for any $\a$ there exists an element $A_{g,\a}^{\geq 0} \in K_g$ such that $\A_{\a}^{\geq 0}(X)=A_{g,\a}^{\geq 0}(\sigma_X)$ for any $X$. Indeed, this follows immediately from (\ref{E:thm1}) for $\a=(r,d)$
with $r$ and $d$ relatively prime, and from there by an easy induction on $gcd(r,d)$ for an arbitrary pair $\a$. Using Proposition~\ref{P:2} we therefore have~:

\vspace{.1in}

\begin{cor}\label{C:25} For any $g$ and any $\a$ there exists an element $A_{g,\a} \in K_g$ such that
$\A_{\a}(X)=A_{g,\a}(\sigma_X)$ for any smooth projective curve $X$ of genus $g$.
\end{cor}

\vspace{.2in}

\paragraph{\textbf{4.6.}} To finish the proof of Theorem~\ref{T:1} it remains to prove the unicity of $A^{\geq 0}_{g,\a}$.
For this, let us fix a prime number $l$, an embedding $\iota: \qlb \to \C$ and consider the collection $\mathcal{X}_g$ of all smooth projective geometrically connected curves $X$ of genus $g$ defined over some finite field $\mathbb{F}_q$ with $l$ not dividing $q$. Setting
$$\mathcal{W} =\{ \sigma_X\;|\; X \in \mathcal{X}_g\} \subset T_g/W_g$$
we see that the unicity statement of Threorem~\ref{T:1} boils down to the following fact, whose proof is given in the appendix~:

\begin{prop}\label{P:Zariski} The set $\mathcal{W}$ is Zariski dense in
$T_g/W_g$.
\end{prop}

\vspace{.2in}

\section{Volume of the orbifold of pairs}

\vspace{.1in}

\paragraph{\textbf{5.1.}} Let us introduce the following generating series~:
$$G^{\geq 0}_{r_s, \ldots, r_1}(z_s, \ldots, z_1;w):=\left( E_{r_s, \ldots, r_1}(z_s, \ldots, z_1)\;\bigg|\; E_{n}^{\geq 0}(w) \right)$$
where $n =\sum_i r_i$. Note that we allow some of the $r_i$ to be zero. From the proof of Corollary~\ref{C:cor23} we see that $G^{\geq 0}_{r_s, \ldots, r_1}(z_s, \ldots, z_1; w)$ belongs to the vector space $K_g\left[\left[\frac{z_s}{z_{s-1}}, \frac{z_{s-1}}{z_{s-2}}, \ldots, \frac{z_2}{z_1}, z_1, w\right]\right]$, and is the expansion of some rational function.

\vspace{.1in}

\begin{prop}\label{P:4} For any $r_s, \ldots, r_1$ we have 
$$G^{\geq 0}_{r_s, \ldots, r_1}(z_s, \ldots, z_1;w)=X_{r_1, \ldots, r_s}(z_s, \ldots, z_1;w) \cdot A^{\geq 0}_{r_s, \ldots, r_1}(z_r, \ldots, z_1;w)$$
where
$$A^{\geq 0}_{r_s, \ldots, r_1}(z_s, \ldots, z_1;w)=
 \left( E^{vec}_{r_s, \ldots, r_1}(z_s, \ldots, z_1)\;\big|\; E_{n}^{\geq 0}(w) \right)$$
and
$$X_{r_s, \ldots, r_1}(z_s, \ldots, z_1;w)=\text{Exp}\left(\frac{|X(\mathbb{F}_q)|}{q-1}\left[\sum_{i} q^{-\frac{1}{2}(n+r_i)}z_iw \,+ \,\sum_{i>j} \frac{z_i}{z_j}\big(q^{\frac{r_j}{2}}-q^{-\frac{r_j}{2}}\big)q^{-\frac{r_i}{2}}\right]\right).$$
\end{prop}
\begin{proof} Let us abbreviate $\underline{r}=(r_s, \ldots r_1)$ and $\underline{z}=(z_s, \ldots , z_1)$. From the relation $E_n^{\geq 0}(w)=E^{vec,\geq 0}_n(w) E_0(q^{-\frac{n}{2}}w)$ and using (twice) the Hopf property of the pairing, we get
\begin{equation}\label{E:prop41}
\begin{split}
G^{\geq 0}_{\underline{r}}(\underline{z};w)&=\left( \Delta_{r_s,0}(E_{r_s}(z_s)) \cdots \Delta_{r_1,0}(E_{r_1}(z_1)) \;|\; E^{vec,\geq 0}_n(w) \otimes E_0(q^{-\frac{n}{2}}w)\right)\\
&=\left(\big(E_{r_s}(z_s)\textbf{k}_0 \otimes E_0(q^{-\frac{r_s}{2}}z_s)\big) \cdots \big(E_{r_1}(z_1)\textbf{k}_0 \otimes E_0(q^{-\frac{r_1}{2}}z_1)\big)\;|\;E^{vec,\geq 0}_n(w) \otimes E_0(q^{-\frac{n}{2}}w)\right)\\
&=\left( E_{r_s}(z_s) \cdots E_{r_1}(z_1)\mathbf{k}_0^s\otimes E_0(q^{-\frac{r_s}{2}}z_s) \cdots E_0(q^{-\frac{r_1}{2}}z_1)\;|\;E^{vec,\geq 0}_n(w) \otimes E_0(q^{-\frac{n}{2}}w)\right)\\
&=\left(E_{\underline{r}}(\underline{z})\;|\; E_n^{vec,\geq 0}(w)\right) \cdot \left( \prod_i E_0(q^{-\frac{r_i}{2}}z_i)\;|\; E_0(q^{-\frac{n}{2}}w)\right)\\
&=\left(E_{\underline{r}}(\underline{z}))\;|\; E_n^{vec,\geq 0}(w)\right) \cdot \prod_i \left(E_0(q^{-\frac{r_i}{2}}z_i)\;|\; E_0(q^{-\frac{n}{2}}w)\right)\\
&=\left(E_{\underline{r}}(\underline{z})\;|\; E_n^{vec,\geq 0}(w)\right) \cdot \text{Exp}\left(\frac{|X(\mathbb{F}_q)|}{q-1} \sum_i  q^{-\frac{1}{2}(n+r_i)}z_iw \right)
\end{split}
\end{equation}
The last step of the above calculation uses Theorem~\ref{T:volform}, ii). Using the relation $E_{r_i}(z_i)=E^{vec}_{r_i}(z_i) E_0(q^{-\frac{r_i}{2}}z_i)$ and the Hecke relations (see Lemma~\ref{L:Hecke}) we get
\begin{equation}\label{E:prop42}
\begin{split}
\left(E_{\underline{r}}(\underline{z})\;|\; E_n^{vec,\geq 0}(w)\right)&= \left( E_{r_s}^{vec}(z_s)E_0(q^{-\frac{r_s}{2}}z_s) E_{r_{s-1}}^{vec}(z_{r_{s-1}}) \cdots E_{r_1}^{vec}(z_1)E_0(q^{-\frac{r_1}{2}}z_1)\;|\; E_n^{vec,\geq 0}(w)\right)\\
&=\text{Exp}\left(\frac{|X(\mathbb{F}_q)|}{q-1} \left[ \sum_{i>j} \frac{z_i}{z_j}\big( q^{\frac{r_j}{2}}-q^{-\frac{r_j}{2}}\big)q^{-\frac{r_i}{2}} \right]\right) \cdot \left(E_{\underline{r}}(\underline{z};w)\;|\; E_n^{vec,\geq 0}(w)\right)
\end{split}
\end{equation}
Above we have made use of the obvious relation
$$\prod_{i=0}^{n-1} \zeta\left( q^{-\frac{n}{2}+i}\frac{z}{w}\right)=\text{Exp}\left( \frac{|X(\mathbb{F}_q)|}{q-1}\frac{z}{w} \big(q^{\frac{n}{2}}-q^{-\frac{n}{2}}\big)\right).$$
Combining (\ref{E:prop41}) and (\ref{E:prop42}) yields the Proposition.
\end{proof}

\vspace{.2in}

\paragraph{\textbf{5.2.}} Motivated by Proposition~\ref{P:4}, we introduce some generating series
$$A^{*}_{r_s, \ldots, r_1}(z_s, \ldots, z_1;w):=\left( E^{vec}_{r_s, \ldots, r_1}(z_s, \ldots, z_1)\;\bigg|\; E^{*}_{r}(w) \right)$$
where $r=\sum r_i$ and where the symbol $*$ is either empty or corresponds to one of the modified series considered in Section~4.4, i.e. $\star$ belongs to the set $\{\geq 0,<0\}.$
 As before, these series belong to the
vector space of formal sums $K_g[[z_s^{\pm 1}, \ldots, z_1^{\pm 1}, w^{\pm 1}]]$. By construction, the coefficient
$(1_{r_s,d_s} \cdots 1_{r_1,d_1}\;|\; 1_{r,d}^*)$ of $z_s^{d_s}\cdots z_1^{d_1}w^d$ in $A^*_{r_s, \ldots, r_1}(z_s, \ldots, z_1;w)$ is nonzero only if $d=\sum_i d_i$. Observe that
$$A^{\geq 0}_{r_s, \ldots, r_1}(z_s, \ldots, z_1;w) \in K_g\left[\left[\frac{z_s}{z_{s-1}}, \frac{z_{s-1}}{z_{s-2}}, \ldots, \frac{z_2}{z_1}, z_1, w\right]\right]$$
while
$$A^{< 0}_{r_s, \ldots, r_1}(z_s, \ldots, z_1;w) \in z_1^{-1}K_g\left[\left[z_s^{-1},\frac{z_s}{z_{s-1}}, \frac{z_{s-1}}{z_{s-2}}, \ldots, \frac{z_2}{z_1}, w^{-1}\right]\right]$$
and that the Fourier coefficients in $w$ of either $A^{\geq 0}_{r_s, \ldots, r_1}(z_s, \ldots, z_1;w)$ or $A^{<0}_{r_s, \ldots, r_1}(z_s, \ldots, z_1;w)$ are expansions of some rational functions. 

\vspace{.1in}

To unburden the notation we will simply write $A^*_{\underline{r}}(\underline{z},w)$ when the values of the $r_i$ are 
understood and there is no risk of confusion. We will also write, as in Section~1.3,
$$r_{<i}=\sum_{k<i} r_k, \qquad r_{>i}=\sum_{k>i}r_k, \qquad r_{[i,j]}=\sum_{k=i}^j r_k, \qquad etc.$$

\vspace{.2in}

\paragraph{\textbf{5.3.}} As (\ref{E:cohtilde}) and Proposition~\ref{P:4} show, the volume of the moduli spaces $\widetilde{\textbf{Coh}}_{\underline{\alpha}}^{\geq 0}(X)$
are essentially computed by the generating series $A^{\geq 0}_{\underline{r}}(\underline{z};w)$ for suitable choices of $r_s, \ldots. r_1$. In order to determine these series, we will actually calculate all three types of series and use some induction process. We begin with the series $A_{\underline{r}}(\underline{z};w)$, which is easy to compute. 

\vspace{.1in}

\begin{lem}\label{L:52} Assume that $r_i \geq 1$ for all $i$. Then
$$A_{\underline{r}}(\underline{z};w)=q^{\frac{1}{2}(g-1)\sum_{i>j} r_ir_j}\prod_i vol_{r_i} \prod_{i} \bigg\{ \sum_{l \in \Z} z^l_iw^lq^{\frac{1}{2}l(r_{<i}-r_{>i})}\bigg\}.$$
\end{lem}
\begin{proof}
We have
$$A_{\underline{r}}(\underline{z};w)=\bigg( E^{vec}_{r_s}(z_s) \cdots E^{vec}_{r_1}(z_1)\;\big|\; E_{r}(w)\bigg)=\bigg( E^{vec}_{r_s}(z_s) \otimes \cdots \otimes E^{vec}_{r_1}(z_1)\;\big|\; \Delta'_{r_s, \ldots, r_1}(E_r(w)) \bigg).$$
By Lemma~\ref{L:volcoprod}, 
$$\Delta'_{r_s, \ldots, r_1}(E_{r}(w))=q^{\frac{1}{2}(g-1)\sum_{i>j} r_ir_j} E_{r_s}(q^{\frac{1}{2}r_{<s}}w) \otimes
\cdots \otimes E_{r_i}(q^{\frac{1}{2}(r_{<i}-r_{>i})}w) \otimes \cdots \otimes E_{r_1}(q^{-\frac{1}{2}r_{>1}}w).$$
The Lemma follows since by definition $(1^{vec}_{r,d}\;|\; 1_{r,d})=(1^{vec}_{r,d}\;|\; 1^{vec}_{r,d})=vol_r$ for any $r \geq 1$ and any $d$.
\end{proof}

\vspace{.2in}

\paragraph{\textbf{5.4.}} Our next task is to determine explicitly the series $A^{*}_{1, \ldots, 1}(z_s, \ldots, z_1;w)$, which we will simply abbreviate $A^*(\underline{z};w)$ when no confusion is likely. To begin, note that by Lemma~\ref{L:52},
\begin{equation}\label{E:inducarg}
A(\underline{z};w)=q^{\frac{1}{2}(g-1)\frac{s(s-1)}{2}} vol_1^s \sum_{l_1, \ldots, l_s \in \Z} \bigg(z_1^{l_1} \cdots z_s^{l_s}w^{\sum l_i}q^{\sum \frac{1}{2}l_i(2i-s-1)}\bigg).
\end{equation}

\vspace{.1in}

Set $\widetilde{\zeta}(z)=z^{1-g}\zeta(z)$.

\begin{prop}\label{P:31} For any $s \geq 1$ we have
\begin{equation}\label{E:P34}
A^{\geq 0}(\underline{z};w)=\frac{q^{\frac{1}{4}(g-1)s(s-1)}vol_1^s}{\prod_{i<j} \widetilde{\zeta}\big(\frac{z_i}{z_j}\big)} \sum_{\sigma \in \mathfrak{S}_s} \sigma \left[ \prod_{i<j}
\widetilde{\zeta}\big(\frac{z_i}{z_j}\big) \cdot \frac{1}{\prod_{i<s} \big( 1-q\frac{z_{i+1}}{z_i}\big)} \cdot \frac{1}{1-q^{\frac{1-s}{2}}z_1w}\right],
\end{equation}
\begin{equation}\label{E:P35}
A^{<0}(\underline{z};w)=(-1)^{s}\frac{q^{\frac{1}{4}(g-1)s(s-1)}vol_1^s}{\prod_{i<j} \widetilde{\zeta}\big(\frac{z_i}{z_j}\big)} \sum_{\sigma \in \mathfrak{S}_s} \sigma  \left[ \prod_{i<j}
\widetilde{\zeta}\big(\frac{z_i}{z_j}\big) \cdot \frac{1}{\prod_{i<s} \big( 1-q^{-1}\frac{z_{i}}{z_{i+1}}\big)} \cdot \frac{1}{1-q^{\frac{s-1}{2}}z_sw}\right]
\end{equation}
where the rational functions are expanded in the regions $z_1 \gg z_2 \gg \cdots \gg z_s, w \ll 1$ and
$z_1 \gg z_2 \gg \cdots \gg z_s, w \gg 1$ respectively, (i.e. in power series in the $\frac{z_{i+1}}{z_i}$ and $w$, resp. $w^{-1}$).
\end{prop}
\begin{proof} The proof proceeds by induction on $s$, using formulas (\ref{E:><}) and (\ref{E:inducarg}).
When $s=1$ we have $E_s^{vec,\geq 0}(w)=\sum_{d \geq 0}1^{vec}_{1,d}, E_s^{<0}(w)=\sum_{d<0} 1^{vec}_{1,d}w^d$ and hence
$$A^{\geq 0}(z_1;w)=\sum_{d\geq 0} (1^{vec}_{1,d}\;|\;1^{vec}_{1,d}) z_1w= \frac{vol_1}{1-z_1w}, \qquad  
A^{< 0}(z_1;w)=\sum_{d< 0} (1^{vec}_{1,d}\;|\;1^{vec}_{1,d}) z_1w= -\frac{vol_1}{1-z_1w}.$$
Next, fix $s >1$ and assume that the proposition is proved for all $s'<s$. Using (\ref{E:><}) we have
\begin{equation*}
\begin{split}
A(z_s, \ldots, z_1;w)
&=A^{\geq 0}(\underline{z};w) + A^{<0}(\underline{z};w) + \sum_{\substack{u+t=s\\u,t>0}}q^{\frac{1}{2}(g-1)ut} \left( \Delta_{u,t}(E^{vec}_{1^s}(\underline{z}))\;|\; E_u^{<0}(q^{\frac{t}{2}}w)\otimes E_t^{\geq 0}(q^{-\frac{u}{2}}w)\right)
\end{split}
\end{equation*}
Now, from Lemma~	\ref{L:volcoprod}
\begin{equation}\label{E:P31}
\begin{split}
\Delta_{}(E^{vec}_{1^s}(z_s,\ldots, z_1))&=\Delta_{}(E_1^{vec}(z_s)) \cdots \Delta(E_1(z_1))\\
&=\prod_{i}^{\rightarrow}\bigg(E_1^{vec}(z_i) \otimes 1 + E_0(q^{\frac{1}{2}}z_i)E_0(q^{-\frac{1}{2}}z_i)^{-1} \mathbf{k}_1 \otimes E_1^{vec}(z_i)\bigg) 
\end{split}
\end{equation}
Expanding (\ref{E:P31}) yields an expression of $\Delta(E^{vec}_{1^s}(\underline{z}))$ as a sum 
$$\Delta(E^{vec}_{1^s}(\underline{z}))=\sum_{\sigma} X_{\sigma}$$
parametrized by maps $\sigma~: \{1, \ldots, s\} \to \{1,2\}$, with
$$X_{\sigma}=\prod_{i}^{\rightarrow} C_{\sigma(i)}(z_i)$$
where 
$$C_{1}(z)=E_1^{vec}(z) \otimes 1, \qquad C_2(z)=E_0(q^{\frac{1}{2}}z)E_0(q^{-\frac{1}{2}}z)^{-1}\mathbf{k}_1 \otimes E_1^{vec}(z).$$
The component $\Delta_{u,t}(E^{vec}_{1^s}(\underline{z}))$ of $\Delta(E^{vec}_{1^s}(\underline{z}))$
is equal to the same sum, this time ranging over the set of maps $\sigma~:\{1, \ldots, s\} \to \{1,2\}$ such that
$|\sigma^{-1}(1)|=u, |\sigma^{-1}(2)|=t$. We will denote this set of maps $Sh_{u,t}$, for \textit{$(u,t)$-shuffles}.
From Lemma~\ref{L:Hecke} and the defining commutation relations involving $\mathbf{k}_1$ (see Section~\textbf{4.1.}) we derive
$$X_{\sigma}=H_{\sigma}(\underline{z})\left(\prod^{\rightarrow}_{i, \sigma(i)=1} E_1^{vec}(z_i) \prod^{\rightarrow}_{j, \sigma(j)=2} E_0(q^{\frac{1}{2}}z_j)E_0(q^{-\frac{1}{2}}z_j)^{-1} \mathbf{k}_1^{t}\right) \otimes \prod_{j, \sigma(j)=2}E_1^{vec}(z_j)$$
where
$$H_{\sigma}(\underline{z})=\prod_{\substack{(i,j), i >j,\\ \sigma(i)=2, \sigma(j)=1}}q^{(1-g)}\text{Exp}\left( |X(\mathbb{F}_q)|(1-q^{-1})\frac{z_i}{z_j}\right)=\prod_{\substack{(i,j), i >j,\\ \sigma(i)=2, \sigma(j)=1}}\frac{\widetilde{\zeta}\big(\frac{z_i}{z_j}\big)}{\widetilde{\zeta}\big(\frac{z_j}{z_i}\big)}.$$
Putting all the pieces together yields the following recursion formula~:
\begin{equation}\label{E:P33}
A(z_s, \ldots, z_1;w)=A^{\geq 0}(z_s, \ldots, z_1;w) + A^{<0}(z_s, \ldots, z_1;w) + \sum_{\substack{u+t=s\\u,t >0}}  \sum_{\sigma \in Sh_{u,t}} A_{\sigma}(z_s, \ldots, z_1;w)
\end{equation}
with
$$A_{\sigma}(z_s, \ldots, z_1;w)=q^{\frac{1}{2}(g-1)ut}H_{\sigma}(\underline{z}) A^{<0}(z_{i_u}, \ldots, z_{i_1};q^{\frac{t}{2}}w) A^{\geq 0}(z_{j_t}, \ldots, z_{j_1}; q^{-\frac{u}{2}}w)$$
where $(i_u, \ldots, i_1)$ (resp. $(j_t, \ldots, j_1)$ are the reordering (in decreasing order) of the sets $\sigma^{-1}(1)$ (resp. $\sigma^{-1}(2)$).

\vspace{.05in}

Equation (\ref{E:P33}) takes place in the vector space $K_g[[z_s^{\pm 1},\ldots z_1^{\pm 1}, w^{\pm 1}]]$.
From (\ref{E:P33}) we may inductively derive $A^{\geq 0}(z_s, \ldots, z_1;w)$ and $A^{<0}(z_s, \ldots, z_1;w)$ from
$A(z_s, \ldots, z_1;w)$ and the values of the series $A^{<0}(z_u, \ldots, z_1;w)$ and $A^{\geq 0}(z_t, \ldots, z_1;w)$ 
for $u,t < s$. Note that $A^{\geq 0}(z_s, \ldots, z_1;w)$ is a power series in $w$ while $A^{<0}(z_s, \ldots, z_1;w)$ is a power series in $w^{-1}$. In order to establish the statement of Proposition~\ref{P:3} for $s$, it therefore suffices to
show that (\ref{E:P33}) holds with $A^{\geq 0}(z_s, \ldots, z_1;w)$ and $A^{<0}(z_s, \ldots, z_1;w)$ respectively given by (\ref{E:P34}) and (\ref{E:P35}).

\vspace{.1in}

Using the induction hypothesis, and after a little arithmetic, we arrive at
\begin{equation}\label{E:inter}
\begin{split}
A_{\sigma}(z_s, \ldots, z_1;w)=&Z_{\sigma}(z_s, \ldots, z_1) \cdot \sum_{\sigma_1 \in \mathfrak{S}_u} \sigma_1  \left[ \prod_{l < h \leq u}\widetilde{\zeta}\big( \frac{z_{i_l}}{z_{i_h}}\big)\cdot \frac{1}{\prod_{l <u} (1-q^{-1} \frac{z_{i_l}}{z_{i_{l+1}}})} \sum_{n_- < 0} \big(q^{\frac{s-1}{2}} w z_{i_u}\big)^{n_-} \right] \\
& \qquad \cdot  \sum_{\sigma_2 \in \mathfrak{S}_t} \sigma_2  \left[ \prod_{k < m \leq t}\widetilde{\zeta}\big( \frac{z_{j_k}}{z_{j_m}}\big)\cdot \frac{1}{\prod_{k <t} (1-q \frac{z_{j_{k+1}}}{z_{j_{k}}})} \sum_{n_+ \geq 0} \big(q^{\frac{1-s}{2}} w z_{j_1}\big)^{n_+} \right]
\end{split}
\end{equation}

where

$$Z_{\sigma}(z_s, \ldots, z_1)= (-1)^{u-1}q^{\frac{1}{4}(g-1)s(s-1)}vol_1^s\prod_{\substack{(i,j), i >j,\\ \sigma(i)=2, \sigma(j)=1}}\frac{\widetilde{\zeta}\big(\frac{z_i}{z_j}\big)}{\widetilde{\zeta}\big(\frac{z_j}{z_i}\big)}
\cdot \prod_{l < h \leq u} \frac{1}{\widetilde{\zeta}\big( \frac{z_{i_l}}{z_{i_h}}\big)}\cdot \prod_{k < m \leq t} \frac{1}{\widetilde{\zeta}\big( \frac{z_{j_k}}{z_{j_m}}\big)}.$$

Consider the Fourier coefficients 
$$A^{\geq 0}(\underline{z};w)=\sum_{n \geq 0} a^{\geq 0}_n(\uz)w^n, \qquad A^{<0}(\uz;w)=\sum_{n<0 } a^{<0}_n(\uz)w^n, \qquad A_{\sigma}(\uz;w)=\sum_n a_{\sigma,n}(\uz) w^n.$$

Note that $a_n^{\geq 0}(\uz)$ is zero for $n<0$ while $a_n^{<0}(\uz)$ is zero when $n \geq 0$.
Observe also that $a_{\sigma,n}(\uz)$ belongs to the subspace of $K_g[[z_s^{\pm 1}, \ldots, z_1^{\pm 1}]]$ of formal series converging in the asymptotic region
$$U_{\sigma}:=\{(z_s, \ldots, z_1)\;|\; z_{i_1} \gg z_{i_2} \gg \cdots \gg z_{i_u} \gg z_{j_1} \gg \cdots \gg z_{j_t}\}$$
while $a_n^{<0}(\uz)$ and $a_n^{\geq 0}(\uz)$ both belong to the subspace of
$K_g[[z_s^{\pm 1}, \ldots, z_1^{\pm 1}]]$ of formal series converging in the asymptotic region
$$U_{1}:=\{(z_s, \ldots, z_1)\;|\; z_{1} \gg z_{2} \gg \cdots \gg z_{s} \}.$$

The part of (\ref{E:P33}) in which $w$ appears with the exponent $n$ reads
\begin{equation}\label{E:P345}
q^{\frac{1}{2}(g-1)\frac{s(s-1)}{2}} vol_1^s \sum_{\substack{l_1, \ldots, l_s \in \Z,\\ \sum_i l_i=n}} z_1^{l_1} \cdots z_s^{l_s}q^{\sum \frac{1}{2}l_i(2i-s-1)}=a_n^{\geq 0}(\underline{z}) + a_n^{<0}(\underline{z}) + \sum_{u, \sigma} a_{\sigma,n}(\underline{z}).
\end{equation}

The rhs of (\ref{E:P345}) is a sum of Laurent series, each of which is the expansion in a certain asymptotic region $U_{\sigma}$ of some rational function. The lhs of (\ref{E:P345}) may likewise be written as a sum of Laurent series,
each of which is the expansion in an asymptotic direction of some rational function. Moreover, the sum of all the rational functions occuring on the lhs is zero (indeed, the lhs of (\ref{E:P345}) is, up to a global factor, a product of delta functions). 

\vspace{.1in}

\begin{lem}\label{L:obvious} Let $K_g[z_s^{\pm 1}, \ldots, z_1^{\pm 1}]_{loc}$ be the localization of $K_g[z_s^{\pm 1}, \ldots, z_1^{\pm 1}]$ at the set of linear polynomials $z_i-cz_j$ for $c \in K_g$. For any $\gamma \in \mathfrak{S}_s$ let
$$\tau_{\gamma}~:K_g[z_s^{\pm 1}, \ldots, z_1^{\pm 1}]_{loc} \hookrightarrow K_g[[z_{s}^{\pm 1}, \ldots, z_1^{\pm 1}]]$$
be the expansion map in the region
$$U_{\gamma}=\{(z_s, \ldots, z_1)\;|\; z_{\gamma(1)} \gg z_{\gamma(2)} \gg \cdots \gg z_{\gamma(s)} \}.$$
Assume given a collection of elements $f_{\gamma} \in K_g[z_s^{\pm 1}, \ldots, z_1^{\pm 1}]_{loc}$ satisfying
$\sum_{\gamma} \tau_{\gamma}(f_{\gamma})=0$. Then $\sum_{\gamma} f_{\gamma}=0$.
\end{lem}
\begin{proof} Write $f_{\gamma}=R_{\gamma}/Q_{\gamma}$ with $R_{\gamma}, Q_{\gamma} \in K_g[z_s^{\pm 1}, \ldots, z_s^{\pm 1}]$. Let $Q=\prod_{\gamma}Q_{\gamma}$. Then $0=Q\sum_{\gamma} \tau_{\gamma}(f_{\gamma})=\sum_{\gamma} \tau_{\gamma}(Qf_{\gamma})=\sum_{\gamma}Q f_{\gamma}$, since $\tau_{\gamma}$ is a morphism of $K_g[z_s^{\pm 1}, \ldots, z_1^{\pm 1}]$-modules. Hence $\sum_{\gamma} f_{\gamma}=0$.
\end{proof}

\vspace{.1in}

Denote by $\textbf{a}_n^{\geq 0}(\underline{z}), \mathbf{a}_n^{<0}(\underline{z})$ and $\mathbf{a}_{\sigma,n}(\underline{z})$ the rational functions of which ${a}_n^{\geq 0}(\underline{z}), {a}_n^{<0}(\underline{z})$ and ${a}_{\sigma,n}(\underline{z})$ are the expansion (each in its respective region).
By lemma~\ref{L:obvious}, we have
$$\mathbf{a}_n^{\geq 0}(\underline{z})=-\sum_{\sigma} \mathbf{a}_{\sigma,n}(\uz), \qquad (n \geq 0),$$
 $$\mathbf{a}_n^{< 0}(\underline{z})=-\sum_{\sigma} \mathbf{a}_{\sigma,n}(\uz), \qquad (n < 0).$$

Assume $n \geq 0$. A short computation using (\ref{E:inter}) gives
$$\textbf{a}_{\sigma,n}(z_s, \ldots, z_1)= -Z_{\sigma}(z_s, \ldots, z_1)q^{\frac{1-s}{2}n}\sum_{\sigma_1, \sigma_2} \sigma_1 \boxtimes \sigma_2 \left[ \frac{\prod_{l<h \leq u}\widetilde{\zeta}\big( \frac{z_{i_l}}{z_{i_h}}\big) \cdot \prod_{k <m\leq t}\widetilde{\zeta}\big(\frac{z_{j_k}}{z_{j_m}}\big)}{\prod_{l < u} (1-q^{-1}\frac{z_{i_l}}{z_{i_{l+1}}})\cdot  \prod_{k<t} (1-q \frac{z_{j_{k+1}}}{z_{j_k}})} \cdot \frac{z_{j_1}^{n}}{1-q^{s-1}\frac{z_{i_u}}{z_{j_1}}}\right].$$

Summing over $u$ and $\sigma$, we get
\begin{equation*}
\textbf{a}^{\geq 0}_n(\underline{z}):=-\sum_{u, \sigma} \textbf{a}_{\sigma,n}(\underline{z})=vol_1^s \prod_{i>j} \frac{1}{\widetilde{\zeta}\big(\frac{z_j}{z_i}\big)} \cdot q^{\frac{1}{4}(g-1)s(s-1) + \frac{1-s}{2}n} \cdot \sum_{\sigma \in \mathfrak{S}_s} \sigma \left[ \prod_{i>j} \widetilde{\zeta}\big( \frac{z_j}{z_i}\big) \cdot z_1^n \cdot \sum_{u=1}^{s-1} T_{u}(\underline{z})\right]
\end{equation*}
where
$$T_u(\underline{z})=(-1)^{u-1}\frac{1}{\prod_{l < t} (1-q\frac{z_{l+1}}{z_l}) \cdot \prod_{t < k < s} (1-q^{-1} \frac{z_{k}}{z_{k+1}}) \cdot (1-q^{s-1}\frac{z_s}{z_1})}.$$
Now,
\begin{equation*}
\begin{split}
\sum_{u=1}^{s-1}(-1)^uT_u(\underline{z})&=\frac{1}{\prod_{l<s}(1-q\frac{z_{l+1}}{z_l}) (1-q^{s-1}\frac{z_s}{z_1})} \left\{ (1-q\frac{z_s}{z_{s-1}}) + q\frac{z_s}{z_{s-1}}(1-q \frac{z_{s-1}}{z_{s-2}}) + \cdots + q^{s-2}\frac{z_s}{z_2}(1-q\frac{z_2}{z_1})\right\}\\
&=\frac{1}{\prod_{l<s}(1-q\frac{z_{l+1}}{z_l}) (1-q^{s-1}\frac{z_s}{z_1})} \big(1-q^{s-1}\frac{z_s}{z_1}\big)\\
&=\frac{1}{\prod_{l<s}(1-q\frac{z_{l+1}}{z_l})}
\end{split}
\end{equation*}

Summing now over $n \geq 0$ we obtain

$$\sum_{n \geq 0} \textbf{a}^{\geq 0}_n(\uz)w^n=\frac{q^{\frac{1}{4}(g-1)s(s-1)}vol_1^s}{\prod_{i<j} \widetilde{\zeta}\big(\frac{z_i}{z_j}\big)} \sum_{\sigma \in \mathfrak{S}_s} \sigma \left[ \prod_{i<j}
\widetilde{\zeta}\big(\frac{z_i}{z_j}\big) \cdot \frac{1}{\prod_{i<s} \big( 1-q\frac{z_{i+1}}{z_i}\big)} \cdot \frac{1}{1-q^{\frac{1-s}{2}}z_1w}\right]$$
as wanted. This shows (\ref{E:P34}) for $s$. The computations of $\textbf{a}_{\sigma,n}(\uz)$ and $\textbf{a}^{<0}_n(\uz)$ for $n <0$ and hence the proof of (\ref{E:P35}) for $s$ are entirely similar. Proposition~\ref{P:31} is proved.
\end{proof}

\vspace{.2in}

\paragraph{\textbf{5.5.}} Proposition~\ref{P:31} allows us to compute the value of $A^{\geq 0}_{\underline{r}}(\underline{z};w)$ (and hence also $G^{\geq 0}_{\underline{r}}(\underline{z};w)$) for an arbitrary sequence of nonnegative integers $\underline{r}=(r_i)$ by considering appropriate residues. Namely, by Lemma~ \ref{L:residue}, we have

\begin{equation}\label{E:arzy0}
A^{\geq 0}_{\underline{r}}(q^{\frac{1}{2}(1-r_t)}z^{(t)}_1, \ldots,q^{\frac{1}{2}(1-r_1)} z^{(1)}_1;w)=q^{a(\underline{r})} vol_1^{-n} \prod_i vol_{r_i}\cdot \text{Res}_{\underline{r}}\left[A^{\geq 0}_{(1^n)}(z^{(t)}_{r_t}, z^{(t)}_{r_t-1}, \ldots, z^{(t)}_{1}, z^{(t-1)}_{r_{t-1}}, \ldots, z^{(1)}_{1};w)\right]
\end{equation}
where $\text{Res}_{\underline{r}}=\prod_{i=1}^t \text{Res}^{(i)}$ and $\text{Res}^{(i)}$ is the operator of taking the (iterated) residue along 
$$\frac{z^{(i)}_{r_i}}{z^{(i)}_{r_i-1}}= \frac{z^{(i)}_{r_i-1}}{z^{(i)}_{r_i-2}}= \cdots =\frac{z^{(i)}_{2}}{z^{(i)}_{1}}= q^{-1} $$
and where
$$a(\underline{r})=-\frac{1}{4} (g-1)\sum_i r_i(r_i-1).$$
In an effort to unburden the notation let us rename the variables $(z^{(t)}_{r_t}, \ldots, z^{(1)}_1)$ as $(z_n, z_{n-1}, \ldots, z_1)$. In particular,
$$z^{(i)}_1=z_{1+r_{<i}}, \qquad \forall\; i=1, \ldots, t.$$
 Using Proposition~\ref{P:31} we get
\begin{equation}\label{E:arzy}
\begin{split}
A^{\geq 0}_{\underline{r}}&(q^{-\frac{1}{2}r_t}z^{(t)}_1, \ldots, q^{-\frac{1}{2}r_1}z^{(1)}_1;w)\\
&=q^{b(\underline{r})}\prod_ivol_{r_i} \cdot \text{Res}_{\underline{r}}\left[ \frac{1}{\prod_{i<j} \widetilde{\zeta}\big(\frac{z_i}{z_j}\big)} \sum_{\sigma \in \mathfrak{S}_n} \sigma \left\{ \prod_{i<j}
\widetilde{\zeta}\big(\frac{z_i}{z_j}\big) \cdot \frac{1}{\prod_{i<n} \big( 1-q\frac{z_{i+1}}{z_i}\big)} \cdot \frac{1}{1-q^{-\frac{n}{2}}z_1w}\right\}\right],
\end{split}
\end{equation}
where
$$b(\underline{r})=\frac{1}{2}(g-1)\sum_{i<j} r_ir_j.$$

Of course taking appropriate residues in (\ref{E:P35}) yields similar formulas for $A^{<0}_{\underline{r}}(q^{-\frac{1}{2}r_t}z^{(t)}_1, \ldots, q^{-\frac{1}{2}r_1}z^{(1)}_1;w)$. 

%

\vspace{.2in}

\paragraph{\textbf{5.6.}} Fix some $r \geq 0$. By a \textit{generic Jordan type} of weight $r$ we will mean a finite (possibly empty) sequence $\underline{r}=(r_1, \ldots, r_t)$ of nonnegative integers such that $\sum_i i r_i =r$ and $r_t \neq 0$. Observe that the assignement
$$(r_1, \ldots, r_t) \mapsto (1^{r_1} 2^{r_2} \ldots t^{r_t})$$
sets up a bijection between the set $J_{gen}(r)$ of generic Jordan types of weight $r$ and the set of partitions of $r$. 
A \textit{Jordan type} of weight $(r,d)$ is a sequence $\underline{\alpha}=(\alpha_1, \ldots, \alpha_s)$ such that
$\sum_i i \alpha_i =(r,d)$ and $\alpha_s \neq 0$. We denote by $J(r,d)$ the set of Jordan types of weight $(r,d)$.
There is a natural forgetful map 
$$\pi:~\bigsqcup_d J(r,d) \to J_{gen}(r).$$
Let us fix a generic Jordan type $\underline{r}=(r_1, \ldots, r_t)$ of weight $r \geq 0$.
We will now compute the sum
$$\Xi_{\underline{r}}(z)=\sum_{\underline{\alpha} \in \pi^{-1}(\underline{r})} vol(\textbf{Nil}^{\geq 0}_{\underline{\alpha}}) z^{\sum i \alpha_i}=\sum_{\underline{\alpha} \in \pi^{-1}(\underline{r})}q^{d'(\underline{\alpha})-\frac{1}{2}\sum_{i > j} \langle \alpha_i,\alpha_j\rangle}\left( 1_{\alpha_s} \cdots 1_{\alpha_1}\;|\; 1_{\sum \alpha_i}^{\geq 0}\right) z^{\sum i \alpha_i}.$$
Let us fix some $s \geq t$ and let $\Xi_{\underline{r}}^{s}(z)$ the restriction of the sum to the subset of Jordan types $\underline{\alpha}=(\alpha_1, \ldots, \alpha_{s'})$ in
$\pi^{-1}(\underline{r})$ for which $s' \leq s$. To unburden the notation, we set 
$$T=z^{(1,0)}, \qquad z=z^{(0,1)}, \qquad n=\sum r_i.$$
We obtain
\begin{equation*}
\begin{split}
\Xi^{s}_{\underline{r}}(z)&=q^{e(\underline{r})} \sum_{d_1, \ldots, d_{s} \in \mathbb{Z}} q^{\frac{1}{2} \sum_i d_i(r_{>i}-r_{<i})} \left(1_{r_s,d_s}z^{sd_s} \cdots 1_{r_1,d_1}z^{d_1}\;|\; 1^{\geq 0}_{n,\sum d_i}\right) T^{r}\\
&=q^{e(\underline{r})} G^{\geq 0}_{0^{s-t}, r_t, \ldots, r_1}(y_s, \ldots, y_1;1)T^{r}
\end{split}
\end{equation*}
where 
$$e(\underline{r})=(g-1) \left[\sum_i (i-1) r_i^2 + \sum_{i < j} (2i - \frac{1}{2}) r_ir_j\right]$$
and $y_l=z^l q^{\frac{1}{2} ( r_{>l}-r_{<l})}$ for $l=1, \ldots, s$.
Using Proposition~\ref{P:4}, we have
\begin{equation}\label{E:equas}
\begin{split}
\Xi_{\underline{r}}^s(z)&=q^{e(\underline{r})}\text{Exp}\left( \frac{|X(\mathbb{F}_q)|}{q-1} \left[ \sum_{i=1}^s q^{-\frac{1}{2}(n+r_i)}y_i + \sum_{i>j} \frac{y_i}{y_j} (q^{\frac{r_j}{2}}-q^{-\frac{r_j}{2}})q^{-\frac{r_i}{2}}\right]\right)A^{\geq 0}_{r_t,\ldots, r_1}(y_t, \ldots, y_1;1)\\
\end{split}
\end{equation}
since, clearly, $A^{\geq 0}_{0^{s-t},r_t,\ldots, r_1}(y_s, \ldots, y_1;1)=A^{\geq 0}_{r_t,\ldots, r_1}(y_t, \ldots, y_1;1)$.
The number of isomorphism classes of pairs $(\mathcal{F},\theta) \in \textbf{Nil}^{\geq 0}_{\alpha}$ being finite for any fixed $\alpha=(r,d)$, the orbifold $\textbf{Nil}_{\underline{\alpha}}^{\geq 0}$ is empty for almost all $\underline{\alpha}$.
In particular, the coefficient of $T^rz^d$ in $\Xi^s_{\underline{r}}$ stabilizes for any fixed $d$ as $s$ tends to infinity. Taking the limit $s \to \infty$ in (\ref{E:equas}) we get after some straightforward computation
\begin{equation}\label{E:513}\Xi_{\underline{r}}(z)=q^{e(\underline{r})}\text{Exp}\left(  \frac{|X(\mathbb{F}_q)|}{q-1} \left[ \sum_{i,l \geq 1}q^{-r_{[i,i+l]}}(q^{r_{i+l}}-1)z^l  + \frac{z}{1-z}\right]\right)A^{\geq 0}_{r_t,\ldots, r_1}(y_t, \ldots, y_1;1)
\end{equation}
Let $\lambda=(1^{r_1} 2^{r_2} \cdots t^{r_t})$ be the partition associated to $\underline{r}$. Let $\lambda^{\circ}$ denote the the set of boxes $s \in \lambda$ satisfying $a(s) >0$.
\begin{lem} We have
$$\text{Exp}\left(  \frac{|X(\mathbb{F}_q)|}{q-1} \left[ \sum_{i,l \geq 1}q^{-r_{[i,i+l]}}(q^{r_{i+l}}-1)z^l  \right]\right)=\prod_{s \in \lambda^{\circ}} \zeta_X(q^{-1-l(s)}z^{a(s)}).$$
\end{lem}
\begin{proof} A direct verification using the formula $Exp(|X(\mathbb{F}_q)|q^{-u}z^v)=\zeta_X(q^{-u}z^v)$.\end{proof}
Using (\ref{E:arzy}), (\ref{E:513}), the above Lemma and Theorem~\ref{T:volform} i) we arrive at the following expression
\begin{equation}
\Xi_{\underline{r}}(z)=q^{(g-1)\langle \lambda, \lambda \rangle }\cdot \prod_{s \in \lambda} \zeta^*_X(q^{-1-l(s)}z^{a(s)})\cdot H_{\underline{r}}(z)\cdot\text{Exp}\left(  \frac{|X(\mathbb{F}_q)|}{q-1} \cdot \frac{z}{1-z}\right)
\end{equation}
where
$$\langle \lambda, \lambda \rangle=\sum_k (\lambda')_k^2=\sum_i i r_i^2 + \sum_{i<j} 2i r_ir_j,$$
$$\zeta^*_X(q^{-u}z^v)=\begin{cases}\zeta_X(q^{-u}z^v)\qquad & \text{if}\; (u,v) \neq (1,0) \\ q^{-1}vol_1 \qquad & \text{if}\; (u,v)=(1,0) \end{cases}$$ 
and
$$H_{\underline{r}}(z)=\widetilde{H}_{\underline{r}}(z^tq^{-r_{<t}}, \ldots, z^iq^{-r_{<i}}, \ldots, z),$$
$$\widetilde{H}_{\underline{r}}(z_{1+ r_{<t}}, \ldots, z_{1+r_{<i}}, \ldots, z_1)=\text{Res}_{\underline{r}}\left[ \frac{1}{\prod_{i<j} \widetilde{\zeta}\big(\frac{z_i}{z_j}\big)} \sum_{\sigma \in \mathfrak{S}_n} \sigma \left\{ \prod_{i<j}
\widetilde{\zeta}\big(\frac{z_i}{z_j}\big) \cdot \frac{1}{\prod_{i<n} \big( 1-q\frac{z_{i+1}}{z_i}\big)} \cdot 
\frac{1}{1-z_1}\right\}\right].$$

\vspace{.2in}

\paragraph{\textbf{5.7}}  Taking the sum over all generic Jordan types $\underline{r}$, using Proposition~\ref{P:1}
and Corollary~\ref{C:25} and setting in accordance with Section~1.3.
$$ J_{\lambda}(z)=\prod_{s \in \lambda} \zeta^*_X(q^{-1-l(s)}z^{a(s)}), \qquad H_{\lambda}(z)=H_{\underline{r}}(z)$$
when $\lambda=(1^{r_1} 2^{r_2} \cdots)$,
we get the following complicated but nevertheless explicit generating formula for the numbers $A^{\geq 0}_{r,d}$~:
\begin{equation}\label{E:571}
\text{exp}\left(\sum_{l \geq 1} \frac{1}{l} \sum_{r, d} \frac{A^{\geq 0}_{r,d}(X \otimes_{\mathbb{F}_q} \mathbb{F}_{q^l})}{q^l-1}z^{ld}T^{lr} \right)=\sum_{\lambda}\left\{q^{(g-1)\langle \lambda, \lambda \rangle }J_{\lambda}(z) H_{\lambda}(z)T^{|\lambda|}\right\}\cdot\text{Exp}\left(  \frac{|X(\mathbb{F}_{q})|}{q-1} \cdot \frac{z}{1-z}\right).
\end{equation}
In the above, all the rational functions in $z$ are expanded in the region $z \ll 1$, i.e. in $\mathbb{C}[[z]]$.  Observe that
$A^{\geq 0}_{0,d}(X\otimes_{\mathbb{F}_q} \mathbb{F}_{q^l})=|X(\mathbb{F}_{q^l})|$ since a geometrically indecomposable torsion sheaf on $X\otimes_{\mathbb{F}_q} \mathbb{F}_{q^l}$ is the indecomposable $d$fold self extension of the structure sheaf of a rational point in $X(\mathbb{F}_{q^l})$. It follows
that
$$\text{exp}\left(\sum_{l \geq 1} \frac{1}{l} \sum_{d} \frac{A^{\geq 0}_{0,d}(X \otimes_{\mathbb{F}_q} \mathbb{F}_{q^l})}{q^l-1}z^{ld} \right)=\text{Exp}\left(  \frac{|X(\mathbb{F}_q)|}{q-1} \cdot \frac{z}{1-z}\right)$$
and (\ref{E:571}) simplifies to
\begin{equation}\label{E:572}
\text{exp}\left(\sum_{l \geq 1} \frac{1}{l} \sum_{r>0, d} \frac{A^{\geq 0}_{r,d}(X \otimes_{\mathbb{F}_q} \mathbb{F}_{q^l})}{q^l-1}z^{ld}T^{lr} \right)=\sum_{\lambda}\left\{q^{(g-1)\langle \lambda, \lambda \rangle }J_{\lambda}(z) H_{\lambda}(z)T^{|\lambda|}\right\}.
\end{equation}

\vspace{.1in}

Recall from Section~\textbf{4.6.} that the elements $A_{g,r,d}^{\geq 0} \in K_g$ defined by (\ref{E:thm1}) are uniquely characterized by the property that $A_{g,r,d}^{\geq 0}(\sigma_X)=A_{r,d}(X)$ for all smooth projective curves $X$. As a consequence we have the following equality in $K_g[[T,z]]$~:
\begin{equation}\label{E:573}
\text{Exp}\left( \sum_{r>0, d} \frac{A^{\geq 0}_{g,r,d}}{q-1}z^{d}T^{r} \right)=\sum_{\lambda}\left\{q^{(g-1)\langle \lambda, \lambda \rangle }J_{\lambda}(z) H_{\lambda}(z)T^{|\lambda|}\right\}.
\end{equation}

\vspace{.2in}

\paragraph{\textbf{5.8.}} Tensoring by a line bundle of degree one induces a bijection between the set of geometrically indecomposable vector bundles on a curve $X$ of rank $r$ and degrees $d$ and $d+r$ respectively. Therefore $A_{r,d}(X)$ and hence $A_{g,r,d}$ only depend on the class of $d$ in $\Z/r\Z$.
By Proposition~\ref{P:2} the integers $A_{g,r,d}^{\geq 0}(\sigma_X)$ are
eventually periodic in $d$ as $d\to \infty$, with period $r$. Thus so are the
$A^{\geq 0}_{g,r,d}$. This means that if we consider the generating function
$$A^{\geq 0}_{g,r}(z)=\sum_{d \geq 0} A_{g,r,d}^{\geq 0} z^d$$
then we have
\begin{equation}\label{E:581}
A^{\geq 0}_{g,r}(z)=P_{g,r}(z) + \sum_{d=0}^{r-1} \frac{A_{g,r,d} z^d}{1-z^r}
\end{equation}
for some polynomial $P_{g,r}(z) \in K_g[z]$. As a consequence of (\ref{E:581}), the polynomials $A_{g,r,d}$ are expressed as
$$A_{g,r,d}=-\sum_{l \in \Z/r\Z} \text{Res}_{z=\xi^l} \left( A_{g,r}(z)\frac{dz}{z}\right) \xi^{-ld}$$
for $\xi$ a primitive $r$th root of unity. This concludes the proof of Theorem~\ref{T:2}.

\vspace{.2in}

\section{Relation to the number of points of Hitchin moduli spaces}

\vspace{.1in}

\paragraph{\textbf{6.1.}} In this section, we relate the number of indecomposable vector bundles of a given class $\alpha$ to the number of stable Higgs bundles of the same class, under the assumption that the characteristic $p$ of the field is large enough (with an explicit bound, depending on the genus $g$ of $X$ and the class $\a$). Our method is directly inspired by that of Crawley-Boevey, Van den Bergh \cite{CBV} and Nakajima (appendix to \textit{loc.\,cit.}) in the context of moduli spaces of representations of quivers, and hinges on the construction of a smooth
deformation $\mathcal{Y} \to \mathbb{A}^1$ of the moduli space of stable Higgs bundles
$$
\xymatrix{
Higgs^{st}_{r,d} \ar[r] \ar[d] & \mathcal{Y} \ar[d] & \mathcal{Y}' \ar[l] \ar[d] \\
\{0\} \ar[r] & \mathbb{A}^1 & \mathbb{A}^1 \backslash \{0\}\ar[l] }
$$
preserving the number of $\mathbb{F}_q$-rational points and
equipped with a projection map $p: \mathcal{Y} \to Bun_{r,d}$ whose restriction to any fiber $\mathcal{Y}_t$ with $t \neq 0$ is a fibration over the constructible substack $Indec_{r,d} \subset Bun_{r,d}$ of indecomposable vector bundles.
We note that the construction of $\mathcal{Y}$ itself appears slightly non-canonical as it involves an explicit local presentation of the stack $Higgs_{r,d}^{st}$ in terms of quot schemes. In doing so, we borrow some techniques developped in \cite{ACK}.

\vspace{.2in}

\paragraph{\textbf{6.2.}} Let us fix a smooth projective, geometrically connected curve $X$ of genus $g$ defined over $k=\mathbb{F}_q$, and let $\Omega_X$ be the canonical line bundle of $X$.
A Higgs sheaf of rank $r$ and degree $d$ is a pair $(\mathcal{V},\theta)$ with $\mathcal{V}$ a coherent sheaf of rank $r$ and degree $d$ and $\theta \in \Hom(\mathcal{V}, \mathcal{V} \otimes \Omega_X)$. A Higgs subsheaf
of $(\mathcal{V},\theta)$ is by definition a subsheaf $\mathcal{W} \subseteq \mathcal{V}$ such that $\theta(\mathcal{W}) \subseteq \mathcal{W} \otimes \Omega_X$. A Higgs sheaf $(\mathcal{V},\theta)$ is called semistable (resp. stable) if for any proper Higgs subsheaf $\mathcal{W} \subset \mathcal{V}$ we have $\mu(\mathcal{W}) \leq \mu(\mathcal{V})$ (resp. $\mu(\mathcal{W}) < \mu(\mathcal{V})$). A Higg subsheaf $\mathcal{W} \subset \mathcal{V}$ satisfying $\mu(\mathcal{W}) > \mu(\mathcal{V})$ is called destabilizing. It is clear that as soon as $r>0$, a semistable Higgs sheaf $(\mathcal{V},\theta)$ is necessarily a Higgs bundle, ie. $\mathcal{V}$ is a vector bundle.

\vspace{.1in}

Let ${Higgs}_{r,d}(X)$ and ${Coh}_{r,d}(X)$ respectively stand for the moduli stacks of Higgs sheaves and coherent sheaves over $X$ of rank $r$ and degree $d$. These are algebraic stacks defined over $k$, locally of finite type, of respective dimensions $2(g-1)r^2$ and $(g-1)r^2$. If $(r,d)$ are coprime, we let $Higgs^{st}_{r,d}(X)$ be the open substack of $Higgs_{r,d}(X)$ parametrizing stable Higgs bundles. The stack $Higgs^{st}_{r,d}(X)$ is represented by a smooth, connected scheme over $k$.

\vspace{.1in}

Serre duality provides a canonical isomorphism $\Ext^{1}(\mathcal{V},\mathcal{W})^* \simeq \Hom(\mathcal{W}, \mathcal{V} \otimes \Omega_X)$ for any pair of coherent sheaves $(\mathcal{V},\mathcal{W})$. Hence, the moduli
stack $Higgs_{r,d}(X)$ may alternatively be defined as the stack parametrizing pairs $(\mathcal{V}, \nu)$ with $\mathcal{V}$ a coherent sheaf over $X$ of rank $r$ and degree $d$ and
$\nu \in \Ext^1(\mathcal{V},\mathcal{V})^*$. A Higgs subsheaf of such a pair $(\mathcal{V},\nu)$ is a subsheaf $\mathcal{W} \subseteq \mathcal{V}$ satisfying the following condition~:
\begin{equation}\label{E:higgsext}
a(\nu) \in b(\Ext^1(\mathcal{W},\mathcal{W})^*)
\end{equation}
where $a,b$ are the canonical maps in the sequence
$$\xymatrix{ \Ext^1(\mathcal{V},\mathcal{V})^* \ar[r]^-{a} & \Ext^1(\mathcal{V},\mathcal{W})^* & \Ext^1(\mathcal{W},\mathcal{W})^* \ar[l]_-{b}}.$$
The stack $Higgs^{st}_{r,d}(X)$ thus parametrizes pairs $(\mathcal{V},\nu)$ as above such that any proper Higgs subsheaf $\mathcal{W} \subset \mathcal{V}$ verifies $\mu(\mathcal{W}) < \mu(\mathcal{V})=\frac{d}{r}$.

\vspace{.2in}

\paragraph{\textbf{6.3.}} In this section we recall the definition and basic properties of quot schemes. These will be used in the next section to make explicit the construction of the stacks $Coh_{r,d}(X)$ and $Higgs_{r,d}(X)$. 

\vspace{.1in}

We say that a vector bundle $\mathcal{F}$ is strongly generated by another vector bundle $\mathcal{G}$ if $\Ext^1(\mathcal{G},\mathcal{F})=0$ and the canonical map $\mathcal{G} \otimes \Hom(\mathcal{G},\mathcal{F}) \to \mathcal{F}$ is surjective. By definition, if $\mathcal{F}$ is strongly generated by $\mathcal{G}$ then $\dim(\Hom(\mathcal{G},\mathcal{F}))=\langle \mathcal{G},\mathcal{F}\rangle$.

\vspace{.1in}

Given a vector bundle $\mathcal{V}$ over $X$ and a pair $\alpha=(r,d)$, the quot scheme $Quot(\mathcal{V},\alpha)$ is the $k$-scheme representing the functor ${quot}_{\mathcal{V},\alpha} : (\text{Aff}/k) \to \text{Sets}$
which assigns to an affine $k$-scheme $S$ the set of equivalence classes of epimorphisms
$$\phi_S : \mathcal{V} \boxtimes \mathcal{O}_S \twoheadrightarrow \mathcal{F}$$
where $\mathcal{F}$ is an $S$-flat coherent sheaf over $X \times S$ such that for any closed point $s \in S$ the sheaf
$\mathcal{F}_{|s}$ over $X$ is of rank $r$ and degree $d$. Here, two epimorphisms $\phi_S, \phi'_S$ are equivalent if
$\text{Ker}(\phi_S)=\text{Ker}(\phi'_S)$. The quot scheme $Quot(\mathcal{V},\alpha)$ is a (generally singular) projective scheme. 
The tangent space to $Quot(\mathcal{V},\a)$ at a point $\phi : \mathcal{V} \twoheadrightarrow \mathcal{F}$ is equal to
$\Hom(\text{Ker}(\phi), \mathcal{F})$.

\vspace{.1in}

One constructs an explicit closed embedding in a projective variety as follows. There exists a line bundle $\mathcal{L}$ of sufficiently negative degree so that for any $\phi : \mathcal{V} \twoheadrightarrow \mathcal{F}$ with $\mathcal{F}$ of rank $r$ and degree $d$ the sheaf $\text{Ker}(\phi)$ is strongly generated by $\mathcal{L}$. Put 
$$a=\dim(\Hom(\mathcal{L},\mathcal{V}))=\langle \mathcal{L}, \mathcal{V}\rangle, \qquad b= \langle \mathcal{L}, \mathcal{V}-\a\rangle$$
and let $Gr(a,b)$ stand for the Grassmanian of $b$-dimensional subspaces of $k^a$.
Fixing an identification $\Hom(\mathcal{L},\mathcal{V}) \simeq k^a$ we obtain  a map $j~ :Quot(\mathcal{V},\a) \to Gr(a,b)$ by assigning to a point $\phi : \mathcal{V} \twoheadrightarrow \mathcal{F}$ the subspace $\Hom(\mathcal{L}, \Ker(\phi)) \subset \Hom(\mathcal{L},\mathcal{V})$. This is a closed embedding (see e.g. \cite[Thm. 4.4.5.]{Lepotier}). 

\vspace{.2in}

\paragraph{\textbf{6.4.}} Let us fix a class $\alpha=(r,d)$ with $r >0$, and $r,d$ coprime. We will now give a construction of the stacks $Coh_{r,d}$ and $Higgs_{r,d}$, or at least of suitable open subset of these stacks. For reasons that will become clear later, we will use a variant of the standard construction, based on the choice of \textit{two} line bundles instead of one, which we borrow from \cite{ACK}.

\vspace{.1in}

\begin{lem} There exists a pair of line bundles $(\Lc_1,\Lc_2) \in Pic^{-n}(X) \times Pic^{-m}(X)$ such that the following hold~:
\begin{enumerate}
\item[a)] Any semistable Higgs bundle $(\mathcal{V},\theta)$ of class $\a$ is strongly generated by $\Lc_1$ and $\Lc_2$,
\item[b)] Any indecomposable vector bundle $\mathcal{V}$ of class $\a$ is strongly generated by $\Lc_1$ and $\Lc_2$,
\item[c)] For any coherent sheaf $\mathcal{V}$ of class $\a$ and any epimorphism $\phi: \mathcal{L}_1\otimes V \twoheadrightarrow \mathcal{V}$, $Ker(\phi)$ is strongly generated by $\mathcal{L}_2$,
\item[d)] For any unstable Higgs sheaf $(\mathcal{F},\theta)$ of class $\a$  there exists a destabilizing Higgs subsheaf $\mathcal{G} \subset \mathcal{F}$ which is strongly generated by $\Lc_1$ and $\Lc_2$.
\end{enumerate}
\end{lem}
\begin{proof} We first show the existence of a line bundle $\Lc_1$ satisfying a), b) and d). The minimal slope $\mu_{min}(\mathcal{F})$ of a Harder-Narasimhan (HN) factor of an indecomposable vector bundle $\mathcal{F}$ of class $\a$ is bounded below by some constant $\nu$ which only depends on $\a$ (see Proposition~\ref{P:2}). An argument in all points similar shows that the minimal slope $\mu_{min}(\mathcal{F})$ of an HN factor of a semistable Higgs bundle of class $\a$ is likewise bounded below by a constant
$\nu'$ which again only depends on $\a$. Let $(\mathcal{F},\theta)$ be an unstable Higgs sheaf of class $\a$. By definition there exists a semistable Higgs subsheaf $\mathcal{G} \subset \mathcal{F}$ of slope $\mu(\mathcal{G}) > \mu(\alpha)$ and rank $rk(\mathcal{G}) \leq rk(\a)$. Arguing as in Proposition~\ref{P:2} we see that the set of minimal
slopes $\mu_{min}(\mathcal{G})$ of all semistable Higgs sheaves of slope at least $mu(\a)$ and rank at most $rk(\a)$ is bounded below by some constant $\nu'''$ which only depends on $\a$.
For any $\nu \in \mathbb{Q}$ there exists $n \in \mathbb{Z}$ such that any semistable sheaf of slope $\sigma \geq \nu$ is strongly generated by any line bundle of degree $m \leq n$. It suffices to take $n$ as above for $\nu=min\{ \nu',\nu'',\nu'''\}$. This proves the existence of a line bundle $\Lc_1$ satisfying a), b) and d). Let us now fix such a line bundle. 
The set of HN types of sheaves $\mathcal{F}$ of class $\a$ which are generated by $\Lc_1$ is finite, as is the set of HN types of kernels of epimorphisms $\Lc_1 \otimes V \twoheadrightarrow \mathcal{F}$. Therefore there exists $\Lc_2$ such that any such kernel is strongly generated by $\Lc_2$. In particular, $\Lc_1$ is strongly generated by $\Lc_2$ and so is any sheaf strongly generated by $\Lc_1$. We are done.
\end{proof}

\vspace{.1in}

Set
$$l_1=\langle \mathcal{L}_1, \alpha \rangle=(1-g+n)r+d, \qquad  l_2=\langle \mathcal{L}_2, \alpha\rangle=(1-g+m)r+d, \qquad V_i=k^{l_i}, \; i=1,2.$$
Consider the quot schemes
$$Q_{\mathcal{L}_1, \Lc_2}=Quot((\Lc_1 \otimes V_1 )\oplus (\Lc_2 \otimes V_2), \a), \qquad Q_{\Lc_1}=Quot(\Lc_1 \otimes V_1, \a).$$
Points of $Q_{\mathcal{L}_1, \Lc_2}$ correspond to epimorphisms $\phi ~: (\Lc_1 \otimes V_1) \oplus (\Lc_2 \otimes V_2) \twoheadrightarrow \mathcal{F}$; we will usually write $\phi_i=\phi_{|\Lc_i \otimes V_i}$ for $i=1,2$.
 Denote by $Q^{\circ, \circ}_{\Lc_1,\Lc_2}$ the open subscheme of $Q_{\Lc_1,\Lc_2}$ parametrizing epimorphisms
$\phi : (\mathcal{L}_1\otimes V_1)\oplus (\mathcal{L}_2\otimes V_2) \twoheadrightarrow \mathcal{F}$ for which the canonical maps
$$\phi_{i*} : V_i \to \Hom(\mathcal{L}_i, \mathcal{F}), \qquad i=1,2$$ 
are isomorphisms (this implies in particular that $\mathcal{F}$ is strongly generated by $\Lc_1$ and hence by $\Lc_2$). We define $Q_{\Lc_1}^{\circ} \subset Q_{\Lc_1}$ in the same fashion. The schemes $Q^{\circ,\circ}_{\Lc_1,\Lc_2}$ and $Q_{\Lc_1}^{\circ}$ are smooth.
The group 
$G:=GL(V_1) \times GL(V_2)$ naturally acts on $Q_{\Lc_1,\Lc_2}$ and preserves $Q^{\circ,\circ}_{\Lc_1,\Lc_2}$. Similarly, the group $GL(V_1)$ acts on $Q_{\Lc_1}$ and preserves $Q^{\circ}_{\Lc_1}$.
The natural restriction map
$$\left[\phi~: (\mathcal{L}_1\otimes V_1)\oplus (\mathcal{L}_2\otimes V_2) \twoheadrightarrow \mathcal{F} \right]\mapsto \left[ \phi_1 :(\mathcal{L}_1\otimes V_1)\twoheadrightarrow \mathcal{F}\right] $$
is a principal $GL(V_2)$-bundle $Q^{\circ,\circ}_{\Lc_1,\Lc_2} \to Q^{\circ}_{\Lc_1}$. 
 By condition c), the stack quotient
$[Q^{\circ,\circ}_{\Lc_1,\Lc_2}/G]$ (and hence $[Q^{\circ}_{\Lc_1}/GL(V_1)]$) is isomorphic to the open substack $Coh^{>\Lc_1}_{r,d}(X)$ of $Coh_{r,d}(X)$ parametrizing coherent sheaves $\mathcal{V}$ of class $\a$ which are strongly generated by $\Lc_1$ (see e.g. \cite{Lepotier}).

\vspace{.1in}

For later purposes, we introduce the locally closed subscheme $Q^{\circ}_{\Lc_1,\Lc_2}$ of $Q_{\Lc_1,\Lc_2}$ which parametrizes epimorphisms $\phi : (\mathcal{L}_1\otimes V_1)\oplus (\mathcal{L}_2\otimes V_2) \twoheadrightarrow \mathcal{F}$ for which $\phi_{2*}~: V_2 \to \Hom(\Lc_2,\mathcal{F})$ is an isomorphism and for which the restriction of
$\phi$ to $\mathcal{L}_1 \otimes V_1$ is still an epimorphism. There is a natural map $Q^{\circ}_{\Lc_1,\Lc_2} \to Q_{\Lc_1}$ which is a principal $GL(V_2)$-bundle.

\vspace{.1in}

The cotangent space $T^*_{\phi}Q_{\Lc_1,\Lc_2}$ to $Q_{\Lc_1,\Lc_2}$ at a point $\phi : (\mathcal{L}_1 \otimes V_1) \oplus (\mathcal{L}_2 \otimes V_2) \twoheadrightarrow \mathcal{F}$ is identified with $\Hom(\Ker(\phi), \mathcal{F})^*$.
If $\phi \in Q^{}_{\Lc_1,\Lc_2}$ then the restriction of the moment map 
$$\mu ~: T^*Q^{\circ,\circ}_{\Lc_1,\Lc_2} \to \mathfrak{g}^*=\mathfrak{gl}(V_1)^* \times \mathfrak{gl}(V_2)^*$$ to $T^*_{\phi}Q^{}_{\Lc_1,\Lc_2}$ 
is the composition $\mu_{\phi}=\nu_{\phi} \circ \kappa_{\phi}$ of the canonical restriction map
$$\Hom(\Ker(\phi), \mathcal{F})^* \to \Hom( (\mathcal{L}_1 \otimes V_1) \oplus (\mathcal{L}_2 \otimes V_2) ,\mathcal{F})^*$$
arising from the long exact sequence
\begin{equation}\label{E:2}
\xymatrix{0 \ar[r] &
\Ext^1(\mathcal{F},\mathcal{F})^* \ar[r] &  \Hom(\Ker(\phi), \mathcal{F})^* \ar[r]^-{\kappa_{\phi}}& \Hom( (\mathcal{L}_1 \otimes V_1) \oplus (\mathcal{L}_2\otimes V_2) ,\mathcal{F})^* \ar[r]^-{j} & \text{End}(\mathcal{F})^*}
\end{equation}
with the map
$$\nu_{\phi}~:\Hom( (\mathcal{L}_1\otimes V_1) \oplus (\mathcal{L}_2\otimes V_2) ,\mathcal{F})^* =\bigoplus_i \Hom( V_i , \Hom(\Lc_i,\mathcal{F})) \to \bigoplus_i \text{End}(V_i)^*$$
induced by composition with $\phi_i^* \in \Hom(V_i, \Hom(\Lc_i,\mathcal{F}))$.

\vspace{.1in}

The quotient stack $[\mu^{-1}(0)/G]$ is isomorphic to the open substack $Higgs^{>\Lc_1}_{r,d}(X)$ of $Higgs_{r,d}(X)$ parametrizing Higgs bundles $(\mathcal{F},\theta)$ with $\mathcal{F}$ of class $\a$ strongly generated by $\Lc_1$. In particular, by condition a) the stack $[\mu^{-1}(0)/G]$ contains  $Higgs^{st}_{r,d}(X)$ as an open substack.

\vspace{.2in}

\paragraph{\textbf{6.5.}} Following \cite{ACK} we now embed $Q^{\circ,\circ}_{\Lc_1,\Lc_2}$ as a locally closed subvariety of the representation variety of an appropriate Kronecker quiver. Namely, put $h = \text{dim}(\Hom(\mathcal{L}_2,\mathcal{L}_1))=(1-g)+m-n$ and let $Kr$ stand for the quiver with 
vertex set $\{1,2\}$ and $h$ arrows from $1$ to $2$~:
$$\xymatrix{ \underset{\bullet}{1} \ar[r]^-{h} & \underset{\bullet}{2}}.$$
Set $$\mathbb{V}=\Hom(\mathcal{L}_2,\mathcal{L}_1), \qquad E=\Hom(V_1 \otimes \mathbb{V}, V_2).$$
The group $G$ acts on $E$ by conjugation and the quotient stack $[E/G]$ is the moduli stack of representations of $Kr$ of dimension $(l_1,l_2)$. There is a natural map $j: Q^{\circ,\circ}_{\Lc_1,\Lc_2} \to E$ sending the point $\phi : (\mathcal{L}_1\otimes V_1) \oplus (\mathcal{L}_2 \oplus V_2) \twoheadrightarrow \mathcal{F}$ to the induced map
$$\xymatrix{V_1 \otimes \mathbb{V} \ar@{=}[r] & \Hom(\Lc_2,\Lc_1 \otimes V_1) \ar[r]^-{\phi_1} & \Hom(\mathcal{L}_2,\mathcal{F}) \ar[r]^-{\phi_{2*}^{-1} } & V_2}.$$
Condition d) guarantees that this defines an embedding of $Q_{\Lc_1,\Lc_2}^{\circ,\circ}$ in $E$ as a smooth locally closed subvariety. Observe that the embedding $j$ extends to an embedding $Q^{\circ}_{\Lc_1,\Lc_2} \to E$. In fact, 
set $$E^{\circ}=\{u \in E\;|\; \text{Im}(u)=V_2\}.$$
This is a principal $GL(V_2)$-bundle over the Grassmanian $Gr(hl_1,l_2)$. We have the following diagram
\begin{equation}\label{E:651}
\xymatrix{ Q^{\circ, \circ}_{\Lc_1,\Lc_2} \ar[r] &Q^{\circ}_{\Lc_1,\Lc_2} \ar[r]^-{j} \ar[d] & E^{\circ} \ar[d] \\ & Q_{\Lc_1} \ar[r]^-{{j'}} & Gr(hl_1,l_2)}
\end{equation}
in which the two vertical maps are $GL(V_2)$-bundles and the horizontal maps are embeddings, with ${j'}$ being the closed embedding described in Section~\textbf{6.3}.

\vspace{.1in}

Using the trace pairing, we may identify the cotangent space $T^*E=E \times E^*$ with the representation space of the double $\overline{Kr}$ of $Kr$ (that is, the quiver with vertex set $\{1,2\}$, $h$ arrows from $1$ to $2$ and $h$ arrows from $2$ to $1$) of dimension $(l_1,l_2)$~:
$$\xymatrix{ \underset{\bullet}{1} \ar@<1ex>[r]^-{h} & \underset{\bullet}{2} \ar@<1ex>[l]^-{h}}$$
so that
$$T^*E \simeq \Hom(V_1 \otimes \mathbb{V}, V_2) \times \Hom(V_2 \otimes \mathbb{V}^*, V_1).$$
Fixing dual bases $\{v_1, \ldots, v_h\}$ and $\{v_1^*, \ldots, v_h^*\}$ of $\mathbb{V}, \mathbb{V}^*$, we may write an element of $T^*E$ as a pair $(\underline{x}, \underline{y})$ with $\underline{x}=(x_1, \ldots, x_h)$, $\underline{y}=(y_1, \ldots, y_h)$ and $x_i \in \Hom(V_1, V_2), y_i \in \Hom(V_2, V_1)$.
Using this identification, and identifying $\mathfrak{g}$ with $\mathfrak{g}^*$ via the usual trace pairing, the moment map for the action of $G$ on $T^*E$ reads
$$\mu: T^*E \to \mathfrak{g}^*, \qquad \mu(\underline{x}, \underline{y})= \left( \sum_{i=1}^h y_ix_i, -\sum_{i=1}^h x_iy_i\right).$$

The Zariski closure $P=\overline{Q_{\Lc_1,\Lc_2}^{\circ,\circ}}$ of $Q_{\Lc_1,\Lc_2}^{\circ,\circ}$ in $E$ is a (possibly singular) affine variety. Of course, since $Q^{\circ,\circ}_{\Lc_1,\Lc_2}$ is dense in $Q^{\circ}_{\Lc_1,\Lc_2}$, we have $P=\overline{Q_{\Lc_1,\Lc_2}^{\circ}}$.
We will denote by $j : Q_{\Lc_1,\Lc_2}^{\circ,\circ} \to P$ and $i : P \to E$ the open, resp. closed embeddings.
There is a canonical projection $\pi:P \times E^* \to T^*P$ whose fibers are affine spaces. Namely, over a point $\underline{x} \in P$, the map $\pi$ is the natural projection
$$E^*=T^*_x E \to T^*_xE / (T_xP)^{\perp}=T^*_xP.$$
The map $\pi$ restricts to an affine fibration $\pi^{\circ} :Q_{\Lc_1,\Lc_2}^{\circ,\circ} \times E^* \to T^*Q_{\Lc_1,\Lc_2}^{\circ,\circ}$. 
The moment maps on $T^*Q^{\circ,\circ}_{\Lc_1,\Lc_2}, T^*P$ and $T^*E$ fit in a commutative diagram
$$\xymatrix{ Q_{\Lc_1,\Lc_2}^{\circ,\circ} \times E^* \ar[r]^-{j\times Id} \ar[d]^-{\pi^{\circ}}& P\times E^* \ar[r]^-{i \times Id} \ar[d]^{\pi} & T^*E \ar[dd]^-{\mu} \\ T^*Q^{\circ,\circ}_{\Lc_1,\Lc_2}\ar[drr]^-{\mu}  \ar[r]^-{d^*j}& T^*P \ar[rd]^-{\mu}  &\\
& & \mathfrak{g}^* }$$
where $d^*j$ is the open embedding induced by $j$.

\vspace{.1in}

For the reader's convenience we make explicit the map $d^*j$.  We begin with the differential $dj_{\phi} : T_{\phi} Q^{\circ,\circ}_{\Lc_1,\Lc_2} \to T_{j(\phi)}P$ at a point $\phi: \bigoplus_i (\Lc_i \otimes V_i) \twoheadrightarrow \mathcal{F}$.
Recall that we have canonical identifications
$$T_{\phi}Q^{\circ,\circ}_{\Lc_1,\Lc_2}=\Hom(\Ker(\phi), \mathcal{F}),$$
$$T_{j(\phi)}P \subset T_{j(\phi)}E =E=\Hom(V_1 \otimes \mathbb{V},V_2)\simeq \Hom(\Hom(\Lc_2,\Lc_1 \otimes V_1),\Hom(\Lc_2,\mathcal{F})).$$
Consider the exact sequences
\begin{equation}\label{E:seq1}
\xymatrix{ 0 \ar[r] & \Ker(\phi_1) \ar[r] & \Ker(\phi) \ar[r]^-{\rho_2} & \Lc_2 \otimes V_2 \ar[r] & 0},
\end{equation}
\begin{equation}\label{E:seq2}
\xymatrix{ 0 \ar[r] & \Ker(\phi_2) \ar[r] & \Ker(\phi) \ar[r]^-{\rho_1} & \Lc_1 \otimes V_1 \ar[r] & 0}.
\end{equation}
The first exact sequence (\ref{E:seq1}) is split as $\Ext^1(\Lc_2, \Ker(\phi_1))=0$ by condition d).
It follows that
\begin{equation}\label{E:dimeq}
\begin{split}
\text{dim}(\Hom(\Lc_2, \Ker(\phi)))&=\text{dim}(V_2) + \text{dim}(\Hom(\Lc_2, \Ker(\phi_2)))\\
&= \langle \Lc_2, \alpha \rangle + \langle \Lc_2, \Lc_1 \otimes V_1 -\alpha\rangle\\
&=\langle \Lc_2, \Lc_1 \otimes V_1 \rangle\\
&= \text{dim}(\Hom(\Lc_2, \Lc_1 \otimes V_1)).
\end{split}
\end{equation}
On the other hand, the exact sequence (\ref{E:seq2}) gives rise to a sequence
$$\xymatrix{ 0 \ar[r]& \Hom(\Lc_2, \Ker(\phi_2)) \ar[r] & \Hom(\Lc_2, \Ker(\phi)) \ar[r]^-{\rho_{1*}} & \Hom(\Lc_2, \Lc_1 \otimes V_1)}$$ 
and since $\Hom(\Lc_2, \Ker(\phi_2))=0$ this yields by (\ref{E:dimeq}) a canonical isomorphism $\rho_{1*}~:\Hom(\Lc_2, \Ker(\phi)) \to \Hom(\Lc_2, \Lc_1 \otimes V_1)$. The map $dj_{\phi}$ is equal to the composition
$$\Hom(\Ker(\phi), \mathcal{F}) \stackrel{q}{\to} \Hom(\Hom(\Lc_2, \Ker(\phi)), \Hom(\Lc_2, \mathcal{F}) \stackrel{\rho_{1*}}{\to} \Hom(\Hom(\Lc_2, \Lc_1 \otimes V_1), \Hom(\Lc_2, \mathcal{F}))$$
with
$$q : \Hom(\Ker(\phi), \mathcal{F}) {\to} \Hom(\Hom(\Lc_2, \Ker(\phi)), \Hom(\Lc_2, \mathcal{F}), \qquad u \mapsto ( a \mapsto u \circ a).$$
Because $j : Q^{\circ,\circ}_{\Lc_1,\Lc_2} \to P$ is an open embedding, $dj_{\phi}~: T_{\phi} Q^{\circ,\circ}_{\Lc_1,\Lc_2} \to T_{j(\phi)}P$ is an isomorphism. The map
$d^*j_{\phi}$ is the transpose isomorphism $T^*_{\phi} Q^{\circ,\circ}_{\Lc_1,\Lc_2} \to E^* / (T_{j(\phi)}P)^{\perp}$.

\vspace{.2in}

\paragraph{\textbf{6.6.}} We next consider GIT quotients of the various above spaces, following the method in \cite{King}. We will consider the stability condition associated to the character 
$$\gamma : G \to k^*, (g_1,g_2) \mapsto det(g_1)^{l_2}det(g_2)^{-l_1}.$$
Let $(T^*E)^{ss}, (T^*P)^{ss}, (P \times E^*)^{ss}$ denote the open sets of $\gamma$-semistable points, and let $T^*E \quot G, (T^*P)\quot G$ and $(P \times E^*) \quot G$ denote the affine quotients. Because $T^*E, T^*P$ and $P \times E^* $ are all affine varieties there are proper maps
$$p: (T^*E)^{ss} \quot G \to (T^*E) \quot G, \qquad  p': (P \times E^*)^{ss} \quot G \to (P \times E^*) \quot G, \qquad p'': (T^*P)^{ss}\quot G \to (T^*P)\quot G.$$
We have
$$(T^*E)^{ss} \quot G= \text{Proj} \left( \bigoplus_{l \geq 0} k[T^*E]^{\gamma,l}\right), \qquad (T^*E) \quot G= \text{Spec} (k[T^*E]^G),$$
where 
$$k[T^*E]^{\gamma,l}=\{ f \in k[T^*E] \;|\; g \cdot f= \gamma(g)^l f\quad\forall\; g \in G\}$$
and there are similar descriptions in the cases of $P \times E^*$ and $T^*P$. Finally, the closed embedding $i \times Id:~P \times E^* \hookrightarrow T^*E$ and the surjective map
$\pi:P \times E^* \to T^*P$ give rise to maps
$$\xymatrix{(T^*P) \quot G & (P \times E^*) \quot G \ar[r]^-{\overline{i \times Id}} \ar[l]_-{\overline{\pi}} & (T^*E)\quot G}.$$
Note that $\overline{\pi}$ is surjective while $\overline{i \times Id}$ is a closed embedding.

\vspace{.1in}

Recall that a subrepresentation of a representation $(\underline{x},\underline{y}) \in T^*E$ is a pair of subspaces $(W_1 \subseteq V_1, W_2 \subseteq W_2)$ such that $x_i (W_1) \subseteq W_2, {y}_i(W_2) \subseteq W_1$ for all $i$. Similarly, we will call subrepresentation of some $(\underline{x},\underline{y}) \in T^*P$ a pair of subspaces $(W_1 \subseteq V_1, W_2 \subseteq W_2)$ such that ${x}_i (W_1) \subseteq W_2$ and $y_i \in \mathfrak{p}_W / (T_{\underline{x}}P)^\perp$, where
$\mathfrak{p}_W=\{u \in Hom(V_2,V_1)\;|\; u(W_2) \subseteq W_1\}$. 

\vspace{.1in}

\begin{lem}\label{L:semistable} The following hold~:
\begin{enumerate}
\item[i)] A point $(\underline{x},\underline{y}) \in T^*E$ is $\gamma$-semistable (resp. $\gamma$-stable) if and only if
for any subrepresentation $W=(W_1,W_2)$ of $(\underline{x},\underline{y})$ we have $l_1\dim(W_2)-l_2\dim(W_1) \geq 0$ (resp. $>0$),
\item[ii)] A point $(\underline{x},\underline{y}) \in T^*P$ is $\gamma$-semistable (resp. $\gamma$-stable) if and only if
for any subrepresentation $W=(W_1,W_2)$ of $(\underline{x},\underline{y})$ we have $l_1\dim(W_2)-l_2\dim(W_1) \geq 0$ (resp. $>0$).
\end{enumerate}
\end{lem}
\begin{proof} The first statement is well-known and follows from the Hilbert-Mumford numerical criterion (see \cite{King}). The second one can be proved along the same lines. Note that the Hilbert-Mumford criterion is stated in \cite{King} for algebraically closed fields, but it holds over an arbitrary field , see \cite{Kempf} (moreover, the notion of semistability of representations of quivers (or of coherent sheaves ) is stable under field extension).
\end{proof}

\vspace{.1in}

Set $(T^*Q^{\circ,\circ}_{\Lc_1,\Lc_2})^{ss}= T^*Q^{\circ,\circ}_{\Lc_1,\Lc_2} \cap (d^*j)^{-1}\left((T^* P)^{ss}\right)$.

\begin{lem}\label{L:Psemis} We have $(T^* P)^{ss} = d^*j((T^*Q^{\circ,\circ}_{\Lc_1,\Lc_2})^{ss})$, i.e.
$(T^*P)^{ss} \subset d^*j(T^*Q^{\circ,\circ}_{\Lc_1,\Lc_2})$.\end{lem}
\begin{proof} Let $\rho:  T^*P \to P$ be the natural projection. Observe that $(T^*P)^{ss} \subseteq \rho^{-1}(P \cap E^{\circ})$. Indeed, if $(\underline{x},\underline{y}) \in T^*P$ with $\text{Im}(\underline{x}) \subsetneq V_2$ then the subrepresentation $W=(V_1, \text{Im}(\underline{x}))$ violates the semistability condition of Lemma~\ref{L:semistable}~ii).
Similarly, if $\underline{x} \in E$ satisfies $\bigcap_i \Ker(x_i) \neq \{0\}$ then $(\underline{x}, \underline{y})$ is not semistable for any $\underline{y}$ since the subspace $W=(\bigcap_i \Ker(x_i),0)$ violates the semistability condition.
By diagram (\ref{E:651}), $P \cap E^{\circ}=Q^{\circ}_{\Lc_1,\Lc_2}$. Moreover, by construction if $P \cap E^{\circ} \ni \underline{x}=j(\phi : (\Lc_1 \otimes V_1) \oplus (\Lc_2 \otimes V_2) \twoheadrightarrow \mathcal{F})$ satisfies
$\bigcap_i \Ker(x_i)=\{0\}$ then the map $\phi_{1*} : V_1 \to \Hom(\Lc_1,\mathcal{F})$ is injective, hence bijective as
$\dim(V_1)=\dim (\Hom(\Lc_1, \mathcal{F}))$. This implies that $(T^*P)^{ss} \subset \rho^{-1}(Q^{\circ,\circ}_{\Lc_1,\Lc_2})$.
The lemma is proved.
\end{proof}

\vspace{.2in}

\paragraph{\textbf{6.7.}} We will now relate some appropriate fibers of the moment map $\mu: T^*Q^{\circ,\circ}_{\Lc_1,\Lc_2} \to \mathfrak{g}^*$ to indecomposable vector bundles. This explains why we considered the quot scheme construction with two line bundles instead of one. Recall that we have assumed $r$ and $d$ to be coprime. It easily follows that we may pick $n,m$ verifying the hypothesis a)--d) in such a way that $l_1$ and $l_2$ are also coprime. Consider the element
$\lambda \in \mathfrak{g}^*=\mathfrak{gl}(V_1)^* \times \mathfrak{gl}(V_2)^*$ defined by
$$\lambda(u_1,u_2)=l_2Tr(u_1)-l_1Tr(u_2).$$ 
From now on, \textit{we will assume that $p > l_1l_2$.}
By construction we have
\begin{enumerate}
\item[i)]$\lambda(Id,Id)=0$,
\item[ii)] $\lambda(e_1,e_2) \neq 0$ for any nontrivial pair of projectors $(e_1,e_2) \in \mathfrak{gl}(V_1) \times \mathfrak{gl}(V_2)$.
\end{enumerate}

\vspace{.1in}

\begin{lem}\label{L:indecomp} Let $\phi : (\mathcal{L}_1\otimes V_1)\oplus (\mathcal{L}_2\otimes V_2) \twoheadrightarrow \mathcal{F}$
be a $k$-point of $Q^{\circ,\circ}_{\Lc_1,\Lc_2}$. We have $k\lambda \subset \text{Im}(\mu_{\phi})$ if and only if $\mathcal{F}$ is indecomposable.
\end{lem}
\begin{proof} By (\ref{E:2}) we have $k \lambda \in \text{Im}(\mu_{\phi})$ if and only if $j(\lambda)=0$ if and only if $\lambda(f \circ \phi)=0$ for all $f \in \text{End}(\mathcal{F})$. 

\vspace{.1in}

Let us assume that $\mathcal{F}$ is decomposable and let us fix a nontrivial decomposition $\mathcal{F}=\mathcal{G} \oplus \mathcal{H}$. As $\mathcal{F} \in \mathcal{C}_{>0}$ we have $\mathcal{G},\mathcal{H} \in \mathcal{C}_{>0}$. In particular, $\Ext^1(\mathcal{L}_i,\mathcal{G})=\Ext^1(\mathcal{L}_i,\mathcal{H})=\{0\}$ and we have decompositions 
$$\Hom(\mathcal{L}_i,\mathcal{F}) =\Hom(\mathcal{L}_i,\mathcal{G}) \oplus \Hom(\mathcal{L}_i,\mathcal{H})$$
Let $f$ be the projector onto  $\mathcal{G}$ along $\mathcal{H}$. Thus $f \circ \phi$ is the projector
onto $\Hom(\mathcal{L}_1,\mathcal{G}) \oplus \Hom(\mathcal{L}_2,\mathcal{G})$ along $\Hom(\mathcal{L}_1,\mathcal{H}) \oplus \Hom(\mathcal{L}_2,\mathcal{H})$. By ii) above $\lambda (f \circ \phi) \neq 0$ hence $\lambda \not\in \text{Im}(\mu_{\phi})$.

\vspace{.1in}

Next let us assume that $\mathcal{F}$ is indecomposable (and thus also geometrically indecomposable as $(r,d)$ are coprime). By Fitting's lemma $\text{End}(\mathcal{F})$ is a local 
$k$-algebra with $\text{End}(\mathcal{F})/\text{rad}(\text{End}(\mathcal{F}))=k$, and therefore every endomorphism is of the form $f=c Id + n$ for some nilpotent $n$. But then
$f \circ \phi=c(Id,Id) + (n_1,n_2)$ for some nilpotent $n_1$, $n_2$. We deduce using i) that
$$\lambda(f \circ \phi)=c \lambda(Id,Id)+ \lambda(n_1,n_2)=0.$$
 It follows that $\lambda \in \text{Im}(\mu_{\phi})$.
\end{proof}

\vspace{.2in}

\paragraph{\textbf{6.8.}} Put $A=k \lambda \subset \mathfrak{g}^*$ and
$$\X=\mu^{-1}(A) \subset T^*P, \qquad \X_t=\mu^{-1}(\{t\lambda\}), \qquad \X'=\X \backslash \X_0.$$
The idea is now to consider a GIT quotient $\Y$ of $\X$, and view the family of smooth varieties $\Y \to A$
as a deformation of the moduli space of stable Higgs bundles of rank $r$ and degree $d$. 
Because the moment map $\mu: T^*P \to \mathfrak{g}^*$ is $G$-equivariant, we still have a map
$\overline{\mu}: (T^*P)^{ss}\quot G \to A$. We set
$$\Y=\overline{\mu}^{-1}(A), \qquad \Y_t=\overline{\mu}^{-1}(\{t\lambda\}), \qquad \Y'=\Y \backslash \Y_0.$$
By construction, $\Y=\X^{ss} \quot G$ where $ \X^{ss}=\X \cap (T^*P)^{ss}.$
Observe that by Lemma~\ref{L:Psemis} we have $\X^{ss} \subset T^*Q^{\circ,\circ}_{\Lc_1,\Lc_2}$.

\vspace{.1in}

\begin{lem}\label{L:smooth} The $k$-schemes $\Y, \Y'$ and $\Y_t$ for $t \in k$ are smooth. In addition, $\X' \subset \X^{ss}$.
\end{lem}
\begin{proof} Because $l_1,l_2$ are relatively prime, we have $l_2 \dim(W_1) -l_1\dim(W_2) \neq 0$ for any proper pair of subspaces $W_1 \subset V_1, W_2 \subset V_2$. This implies that the notions of semistability and stability coincide in $T^*P$. The action of $PG:=G/\mathbb{G}_m$ on $(T^*P)^{st}$ is free, hence $(T^*P)^{st}\quot G=\big((T^*P)^{st} \cap T^*Q^{\circ,\circ}_{\Lc_1,\Lc_2}\big) \quot G$ is smooth. The first statement will be proved once we show that the map $\overline{\mu} : (T^*P)^{st} \quot G \to \mathfrak{g}^*$ is submersive. This follows from \cite[Lemma~2.1.5.]{CBV} (note that the hypothesis that the field be algebraically closed is not used in the proof there).
We turn to the second statement. Let $u=(\underline{x}, \underline{y}) \in \X'$ and let us assume that $u$ is not semistable. Thus, by Lemma~\ref{L:semistable} there exists a subrepresentation $(W_1,W_2)$ of $u$ such that $l_2 \dim(W_1) - l_1 \dim(W_2) >0$. There exists a lift $u'=(\underline{x}, \underline{y'}) \in P \times E^*$ of $u$ for which $(W_1,W_2)$ is also a subrepresentation. Moreover, we have $\mu(u')=\mu(u)=t \lambda$ with $t \neq 0$. But then $0=Tr(\mu(u')_{|W_1 \oplus W_2})=t(l_2\dim(W_1)-l_1\dim(W_2))$ in contradiction with property ii) of $\lambda$ (see Section~\textbf{6.7}.).
\end{proof}

\vspace{.1in}

\begin{prop}\label{P:Higgs}
There is a canonical isomorphism of schemes $\Y_0 \simeq Higgs^{st}_{r,d}(X)$.
\end{prop}
\begin{proof}
Let us fix a pair $(\phi,\theta) \in \mu^{-1}(0) \subset T^*Q^{\circ,\circ}_{\Lc_1,\Lc_2}$ with
$$\phi : (\Lc_1 \otimes V_1) \oplus (\Lc_2 \otimes V_2) \twoheadrightarrow \mathcal{V}.$$
We will say that $(\phi, \theta)$ is $\mu$-stable if $(\mathcal{V},\theta)$ is a stable Higgs bundle (as in Section~\textbf{6.2.}). We will say that 
$(\phi,\theta)$ is $\gamma$-stable if $d^*j((\phi,\theta)) \in (T^*P)^{st}$ (i.e. is $\gamma$-semistable). Recall (see Section~\textbf{6.4.}) that $Higgs^{st}_{r,d} \subset [\mu^{-1}(0)/G]$. The proof of Proposition~\ref{P:Higgs} boils down to showing that $(\phi,\theta) \in \mu^{-1}(0)$ is $\mu$-stable if and only if it is $\gamma$-stable. 

\vspace{.1in}

Let us denote by $\mathcal{S}_X$ the (finite) set of subsheaves $\mathcal{W} \subset \mathcal{V}$ which are strongly generated by $\Lc_1$. We will also denote by $\mathcal{S}'_X$ the subset of $\mathcal{S}_X$ consisting of Higgs subsheaves. Likewise, let us denote by $\mathcal{S}_{Kr}$ and $\mathcal{S}'_{Kr}$ the (finite) sets of submodules
of $j(\phi)$ and $d^*j(\phi, \theta)$ respectively. There is a natural injective map
$$\psi~: \mathcal{S}_X \to \mathcal{S}_{Kr}, \qquad \mathcal{W} \mapsto (\Hom(\Lc_1,\mathcal{W}), \Hom(\Lc_2, \mathcal{W})) \subseteq (\Hom(\Lc_1,\mathcal{V}), \Hom(\Lc_2, \mathcal{V})) \simeq (V_1,V_2).$$ 

\vspace{.1in}

\begin{lem}\label{L:H1} We have $\mathcal{W} \in \mathcal{S}'_X$ if and only if $\psi(\mathcal{W}) \in \mathcal{S}'_{Kr}$, i.e. $\psi^{-1}(\mathcal{S}'_{Kr})=\mathcal{S}'_X$.
\end{lem}
\begin{proof}
By definition a Higgs subsheaf of $(\mathcal{V},\theta)$ is a subsheaf $\mathcal{W} \subset \mathcal{V}$ satisfying (\ref{E:higgsext}). For a subsheaf $\mathcal{W} \subset \mathcal{V}$ which is strongly generated by $\Lc_1$, we have a commutative diagram
$$\xymatrix{0 \ar[r] & \Ker(\phi) \ar[r] & \bigoplus_i \Lc_i \otimes V_i \ar[r]^-{\sim} & \bigoplus_i \Lc_i \otimes \Hom(\Lc_i, \mathcal{V}) \ar[r] & \mathcal{V} \ar[r] & 0\\
0 \ar[r] & \Ker(\phi_{\mathcal{W}}) \ar[u] \ar[r] &  \bigoplus_i \Lc_i \otimes V_i \ar[u] \ar[r]^-{\sim} & \bigoplus_i \Lc_i \otimes \Hom(\Lc_i, \mathcal{W}) \ar[r]\ar[u]  & \mathcal{W} \ar[r]\ar[u]  & 0}$$
where the upward arrows are canonical embeddings.
This gives rise to a commutative diagram
$$\xymatrix{ \Ext^1(\mathcal{V},\mathcal{V})^* \ar[d]^a \ar[r]^-{i} & \Hom(\Ker(\phi_\mathcal{V}), \mathcal{V})^* \ar[d]^-{a'}  & \Hom(\Hom(\Lc_2, \Ker(\phi_{\mathcal{V}})), \Hom(\Lc_2, \mathcal{V}))^* \ar^-{\sim}[r] \ar[d] \ar[l]_-{\pi}& \Hom(V_1, V_2\otimes \mathbb{V}^*)^* \ar[d]^{a''} \\
\Ext^1(\mathcal{V},\mathcal{W})^*  \ar[r]^-{i'} & \Hom(\Ker(\phi_\mathcal{V}), \mathcal{W})^*   & \Hom(\Hom(\Lc_2, \Ker(\phi_{\mathcal{V}})), \Hom(\Lc_2, \mathcal{W}))^* \ar^-{\sim}[r] \ar[l]_-{\pi'} & \Hom(V_1, W_2\otimes \mathbb{V}^*)^*  \\
\Ext^1(\mathcal{W},\mathcal{W})^* \ar[u]_-{b} \ar[r]^-{i''} & \Hom(\Ker(\phi_\mathcal{W}), \mathcal{W})^* \ar[u]_-{b'}  & \Hom(\Hom(\Lc_2, \Ker(\phi_{\mathcal{W}})), \Hom(\Lc_2, \mathcal{W}))^* \ar^-{\sim}[r] \ar[u] \ar[l]_-{\pi''}& \Hom(W_1, W_2\otimes \mathbb{V}^*)^* \ar[u]_-{b''} }.$$
The maps $i,i',i''$ are injective while the maps $\pi,\pi',\pi''$ are surjective. The rightmost horizontal isomorphisms
are induced by the isomorphisms $\Hom(\Lc_2,\Ker(\phi_{\mathcal{V}})) \simeq \Hom(\Lc_2, \Lc_1 \otimes V_1)$ and
$\Hom(\Lc_2,\mathcal{V}) \simeq V_2$ (see Section~\textbf{6.5.}) and the similar isomorphisms with $\mathcal{W}$
instead of $\mathcal{V}$. Observe that $\pi$ is identified with the projection 
$$\xymatrix{T^*_{j(\phi)}E \ar[r] & T^*_{j(\phi)}P \ar[r]^-{\sim}_-{d^*j_{|\phi}^{-1}} & T^*_{\phi}Q^{\circ,\circ}_{\Lc_1,\Lc_2}}.$$

  The subsheaf $\mathcal{W}$ is a Higgs subsheaf if and only if $a(\theta) \in b(\Ext^1(\mathcal{W}, \mathcal{W})^*)$. Now consider the morphism of exact sequences
$$\xymatrix{0 \ar[r] & \Ext^1(\mathcal{V},\mathcal{W})^* \ar[r]^-{i'} & \Hom(\Ker(\phi_{\mathcal{V}}),\mathcal{W})^* \ar[r]^-{s'} & \Hom(\Lc_{\mathcal{V}}, \mathcal{W})^*\\
0 \ar[r] & \Ext^1(\mathcal{W},\mathcal{W})^* \ar[r]^-{i''}\ar[u]_-{b} & \Hom(\Ker(\phi_{\mathcal{W}}),\mathcal{W})^* \ar[r]^-{s''}\ar[u]_-{b'} & \Hom(\Lc_{\mathcal{W}}, \mathcal{W})^* \ar[u]_-{c}}$$
in which we have set for simplicity $\Lc_{\mathcal{V}}=\bigoplus_i \Lc_i \otimes V_i$ and $\Lc_{\mathcal{W}}=\bigoplus_i \Lc_i \otimes W_i$. Note that the map $c$ is injective since $\Lc_{\mathcal{W}}$ is a direct summand of $\Lc_{\mathcal{V}}$.
It follows that $a(\theta) \in b(\Ext^1(\mathcal{W}, \mathcal{W})^*)$ if and only if $i'(a(\theta)) \in b'(\Hom(\Ker(\phi_{\mathcal{W}}),\mathcal{W})^*)$, and this holds if and only $i(\theta)$ can be lifted to an element
$\underline{y} \in \Hom(V_1, V_2 \otimes \mathbb{W})^*$ satisfying $a''(\underline{y}) \in b''(\Hom(W_1, W_2 \otimes \mathbb{V})^*)$. But this last condition is equivalent to the fact that $(W_1,W_2)$ is a subrepresentation of 
$d^*j(\phi,\theta)$. The Lemma is proved.
\end{proof}

\vspace{.1in}

\begin{lem}\label{L:H2} Let $\mathcal{W} \in \mathcal{S}_X$ and $(W_1,W_2)=\psi(\mathcal{W})$. Then $\mu(\mathcal{W}) > \mu(\a)$ if and only if $l_1 \dim(W_2) < l_2 \dim(W_1)$.
\end{lem}
\begin{proof}
This is a straightforward computation (see e.g. \cite[Lemma 3.2]{ACKsurvey}).
\end{proof}

\vspace{.1in}

Let $(\phi, \theta) \in \mu^{-1}(0)$ be $\mu$-unstable. Then by the defining property d) of $(\Lc_1,\Lc_2)$ there exists
a destabilizing subsheaf $\mathcal{W} \in \mathcal{S}'_X$ (see Section~\textbf{6.4}). Therefore $\psi(\mathcal{W})$ is a destabilizing subrepresentation of $d^*j(\phi,\theta)$. Conversely, assume that $d^*j(\phi,\theta)$ is $\gamma$-unstable.
Following the authors of \cite{ACK} we will call tight a subrepresentation $(W_1,W_2)$ of $d^*j(\phi,\theta)$ satisfying the following condition~: if $(W'_1,W'_2)$ is a subrepresentation of $d^*j(\phi,\theta)$ such that $W_1 \subseteq W'_1$
and $W_2 \supseteq W'_2$ then $W'_1=W_1$ and $W'_2=W_2$. Clearly, there exists a tight destabilizing subrepresentation $(W_1,W_2)$ of $d^*j(\phi,\theta)${ }\footnote{Such a representation may be thought of as a maximally destabilizing subrepresentation of $d^*j(\phi,\theta)$.}. Observe that $(W_1,W_2)$ is also tight as a submodule of the (non-doubled) Kronecker representation $j(\phi)$. Using \cite[Lemma~5.5.]{ACK} we conclude that the submodule $(W_1,W_2)$ is equal to $\psi(\mathcal{W})$ for some $\mathcal{W} \in \mathcal{S}_X$. By Lemmas~\ref{L:H1} and \ref{L:H2} it follows that $\mathcal{W}$ is a destabilizing Higgs subsheaf of $(\phi,\theta)$ and thus that $(\phi,\theta)$ is $\mu$-unstable. The Proposition is proved.
\end{proof}

\vspace{.2in}

\paragraph{\textbf{6.9.}} Let us consider the action of $\mathbb{G}_m$ on $T^*E$ given by
$$z \cdot (\underline{x}, \underline{y})=(\underline{x}, z\underline{y}).$$
This action preserves $P \times E^*$ and descends to an action of $\mathbb{G}_m$ on $T^*P$, which in turn preserves $T^*Q^{\circ,\circ}_{\Lc_1,\Lc_2}$. Since the map $\mu: T^*Q^{\circ,\circ}_{\Lc_1,\Lc_2} \to \mathfrak{g}^*$ is equivariant (for the standard weight one action of $\mathbb{G}_m$ on $\mathfrak{g}^*$), this action preserves $\X$ and thus induces a $\mathbb{G}_m$-action on $\Y$. Observe that the schemes $\Y_t$ with $t \neq 0$ are transformed into each other by the $\mathbb{G}_m$-action and in particular are all isomorphic. 

\vspace{.1in}

\begin{prop}\label{P:contract} The $\mathbb{G}_m$-action on $\Y$ is contracting, i.e. for any $y \in \Y$ the action map $\mathbb{G}_m\to \Y, z \mapsto z \cdot y$ extends to a map $\mathbb{A}^1 \to \Y$.
\end{prop}
\begin{proof} It is enough to prove that the $\mathbb{G}_m$-action on $(T^*P)^{ss}\quot G=(T^*Q^{\circ,\circ}_{\Lc_1,\Lc_2})^{ss}\quot G$ is contracting since $\Y$ is closed in $(T^*{Q^{\circ,\circ}_{\Lc_1,\Lc_2}})^{ss}\quot G$. Because the map $p'': (T^*P)^{ss}\quot G \to (T^*P)\quot G$ is proper, it is in turn enough to prove that the $\mathbb{G}_m$-action on $(T^*P)\quot G$ is contracting. It is clear that the $\mathbb{G}_m$-action on $(T^*E)\quot G$ is contracting. Since $(P\times E^*)\quot G$
is a $\mathbb{G}_m$-invariant closed subvariety of $ (T^*E)\quot G$, the $\mathbb{G}_m$-action on $(P\times E^*)\quot G$ is contracting as well. But there is a surjective $\mathbb{G}_m$-equivariant morphism $(P\times E^*)\quot G \to (T^*P)\quot G$ and hence the $\mathbb{G}_m$-action on $(T^*P)\quot G$ is also contracting. The Proposition is proved.
\end{proof}

\vspace{.1in}

Denote by $\cZ$ the scheme of $\mathbb{G}_m$-fixed points in  $\Y$, a smooth subscheme of $\Y_0$. An explicit description of $\cZ$ is given and studied in \cite{Heinloth} (the so-called moduli of chains on $X$). Let $\cZ=\bigsqcup_{i} \cZ_i$ denote the decomposition of $\cZ$ into connected components. The tangent space to $\Y$ at a point $z \in \cZ$ splits as a direct sum 
$$T_z\Y=T_z\Y^+ \oplus T_z\cZ \oplus T_z\Y^-$$
where $T_z\Y^+$, resp. $T_z\Y^-$  stands for the subspace over which the $\mathbb{G}_m$-action is of strictly positive (resp. strictly negative) weight. Let $n_i$ be the dimension of $T_z\Y^+$ for $z \in \cZ_i$. Replacing $\Y$ by $\Y_0$ one similarly defines integers $n'_i$. Observe that $n'_i=n_i-1$ as $\nu$ is $\mathbb{G}_m$-equivariant and $\mathbb{G}_m$ acts on $L\simeq
\mathbb{A}^1$ with weight one.

By Lemma~\ref{L:smooth} i) and Proposition~\ref{P:contract}, the Hesselink-Byaliniki-Birula decomposition for $\Y$ and $\Y_0$ provide locally closed partitions
\begin{equation}\label{E:1}
\Y=\bigsqcup_i \W_i, \qquad \Y_0=\bigsqcup_i \W'_i
\end{equation}
where $\W_i$ (resp. $\W'_i$) is an $\mathbb{A}^{n_i}$-fibration (resp. an $\mathbb{A}^{n'_i}$-fibration) over $\cZ_i$,
see \cite[Thm. 5.7]{Hesselink}.

The decompositions (\ref{E:1}) and the fibrations $\W_i \to \cZ_i$, $\W'_i \to \cZ_i$ are all defined over $k$. It follows on the one hand that
$$|\Y(k)|=|\Y_0(k)| + (q-1) |\Y_1(k)|$$
and on the other hand that
$$|\Y(k)|=\sum_i q^{n_i} |\cZ_i(k)|, \qquad |\Y_0(k)|=\sum_i q^{n_i-1} |\cZ_i(k)|$$
We deduce that $|\Y_0(k)|=|\Y_1(k)|$. 
By Lemma~\ref{L:indecomp}
$$|\Y_1(k)|=\sum_{\mathcal{F}\in I_{r,d}} q^{\dim(\Ext^1(\mathcal{F},\mathcal{F}))} |\{\phi \in Q \;|\; \phi: \mathcal{L}_1^{\oplus l_1} \oplus \mathcal{L}_2^{\oplus l_2} \twoheadrightarrow \mathcal{F}\}| / |PG(k)|$$
where $I_{r,d}$ stands for the set of indecomposable (and hence geometrically indecomposable) vector bundles of rank $r$ and degree $d$ over $X$. For such a bundle, we have
$$|\{\phi \in Q \;|\; \phi: \mathcal{L}_1^{\oplus l_1} \oplus \mathcal{L}^{\oplus l_2} \twoheadrightarrow \mathcal{F}\}| / |PG(k)|=(q-1) /|\text{Aut}(\mathcal{F})|=q / |\text{End}(\mathcal{F})|,$$
from which we deduce that
$$|\Y_1(k)|=\sum_{\mathcal{F}\in I_{r,d}} q^{\dim(\Ext^1(\mathcal{F},\mathcal{F}))-\dim(\End(\mathcal{F}))+1}=q^{1-\langle \alpha,\alpha \rangle}|I_{r,d}|=q^{1+(g-1)r^2} |I_{r,d}|$$
as wanted. This finishes the proof of Theorem~\ref{T:3}.

\vspace{.2in}

\paragraph{\textbf{6.10.}} In this section we provide the (standard) proofs of Corollaries~\ref{Cor:1} and \ref{Cor:2}.

\noindent
\begin{proof}[{Proof of Corollary~\ref{Cor:1}}] Let $C_i=\{c_{i,j}\;|\; j \in K_i\}$ be the collection of Frobenius eigenvalues in $H^i_c(Higgs_{r,d}(X \otimes \overline{\mathbb{F}_q}), \qlb)$, counted with multiplicity. It is known that $Higgs_{r,d}(X \otimes \overline{\mathbb{F}_q})$ is cohomologically pure (see e.g. \cite[Cor. 1.2.3]{HRV2} for the similar case of the mixed Hodge structure on the moduli space of stable Higgs bundles over a complex curve, or see Section 6.11 below). Therefore $|c_{i,j}|=q^{i/2}$ for all $j \in K_i$. By Theorems~\ref{T:1} and \ref{T:3} there exists a polynomial $B_{r,d} \in \mathbb{Q}[T_g]^{W_g}$ and a unitary polynomial $R(q) \in \mathbb{Z}[q]$ such that for any $l \geq 1$
\begin{equation}\label{E:cor21}
B_{r,d}(\sigma_1^l, \ldots, \sigma_{2g}^l)=\left( \sum_{i,j} (-1)^i c_{i,j}^l\right) R(q^l),
\end{equation}
where $\sigma_X=(\sigma_1, \ldots, \sigma_{2g})$. Multiplying $B_{r,d}$ by some positive integer $N$ if necessary and repeating each $c_{i,j}$ $N$ times accordingly, we may assume that $B_{r,d} \in \mathbb{Z}[T_g]^{W_g}$. Expanding the product $\left( \sum_{i,j} (-1)^i c_{i,j}^l\right) R(q^l)$ and gathering together terms with the same sign, we may write
(\ref{E:cor21}) as an equality
\begin{equation}\label{E:cor22}
\sum_{a \in A} u_a^l=\sum_{b \in B}v_b^l,
\end{equation}
where $u_a, v_b$ are either some monomials of the form $\sigma_1^{i_1} \cdots \sigma_{2g}^{i_{2g}}$ or of the form
$q^nc_{i,j}$ for some $i$ and $j \in K_i$. Because (\ref{E:cor22}) holds for all $l$, we deduce that $\{u_a\;|\; a \in A\}=\{v_b\;|\; b \in B\}$. We may decompose the sets $\{u_a\}, \{v_b\}$ according to the complex norm, yielding for each $n$ an equality
$$\{u_a\;|\; a \in A, |u_a|=n\}=\{v_b\;|\; b \in B, |v_b|=n\}.$$
Let $d$ be the degree of $R(q)$, so that $R(q)=q^d + l.o.t.$. Set $l=\text{max}\;\{l\;|\; K_l \neq \emptyset\}$. Depending on the parity of $l$, the monomials of the form $q^dc_{l,j}$ either all belong to $\{u_a\;|\; a \in A, |u_a| =q^{d + l/2}\}$ or all belong to $\{v_b\;|\; b \in B, |v_b| =q^{d + l/2}\}$. This implies that the $c_{l,j}, j \in K_l$ are all equal to monomials 
$\sigma_1^{i_1} \cdots \sigma_{2g}^{i_{2g}}$ with $\sum_k i_k=l$. Cancelling from (\ref{E:cor22}) all the terms arising in the products $c_{l,j}R(q)$ for $j \in K_l$ and arguing by induction we deduce that the same holds for the $c_{i,j}$ with $j \in K_i$ and $i$ arbitrary. This proves the first point of Corollary~\ref{Cor:1}. The second point is proved by a careful bookkeeping of the exact same argument. We leave the details to the reader.
\end{proof}

\vspace{.1in}

\noindent
\begin{proof}[Proof of Corollary~\ref{Cor:2}] Let $X_{\mathbb{Q}}$ be a smooth projective curve of genus $g$ defined over $\mathbb{Q}$ and let $X_R$ be a spreading out of $X_{\mathbb{Q}}$ defined over some ring $R=\mathbb{Z}[\frac{1}{N}]$. Consider the $R$-scheme $\pi: Higgs_{r,d}(X_R) \to \text{Spec}(R)$. The complex $R\pi_!(\qlb)$ is locally constant over an open subset $U \subseteq \text{Spec}(R)$. For any field $k$ and any point $j_k:~\text{Spec}(k) \to \text{Spec}(R)$, the proper base change theorem provides an isomorphism
 $j_k^*R\pi_!(\qlb) \simeq R\pi_{k,!}(\qlb)$ where $\pi_k : Higgs_{r,d}(X_R \otimes k) \to \text{Spec}(k)$. 
If $j_{\mathbb{F}_q} \in U$ then 
$j_{\mathbb{F}_q}^*R\pi_!(\qlb) \simeq j_{\mathbb{Q}}^*R\pi_{!}(\qlb) \simeq $, where $j_{\mathbb{Q}} : \text{Spec}(\mathbb{Q}) \to \text{Spec}(R)$ is the generic point. As $j_{\mathbb{F}_q} \in U$ for $q \gg 0$, this yields an equality
$$\sum_n \text{dim}(H^n_c(Higgs_{r,d}(X_R\otimes \mathbb{F}_q), \qlb))t^n = \sum_n \text{dim}(H^n_c(Higgs_{r,d}(X_R\otimes \mathbb{Q}), \qlb))t^n.$$
Finally, by the Artin-Grothendieck comparison theorem, 
$$  \sum_n \text{dim}(H^n_c(Higgs_{r,d}(X_R\otimes \mathbb{Q}), \qlb))t^n= \sum_n \text{dim}(H^{n,sing}_c(Higgs_{r,d}(X_R\otimes \mathbb{C}), \mathbb{C}))t^n.$$
We conclude using the fact that the all the complex varieties $Higgs_{r,d}(X)$ as $X$ runs through the set of Riemann surface of genus $g$ are diffeomorphic (and all diffeomorphic to the genus $g$ character variety for the group $GL(r)$, see \cite{HRVInvent}).
\end{proof}

\vspace{.2in}

\paragraph{\textbf{6.11.}} Finally, let us prove Corollaries~\ref{Cor:25} and \ref{Cor:26}.

 \vspace{.1in}
 
 \begin{proof}[Proof of Corollary~\ref{Cor:25}] Assume first that $k=\mathbb{F}_q$.  Consider the $\mathbb{G}_m$-action on $Higgs^{st}_{r,d}$ defined by $\rho(z) (\mathcal{V},\theta)=(\mathcal{V}, z\theta)$. Observe that the Hitchin map $\mu$ is naturally $\mathbb{G}_m$-equivariant for the weight one action of $\mathbb{G}_m$ on the Hitchin base. Since $\mu$ is proper, it follows (as in Section~6.9) that this action is contracting. Let $Z=(Higgs_{r,d}^{st})^{ \mathbb{G}_m}$ be the be the fixed point subvariety and $Z= \bigsqcup_i Z_i$ its decomposition into connected components. Each $Z_i$ is a smooth subvariety
of $Higgs_{r,d}^{st}$ which is included in $\Lambda^{st}_{r,d}$ and hence is projective. The tangent space
of $Higgs_{r,d}^{st}$ at a point $z_i$ decomposes according to the $\mathbb{G}_m$-character as
$$T_{z_i} Higgs_{r,d}^{st}= T^{>0}_{z_i} \oplus T_{z_i} Z_i \oplus T^{<0}_{z_i}.$$
We have Byalinicki-Birula-Hesselink decompositions (for $\rho$ and $\rho^{-1}$ respectively)
$$Higgs_{r,d}^{st}= \bigsqcup_i Y^+_i, \qquad \Lambda^{st}_{r,d}=\bigsqcup_i Y_i^-$$
where $Y_i^+$ is a locally trivial $\mathbb{A}^{n^+_i}$-fibration over $Z_i$ and $Y_i^-$ is a locally trivial $\mathbb{A}^{n^-_i}$-fibration over $Z_i$, where
$$n_i^+=dim\; T^{>0}_{z_i}, \qquad n_i^-=dim\; T^{<0}_{z_i}$$
(this is independent of the choice of $z_i$ ). Because $\Lambda^{st}_{r,d}$ is lagrangian and $Z_i$ is included in the smooth locus of $\Lambda^{st}_{r,d}$ we have $n_i^+= \frac{1}{2} dim\; Higgs_{r,d}^{st}= 1 + (g-1)r^2$.
The varieties $Z_i$ being smooth and projective they are pure, and hence so are the $Y_i^+, Y_i^-$ (for the compactly supported cohomology). This implies that $\Lambda^{st}_{r,d}$ and $Higgs^{st}_{r,d}$ are pure as well,
and that there is a (non canonical) isomorphism in the Grothendieck group of $Gal(\overline{k} / k)$-modules
\begin{equation}\label{E:cor25:1}
H^n_c(Higgs^{st}_{r,d}, \qlb) \simeq \bigoplus_i H_c^{n-1-(g-1)r^2} (Z_i,\qlb) \{1 + (g-1)r^2\}
\end{equation}
where $\{\;\}$ denotes a Tate twist.
Similarly, there is an isomorphism
\begin{equation}\label{E:cor25:2}
H^n_c(\Lambda^{st}_{r,d}, \qlb) \simeq \bigoplus_i H_c^{n-n_i^{<0}} (Z_i,\qlb) \{n_i^{<0}\}.
\end{equation}
By Poincarr\'e duality, 
\begin{equation}\label{E:cor25:3}
H^{2 dim\;Z_i-l}_c(Z_i, \qlb)^*\{dim\;Z_i\} \simeq H^{l}_c(Z_i, \qlb). 
\end{equation}
Observe that $dim\; Z_i=1+(g-1)r^2-n_i^{<0}$. Combining (\ref{E:cor25:1}), (\ref{E:cor25:2}) and (\ref{E:cor25:3})
and taking the trace with respect to the Frobenius element yields statement i). Statement ii) for $k=\mathbb{F}_q$
follows by considering the appropriate cohomological degrees, and for $k=\mathbb{C}$ by the same type of arguments as in the proof of Corollary~\ref{Cor:2}. \end{proof}

 \vspace{.1in}
 
 \begin{proof}[Proof of Corollary~\ref{Cor:26}] We need to specialize (\ref{E:mainT1}) to $\a_1=\cdots =\a_{2g}=0$. To this end we rewrite the terms entering (\ref{E:mainT1}) as follows~:
 $$q^{(g-1)\langle \lambda, \lambda\rangle} J_{\lambda}(z)=\prod_{s \in \lambda^{\diamond}} \frac{\prod_{i} (\a_i q^{1+l(s)} -z^{a(s)})}{(q^{1+l(s)}-z^{a(s)})(q^{l(s)}-z^{a(s)})} \cdot \prod_{s \in \lambda \backslash \lambda^{\diamond}} \frac{\prod_i (\a_i-1)}{q-1}$$
 where $\lambda^{\diamond}$ denotes the set of $s \in \lambda$ satisfying $a(s) >0$ or $l(s) >0$.
 This expression is regular at the point $\a_1= \cdots = \a_{2g}=0$ and evaluates to 
 \begin{equation*}
 \begin{split}
q^{(g-1)\langle \lambda, \lambda\rangle} J_{\lambda}(z)_{| \alpha_i=0}&=\prod_{s \in \lambda^{\diamond}} z^{2(g-1)a(s)} \cdot \prod_{\substack{s \in \lambda^{\diamond}\\ l(s)=0}} \frac{1}{1-z^{-a(s)}} \cdot (-1)^{| \lambda \backslash \lambda^{\diamond}|}\\
& =(-1)^{| \lambda \backslash \lambda^{\diamond}|}z^{(g-1)(\sum_i \lambda_i^2 -\sum \lambda_i)}\prod_{\substack{s \in \lambda^{\diamond}\\ l(s)=0}} \frac{1}{1-z^{-a(s)}}.
\end{split}
\end{equation*}
 Next, we have
 $$L(z_n, \ldots, z_1)=\sum_{\sigma \in \mathfrak{S}_n} \epsilon(\sigma) \prod_{\substack{k <l \\ \sigma^{-1}(k) > \sigma^{-1}(l)}} \left( \frac{z_l}{z_k}\right)^{-g} \frac{\prod_i (1-\a_i \frac{z_l}{z_k})\cdot (\frac{z_l}{z_k}-q)}{\prod_i ( \frac{z_l}{z_k}-\a_i)\cdot (1-q\frac{z_l}{z_k})} \cdot \frac{1}{\prod_{j<n} (1-q\frac{z_{\sigma(j+1)}}{z_{\sigma(j)}})\cdot (1-z_{\sigma(1)})}.$$
 We see that for $\sigma \neq Id$, the evaluation of the quantity
$$\text{Res}_{\lambda} \left[\prod_{\substack{k <l \\ \sigma^{-1}(k) > \sigma^{-1}(l)}} \left( \frac{z_l}{z_k}\right)^{-g} \frac{\prod_i (1-\a_i \frac{z_l}{z_k})\cdot (\frac{z_l}{z_k}-q)}{\prod_i ( \frac{z_l}{z_k}-\a_i)\cdot (1-q\frac{z_l}{z_k})} \cdot \frac{1}{\prod_{j<n} (1-q\frac{z_{\sigma(j+1)}}{z_{\sigma(j)}})\cdot (1-z_{\sigma(1)})} \prod_{\substack{j =1 \\ j \not\in \{r_{\leq i}\}}}^{n}\frac{dz_j}{z_j}\right]$$
at $z_{1 + r_{<i}}=z^iq^{-r_{<i}}$ for $i=1, \ldots, t$ is a rational function of $\a_1, \ldots, \a_{2g}$ with coefficients in $\mathbb{Q}(z)$ which is regular and vanishes at the point $\a_1= \cdots = \a_{2g}=0$. As a consequence, if we
write $\lambda=(1^{r_1}, 2^{r_2}, \ldots)$ and denote by $i_1 < i_2 < \cdots < i_s$ the integers satisfying $r_{i_j}\neq 0$ then
\begin{equation*}
H_{\lambda}(z)_{|\a_i=0}=\frac{1}{(1-z^{i_1})(1-z^{i_2-i_1}) \cdots (1-z^{i_s-i_{s-1}})}=(-1)^{s}\frac{z^{-i_s}}{(1-z^{-i_1})
(1-z^{i_1-i_2}) \cdots (1-z^{i_{s-1}-i_{s}})}.
\end{equation*}
Observing that $s=|\lambda \backslash \lambda^{\diamond}|$, we get
\begin{equation}\label{E:proofcor110}
q^{(g-1)\langle \lambda, \lambda\rangle}J_{\lambda}(z)H_{\lambda(z)}=z^{(g-1)(\langle \lambda', \lambda' \rangle-|\lambda'|)-l(\lambda')} K_{\lambda'}(z).
\end{equation}
Finally, observe that since $A_{g,r}(z)$ has at most simple poles at $r$th roots of unity, the same holds for $A_{g,r}(z)_{|\a_i=0}$, and hence the residue at $r$th roots of unity is unchanged upon rescaling by a factor of $z^{-r}$. This allows us to remove the term $z^{-\sum_i \lambda_i}=z^{-r}$ in (\ref{E:proofcor110}).
We are done.\end{proof}
 
\vspace{.2in}

\section{Extension to the parabolic case}

\vspace{.1in}

\paragraph{\textbf{7.1.}} There are analogous result for vector bundles with (quasi)-parabolic structure. Let $X$ be as before a smooth projective curve defined over a finite field $\mathbb{F}_q$. Fix an effective divisor $D=\sum_{i=1}^N p_i x_i$ where for simplicity we assume that the $x_i$ are $\mathbb{F}_q$-rational points of $X$. By definition a \textit{quasi-parabolic vector bundle} $(\mathcal{V}, F^\bullet)$ \textit{on} $(X,D)$ is a vector bundle $\mathcal{V}$ on $X$ equipped with a collection of filtrations
$$F^{(i)}_1 \subseteq F^{(i)}_2 \subseteq \cdots \subseteq F^{(i)}_{p_i}=\mathcal{V}_{|x_i}$$
for $i=1, \ldots, N$. The sequence $(\text{dim}(F^{(i)}_1), \text{dim}(F^{(i)}_2), \ldots, \text{dim}(F^{(i)}_{p_i}))$ is called the \textit{dimension type} of $(\mathcal{V}, F^{\bullet})$ at $x_i$. 

Given $r >0, d \in \mathbb{Z}$ and fixed dimension types $\mathbf{d}^{(i)}=d^{(i)}_1 \leq \cdots \leq d^{(i)}_{p_i}=r$
for $i=1, \ldots, N$ we let $\A_{r,d, \mathbf{d}^{(1)}, \ldots, \mathbf{d}^{(N)}}(X)$ stand for the number of geometrically
indecomposable quasi-parabolic bundles on $(X,D)$ of rank $r$, degree $d$ and dimension type $\mathbf{d}^{(i)}$ at $x_i$ for all $i$. Again the finiteness of such number is a consequence of the existence of Harder-Narasimhan filtrations.

\vspace{.1in}

\begin{theo}\label{T:4} For any fixed genus $g$, any positive integer $N \geq 0$, any collection of positive integers
$\mathbf{p}=(p_1, \ldots, p_N)$ and any tuple $\boldsymbol{\alpha}=(r,d, \mathbf{d}^{(1)}, \ldots, \mathbf{d}^{(N)})$
satisfying 
$$(r,d)\in \mathbb{N} \times \mathbb{Z},$$ 
$$\mathbf{d}^{(i)}=(d^{(i)}_1 \leq \cdots \leq d^{(i)}_{p_i}=r)\quad \forall i$$
there exists a unique polynomial $A_{g,\mathbf{p},\boldsymbol{\alpha}} \in \mathbb{Q}[T_g]^{W_g}$
such that for any smooth projective curve $X$ of genus $g$ defined over a finite field, for any divisor $D=\sum_i p_i x_i$ with $x_i \in X(\mathbb{F}_q)$ we have
$$\A_{\boldsymbol{\alpha}}(X)=A_{g, \mathbf{p},\boldsymbol{\alpha}}(\sigma_X).$$
\end{theo}

\vspace{.1in}

When $g=0$, the above theorem settles Conjecture~9.2, ii) in \cite{SDuke}.

\vspace{.1in}

\paragraph{\textbf{7.2.}} The proof of Theorem~\ref{T:4} is completely parallel to that of Theorem~\ref{T:1}, using the spherical Hall algebra of the category of $D$-parabolic coherent sheaves over $X$ in place of the spherical Hall algebra of $X$. Shuffle presentations for such Hall algebras are studied in \cite{Jyun-Ao}.  There is also an effective version of Theorem~\ref{T:4}, whose proof is again similar to that in the non-parabolic case. It is natural to expect that the results and methods of Section~6 extend to the parabolic setting as well. This would then provide an answer to a question raised by Deligne in the context of the counting of the number of irreducible $l$-adic local systems on a curve defined over a finite field (see \cite{DF}). These extensions to the parabolic setting will be the subject of a companion paper.

\vspace{.2in}

\section{Refinements and  conjectures}

\vspace{.1in}

\paragraph{\textbf{8.1.}} Let $\nu \in \mathbb{Q}$. Denote by $\A^{\geq \nu}_{r,d}(X)$, resp. $\A^{\leq \nu}_{r,d}(X)$  the number of absolutely indecomposable vector bundles over $X$ of class $(r,d)$ lying in $\mathcal{C}_{\geq \nu}$, resp. $\mathcal{C}_{\leq \nu}$. From the proof of Theorem~1 it follows that

\vspace{.1in}

\begin{cor}\label{Cor:81} Fix $g \geq 0$ and $\nu \in \mathbb{Q}$. For any $(r,d) \in (\mathbb{Z}^2)^+$ there exists polynomials
$A_{g,r,d}^{\geq \nu}, A_{g,r,d}^{\leq \nu} \in \mathbb{Q}[T_g]^{W_g}$ such that for any smooth projective curve $X$ of genus $g$ defined over a finite field we have
$$\A^{\geq \nu}_{r,d}(X)=A^{\geq \nu}_{g,r,d}(\sigma_X), \qquad \A^{\leq \nu}_{r,d}(X)=A^{\leq \nu}_{g,r,d}(\sigma_X).$$
\end{cor}

\vspace{.1in}

\noindent
\textbf{Remarks.}

\vspace{.03in}

\noindent
$\mathbf{i)}$ When $\mu((r,d))=\nu$ we have $Coh_{r,d}(X) \cap \mathcal{C}_{\geq \nu}=Coh_{r,d}(X) \cap \mathcal{C}_{\nu}$. The above result thus implies that there exists polynomials counting the number of geometrically indecomposable semistable sheaves of any given slope $\nu$. 

\vspace{.03in}

\noindent
$\mathbf{ii)}$ We have $\mathcal{C}_{\geq 0} \supset \mathcal{C}_{\geq 1} \supset \cdots$. Thus for a given $(r,d)$ with $r>0, d \geq 0$ there is (for each curve $X$) a decreasing sequence of positive integers 
$$\mathcal{A}_{r,d}^{\geq 0}(X) \geq \mathcal{A}_{r,d}^{\geq 1}(X) \geq \cdots \geq \mathcal{A}_{r,d}^{\geq \frac{d}{r}}(X).$$
Of course, taking $d \gg 0$ we have $\mathcal{A}_{r,d}^{\geq 0}(X) =\mathcal{A}_{r,d}(X)$. It seems likely that
$$A_{g,r,d}^{\geq i}-A_{g,r,d}^{\geq i+1} \in \mathbb{N}[-z_i]_i^{W_g}, \qquad \text{for}\; i \geq 0.$$
When $(r,d)$ are coprime it would be interesting to interpret the ensuing termwise decreasing sequence  
$$A_{g,r,d}(t, \ldots, t) \geq \cdots \geq A_{g,r,d}^{\geq \frac{d}{r}}(t, \ldots, t)$$
of single variable polynomials as corresponding to some natural filtration in the cohomology of moduli spaces of Higgs bundles over complex curves (or character varieties).

\vspace{.03in}

\noindent
$\mathbf{iii)}$ The above remarks i), ii) can be made also in the case of vector bundles equipped with quasi-parabolic structures (for any choice of slope function). 

\vspace{.2in}

\paragraph{\textbf{8.2.}} From Corollary~\ref{Cor:81}, it seems natural to make the following conjecture. For $(\alpha_1, \ldots, \alpha_t)$ a Harder-Narasimhan type, let us denote by $\A_{\alpha_1, \ldots, \alpha_t}(X)$ the number of
absolutely indecomposable vector bundles over $X$ which belong to $\mathcal{C}_{\alpha_1, \ldots, \alpha_t}$. 
 
\begin{conj} Fix $g \geq 0$. For any Harder-Narasimhan type $(\alpha_1, \ldots, \alpha_t)$ there exists a polynomial $A_{g,\alpha_1, \ldots, \alpha_t} \in \mathbb{Q}[T_g]^{W_g}$ such that for any smooth projective curve $X$ of genus $g$ defined over a finite field we have
$$\A_{\alpha_1, \ldots, \alpha_t}(X)=A_{g,\alpha_1, \ldots, \alpha_t}(\sigma_X).$$
\end{conj}

\vspace{.1in}

Again, one may formulate an entirely similar conjecture in the case of vector bundles equipped with quasi-parabolic structures (for any choice of slope function). One may also formulate exactly the same conjecture in the context of representations of quivers.

\vspace{.2in}

\paragraph{\textbf{8.3.}} In the context of quivers Kac conjectured (see \cite{Kac2}), and Hausel proved in general, that the constant term
$A_{\mathbf{d}}(0)$ of the Kac polynomial attached to a quiver $Q$ and a dimension vector $\mathbf{d}$ is equal to the multiplicity of the root $\sum_i d_i \alpha_i$ in the Kac-Moody Lie algebra $\mathfrak{g}_Q$ canonically associated to $Q$ (see \cite{HauselKac} for details). 

\vspace{.1in}

In the context of a smooth projective curve one is therefore led to seek an analog of the Kac-Moody Lie algebra $\mathfrak{g}_Q$. Motivated by Ringel's theorem relating Hall algebras and quantum groups, we suggest the following construction. Let $X$ be a smooth projective curve of genus $g$ defined over an algebraically closed field, and let $\mathcal{H}^{\chi}_{\nu}$ be the space of all $\mathbb{C}$-valued constructible functions on the moduli stack $\textbf{Coh}_{\nu}$. The space $\mathcal{H}^{\chi}:= \bigoplus_{\nu} \mathcal{H}^{\chi}_{\nu}$ has the structure of a cocommutative Hopf algebra (see e.g. \cite[Sec. 10.20]{Lusztig} or \cite{BTL} ) and is sometimes called the $\chi$-Hall algebra of $X$. Let $\mathcal{H}^{\chi, sph}$ stand for the sub Hopf algebra generated by
the constant functions on $\textbf{Coh}_{0,d}$ and $\textbf{Bun}_{1,l}$ for $d \geq 0$ and $l \in \mathbb{Z}$. This Hopf algebra may be thought of as an $\alpha_i =1$ limit of the spherical Hall algebra $\mathcal{H}^{sph}$ of a curve of genus $g$ defined over a finite field. We define the spherical Hall Lie algebra $\mathfrak{h}^{sph}_X$ of $X$ as the Lie algebra of primitive elements in $\mathcal{H}^{\chi,sph}$.  We conjecture that this Lie algebra is independent of the choice of $X$ and has finite dimensional $\mathbb{Z}^2$ graded components. The analog of Kac's conjecture may now be formulated as follows~:

\vspace{.1in}

\begin{conj}\label{Conj: Kac} For any $(r,d) \in (\mathbb{Z}^2)^+$ we have $A_{g,r,d}(0)=dim\; \mathfrak{h}^{sph}_{(r,d)}$.
\end{conj}

\vspace{.1in}

One may also formulate a version of Kac's conjecture diretcly in terms of the spherical Hall algebra $\mathcal{H}^{sph}$ of the curve $X$ itself (and its integral form ${}_{R}\mathcal{H}^{sph}$), which are natural analogs in the contect of curves of the quantum envelopping algebra $U_q(\mathfrak{n}_Q)$. These algebras are $(\mathbb{Z}^2)^+$-graded but with graded components of infinite dimension in general. In order to circumvent this difficulty, let us denote by $\mathcal{H}^{sph, \geq 0}_{\nu}$ the subspace of $\mathcal{H}^{sph}_\nu$ consisting of those functions on $\textbf{Coh}_{\nu}$ which are supported on the sub-orbifold $\textbf{Coh}_{\nu}^{\geq 0}$. It is easy to check that $\mathcal{H}^{sph,\geq 0}= \bigoplus_{\nu} \mathcal{H}^{sph,\geq 0}_{\nu}$ is a $(\mathbb{N}^2)$-graded algebra with finite dimensional graded components.

\vspace{.1in}

\begin{conj}\label{Conj:DT} The following equality holds in the ring of power series $\mathbb{N}[[z^{(0,1)}, z^{(1,0)}]]$~:
$$\sum_{\nu} dim( \mathcal{H}^{sph, \geq 0}_{\nu}) z^{\nu}=\text{Exp}\left( \sum_{\nu} A^{\geq 0}_{g,r,d}(0)z^{\nu}\right).$$
\end{conj}

These conjectures may be directly checked for $g=0,1$ using the results in \cite{Kapranov}, \cite{BS} respectively.

\vspace{.1in}

\noindent
\textit{Remark.} By the main result in \cite{SVMathAnn}, the spherical Hall algebra $\mathcal{H}^{sph}$ of $X$ is isomorphic
to the spherical part of the K-theoretic Hall algebra 
$$\mathbf{K}_g=\bigoplus_{r \geq 0} K^{GL_r \times T_g}(\mathcal{C}_{g,r})$$ of the commuting variety $\mathcal{C}_g=\bigsqcup_r \mathcal{C}_{g,r}$ where
$$\mathcal{C}_{g,r}=\left\{ (x_1,y_1 \ldots, x_g,y_g) \in \mathfrak{gl}_r(\mathbb{C})^{2g}\;|\; \sum_i [x_i, y_i]=0\right\}.$$

We don't know how to describe geometrically the subalgebra $\mathbf{K}^{\geq 0}_g$ of $\mathbf{K}_g$ corresponding to $\mathcal{H}^{sph,\geq 0}$. However, one may expect the existence of a degeneration from $\mathbf{K}^{\geq 0}_g$ to the \textit{cohomological} Hall algebra 
$$\mathbf{C}_g=\bigoplus_{r \geq 0} H^{\bullet}_{GL_r \times T_g}(\mathcal{C}_{g,r})$$
(see \cite{SVIHES} where such a degeneration is performed (algebraically) in the case of $g=1$). In particular,
one may formulate an analog of Conjecture~\ref{Conj:DT} in which $\mathcal{H}^{sph,\geq 0}$ is replaced by the
spherical part of $\mathbf{C}_g$ (this suggests a relation between $A_{g,r,d}(0)$ and the Donaldson-Thomas invariants of the $2g$-loop quiver, with preprojective relations).

\vspace{.2in}

\centerline{\textbf{Acknowledgements}}

\vspace{.05in}

I am grateful to M. Brion, P.-H. Chaudouard, G. Chenevier, T. Hausel, F. Orgogozo, M. Reineke, V. Toledano-Laredo and E. Vasserot for helpful discussions and correspondence. Special thanks are due to G. Laumon and E. Letellier for many fruitful exchanges and for sharing their insight with me during the preparation of this work.

\vspace{.2in}

\appendix

\section{Volume of moduli stacks of torsion sheaves}

\vspace{.1in}

\paragraph{}\textit{Proof of Theorem~\ref{T:volform}, ii)}. Let $\textbf{Coh}^{(x)}_{0,d}$ be the sub-orbifold of $\textbf{Coh}_{0,d}$ parametrizing torsion sheaves of degree $d$ supported at a closed point $x$ of $X$, and let $1_{0,d}^{(x)}$ stand for the characteristic function of $\textbf{Coh}^{(x)}_{0,d}$. Observe that $1_{0,d}^{(x)}=0$ unless $deg(x)\, |\, d$ and
$$vol(\textbf{Coh}^{(x)}_{0,ndeg(x)})=(1^{(x)}_{0,ndeg(x)} \;|\; 1^{(x)}_{0,ndeg(x)})=\frac{|\mathcal{N}_{n}(k_x)|}{|GL_{n}(k_x)|}=\frac{q^{n(n-1)deg(x)}}{(q^{ndeg(x)}-1) \cdots (q^{ndeg(x)}-q^{(n-1)deg(x)})}$$
where $k_x \simeq \mathbb{F}_{q^{deg(x)}}$ is the residue field at $x$. Moreover,
$$\sum_{d \geq 0} 1_{0,d}s^d=\prod_{x \in X} \left( \sum_{n \geq 0} 1^{(x)}_{(0,ndeg(x))}s^{ndeg(x)}\right)$$
and
$$\sum_{l \geq 0} (1_{0,l}\;|\; 1_{0,l}) s^l=\prod_{x \in X} \sum_{n \geq 0}(1^{(x)}_{0,ndeg(x)}\;|\; 1^{(x)}_{0,ndeg(x)})s^{ndeg(x)}.$$
Using Heine's formula, we obtain
$$\sum_{n \geq 0}(1^{(x)}_{0,ndeg(x)}\;|\; 1^{(x)}_{0,ndeg(x)})s^{ndeg(x)}=\text{exp} \left( \sum_{l \geq 1} \frac{(sq^{-1})^{deg(x)}}{l(1-q^{-ldeg(x)})}\right).$$
Using the relation
$$\sum_{d\,|\, l} \sum_{\substack{x\\ deg(x)=d}} d=|X(\mathbb{F}_{q^l})|,$$
we finally obtain
$$\sum_{d \geq 0} (1_{0,d}\;|\; 1_{0,d}) s^d=\text{exp}\left( \sum_{l \geq 1} \frac{|X(\mathbb{F}_{q^l})|}{l(q^l-1)}s^l\right)=\text{Exp}\left(\frac{|X(\mathbb{F}_q)|}{q-1}s\right)$$
as wanted.

\vspace{.2in}

\section{Density of Weil numbers of smooth projective curves}

\vspace{.1in}

\begin{proof}[Proof of Proposition~\ref{P:Zariski}] The set $\mathcal{W}$ is constructed as the collection of Weil numbers (in the $l$-adic cohomology) of smooth projective curves defined over finite fields, allowing both the curve and the finite field to vary (as long as the characteristic is different from $l$). Let us first fix a finite field $\mathbb{F}_q$ with $l$ not dividing $q$ and denote by
$\mathcal{W}_q$ the set of all (collections of) Weil numbers of smooth projective curves defined over $\mathbb{F}_q$. Let $M_g^{\circ}$ be a smooth open subset of the moduli space of smooth projective curves over $\mathbb{F}_q$ of genus $g$.

\vspace{.1in}

 Let $\rho:\mathfrak{X} \to M_g^{\circ}$ denote the universal curve and set $\mathcal{F}=R^1\rho_{!}(\overline{\mathbb{Q}_l})(-1)$, a pure lisse sheaf of weight zero whose stalk at a point $\text{Spec}(\mathbb{F}_{q^n}) \to M_g^{\circ}$ corresponding to a curve $X$ defined over $\mathbb{F}_{q^n}$ is equal to $H^1(X \otimes \overline{\mathbb{F}_q},\qlb)(-1)$. Let $\pi_1(M_g^{\circ})$, resp. $\pi_1^{geom}(M_g^{\circ})$ stand for the fundamental group (resp. geometric fundamental group) of $M^{\circ}_g$, and let $\rho : \pi_1(M_g^{\circ}) \to GL(2g,\qlb)$ stand for the representation associated to $\mathcal{F}$ (well-defined up to conjugation). By \cite[Thm 10.1.16, Thm. 10.2.2]{KS}, the Zariski closure of 
$\rho(\pi_1^{geom}(M^{\circ}_{g}))$ is equal to $Sp(2g,\qlb)$. Moreover, $\rho(\pi_1(M^\circ_{g})) \subset Sp(2g, \qlb)$.

\vspace{.1in}

To every point $x~:\text{Spec}(\mathbb{F}_{q^n}) \to M_g^{\circ}$ there corresponds a map $\pi_1(\text{Spec}(\mathbb{F}_{q^n})) \to \pi_1(M_g^{\circ})$, and hence a Frobenius element $\rho(Fr_{x,n}) \in Sp(2g,\qlb)$ (well-defined up to conjugation).
Let $\rho(Fr_{x,n})^{ss}$ stand for the semi-simple part of $\rho(Fr_{x,n})$. Using the embedding $\iota~: \qlb \to \mathbb{C}$, we may view $\rho(Fr_{x,n})^{ss}$ as a semisimple conjugacy class in $Sp(2g,\mathbb{C})$. Because $\mathcal{F}$ is pure of weight zero, the eigenvalues of $\rho(Fr_{x,n})^{ss}$ are all unitary, i.e. the conjugacy class of $\rho(Fr_{x,n})^{ss}$ intersects the maximal compact subgroup $K \subset Sp(2g,\mathbb{C})$ in a $K$-conjugacy class
which we denote by $C_{x,n}$. If $x : \text{Spec}(\mathbb{F}_{q^n}) \to M^\circ_g$ corresponds to a curve $X$ defined over $\mathbb{F}_{q^n}$ then $C_{x,n}$ is the conjugacy class whose eigenvalues are $(q^{-n/2}\sigma_1, \ldots, q^{-n/2}\sigma_{2g})$ where $(\sigma_1, \ldots, \sigma_{2g})$ are the Weil numbers of $X$. By Deligne's equidistribution theorem (see \cite[3.5.3]{DeligneWeil2} and \cite[Thm. 9.2.6]{KS}), the set of conjugacy classes $\mathcal{C}_{\leq n}:=\{C_{x,m}\;|\; m \leq n, x \in M^{\circ}_g(\mathbb{F}_{q^m})\} $ becomes equidistributed for the Haar measure as $n$ tends to infinity. 

\vspace{.1in}

The maximal torus $T$ of $K$ is equal to
$$T=\{(z_1, \ldots, z_{2g}) \in (\mathbb{C}^*)^{2g}\;|\; |z_i|=1, z_{2i-1}z_{2i}=1 \; \forall\; i\} \simeq (S^1)^{g}.$$
Set ${\mathcal{W}'_q}=\bigcup_{n \geq 1} \mathcal{W}'_{q,n}$ where
$$\mathcal{W}'_{q,n}=\{q^{-n/2}\sigma_X=(q^{-n/2}\sigma_1, \ldots, q^{-n/2}\sigma_{2g})\;|\;X \in M_g^{\circ}(\mathbb{F}_{q^n})\}.$$
Deligne's equidistribution theorem implies that $\mathcal{W}'_{q,n}$ is equidistributed in $T/{W_g}$ as $n$ tends to 
infinity. In particular $\mathcal{W}'_q$ is dense in $T$ (for the analytic topology). We claim\footnote{We thank Ga\"etan Chenevier for providing us the argument.}, that this implies that $\mathcal{W}_q=\bigcup_{n \geq 1} \mathcal{W}_{q,n}$ is Zariski dense in $T_g/W_g$. Indeed let $f \in \mathbb{C}[T_g]^{W_g}$ be a polynomial function vanishing on $\mathcal{W}_qW_g$. Consider the (real) algebraic map $r: T \times \mathbb{R}^* \to T_g, ((z_1, \ldots, z_{2g}), t) \mapsto (tz_1, \ldots, tz_{2g})$. The image of $r$ contains $\mathcal{W}_qW_g$ and is Zariski dense in $T_g$. Assume that $f \neq 0$ so that $r^*f \neq 0$ and let us write $r^*f(\underline{z},t)=\sum_i h_i(\underline{z})t^i$.
Rescaling by a power of $t$ if necessary, we may assume that $h_0 \neq 0$ and $h_i=0$ for $i >0$. Let $\underline{z} \in T$ such that $h_0(\underline{z}) \neq 0$. Because each $\mathcal{W}'_{q,n}$ is finite and $\mathcal{W}'_q$ is dense in $T$ there exists a sequence $(\omega_i, n_i)_i$ with $\omega_i \in \mathcal{W}'_{q,n_i}$ and $n_i \mapsto \infty$ such that $\omega_i \mapsto \underline{z}$. The functions $h_i, i <0$ being bounded on the compact set $T$, it follows that $r^*f(\omega_i, q^{{n_i}/2}) \mapsto h_0(\underline{z}) \neq 0$, in contradiction with our hypothesis on $f$. This proves that $\mathcal{W}_qW_g$ is dense in $T_g$ and thus that $\mathcal{W}_q$ (and a fortiori $\mathcal{W}$) is dense in $T_g/W_g$. We are done. 
\end{proof}

\vspace{.2in}

\section{Proof of Conjecture~\ref{C:conj1} when $r$ is prime}

\vspace{.1in}

This is a straightforward computation. By the proof of Theorem~\ref{T:1}, $A_{g,r}(z)$ may have
poles only at $r$th roots of unity, and these poles are of order at most one. If $r$ is assumed to be prime then all the nontrivial $r$th roots of unity are primitive, and hence only occur as poles of terms in $A_{g,r}(z)$ containing a factor $(1-z^r)^{-1}$. Upon inspection, on easily sees that this factor arises (as a coefficient of $T^r$) on the r.h.s. of (\ref{E:mainT1}) in only two terms, namely
 $$q^{\langle (r), (r) \rangle} J_{(r)}(z) H_{(r)}(z)=\frac{\prod_i (\a_i-1)(\a_i-z) \cdots (\a_i-z^{r-1})}{(q-1) (q-z) \cdots (q-z^{r-1})} \cdot \frac{1}{(1-z) (1-z^2) \cdots (1-z^r)},$$
 and
 $$\frac{\mu(r)}{r} \psi_r \left( q^{\langle (1), (1) \rangle} J_{(1)}(z)H_{(1)}(z) \right)=-\frac{1}{r}\cdot \frac{\prod_i (\a_i^r-1)}{(q^r-1)(1-z^r)}.$$
 The result follows by a simple residue computation.

\vspace{.3in}

\noindent
D\'epartement de Math\'ematiques,
B\^at. 425, 
Universit\'e de Paris-Sud Orsay,\\
Facult\'e des sciences d'Orsay,\;
F-91405 Orsay Cedex \\
email~: \texttt{olivier.schiffmann@math.u-psud.fr}

\end{document}